\documentclass[10pt,a4paper,oneside,reqno]{amsart}
\usepackage{setspace}
\usepackage[totalwidth=14.6cm,totalheight=23cm]{geometry}
\usepackage[T1]{fontenc}
\usepackage[autostyle]{csquotes}
   \MakeAutoQuote{‘}{’}
   \MakeOuterQuote{"}
\usepackage{graphicx}
\usepackage{caption}
\usepackage{epsfig}
\usepackage{amsmath,amssymb,amscd,amsthm}
\usepackage{subfigure}
\usepackage{latexsym}
\usepackage{graphics}
\usepackage{colortbl} 
\usepackage{color}
\usepackage{ae}
\usepackage{enumitem}
\usepackage[all]{xy}
\usepackage{scalerel}
\usepackage{tikz}
\usepackage{cancel}
\usepackage{bbm,amsbsy}
\usepackage{mathrsfs}  
\usepackage[colorlinks=true,pagebackref=true]{hyperref}
\usepackage[normalem]{ulem}
\usepackage{nicefrac}
\usepackage{xfrac}
\usepackage{faktor}
\usepackage{tikz-cd} 
\usepackage{systeme}
\usepackage{nomencl}
\usepackage[normalem]{ulem}
\makenomenclature

\makeatletter
\newtheorem*{rep@theorem}{\rep@title}
\newcommand{\newreptheorem}[2]{%
\newenvironment{rep#1}[1]{%
 \def\rep@title{#2 \ref{##1}}%
 \begin{rep@theorem}}%
 {\end{rep@theorem}}}
\makeatother

\makeatletter
\newtheorem*{rep@cor}{\rep@title}
\newcommand{\newrepcor}[2]{%
\newenvironment{rep#1}[1]{%
 \def\rep@title{#2 \ref{##1}}%
 \begin{rep@cor}}%
 {\end{rep@cor}}}
\makeatother

\makeatletter
\newtheorem*{rep@prop}{\rep@title}
\newcommand{\newrepprop}[2]{%
\newenvironment{rep#1}[1]{%
 \def\rep@title{#2 \ref{##1}}%
 \begin{rep@prop}}%
 {\end{rep@prop}}}
\makeatother

\newtheorem{cor}{Corollary}[section]

\newtheorem{theorem}[cor]{Theorem}

\newtheorem{prop}[cor]{Proposition}

 \allowdisplaybreaks
\newrepcor{cor}{Corollary}
\newreptheorem{theorem}{Theorem}
\newrepprop{prop}{Proposition}

\newtheorem{lemma}[cor]{Lemma}

\theoremstyle{definition}
\newtheorem{defi}[cor]{Definition}
\theoremstyle{remark}
\newtheorem{remark}[cor]{Remark}
\newtheorem*{remark*}{Remark}

\newtheorem*{notation*}{Notation}

\newlist{steps}{enumerate}{1}
\setlist[steps, 1]{itemsep=8pt,leftmargin=0cm,itemindent=.5cm,labelwidth=\itemindent,labelsep=0cm,align=left,label = \textbf{\emph{Step \arabic*}:\,}}

\makeatletter
\newcommand{\myitem}[1]{%
\item[#1]\protected@edef\@currentlabel{#1}%
}
\makeatother

\newcommand{\bb}{\mathrm{b}}
\newcommand{\minko}{\mathbb{R}^{1,2}}

\newcommand{\ads}{\mathbb{A}\mathrm{d}\mathbb{S}^3}
\newcommand{\HP}{\mathbb{HP}^3}
\newcommand{\isom}{\mathrm{Isom}_0}

\newcommand{\I}{\mathrm{I}}
\newcommand{\II}{\mathrm{I}\hspace{-0.04cm}\mathrm{I}}

\newcommand{\grad}{\mathrm{grad}^{\mathbb{H}^2}}
\newcommand{\hess}{\mathrm{Hess}^{\mathbb{H}^2}}
\newcommand{\J}{\mathrm{J}}

\newcommand{\B}{\mathrm{B}}

\newcommand{\inner}[1] {\langle #1 \rangle}
\newcommand{\inners}{\langle \cdot, \cdot \rangle}

\begin{document}\raggedbottom

\title[Harmonic Lagrangian vector field]{Mean surfaces in Half-Pipe space and infinitesimal Teichmüller theory}

\author[Farid Diaf]{Farid Diaf}
\address{Farid Diaf: Univ. Grenoble Alpes, CNRS, IF, 38000 Grenoble, France.} \email{farid.diaf@univ-grenoble-alpes.fr}
\thanks{}
\maketitle
\begin{abstract}
We study a correspondence between smooth spacelike surfaces in Half-Pipe space $\HP$ and divergence-free vector fields on the hyperbolic plane \(\mathbb{H}^2\). We show that a particular case involves harmonic Lagrangian vector fields on \(\mathbb{H}^2\), which are related to mean surfaces in \(\HP\). Consequently, we prove that the infinitesimal Douady-Earle extension is a harmonic Lagrangian vector field that corresponds to a mean surface in \(\HP\) with prescribed boundary data at infinity.

We establish both existence and, under certain assumptions, uniqueness results for harmonic Lagrangian extension of a vector field on the circle. Finally, we characterize the Zygmund and little Zygmund conditions and provide quantitative bounds in terms of the Half-Pipe width.

\end{abstract}

\tableofcontents

\section{Introduction}
The goal of this paper is threefold:
\begin{enumerate}
    \item Study a correspondence between smooth spacelike surfaces in three-dimensional \textit{Half-pipe} space $\HP$ and divergence-free vector fields on the hyperbolic plane $\mathbb{H}^2$. This can be seen as an infinitesimal version of a well-known correspondence between smooth spacelike surfaces in three-dimensional \textit{Anti-de Sitter} space and area-preserving diffeomorphisms of the hyperbolic plane $\mathbb{H}^2$.
    \item Study \textit{harmonic Lagrangian} vector fields on $\mathbb{H}^2$ by giving several characterizations of them. In particular, we show that under the above correspondence, they correspond to \textit{mean surfaces} in $\HP$. Furthermore, we will see that the so-called \textit{infinitesimal Douady-Earle extension} is a particular case of harmonic Lagrangian vector fields. Hence, we obtain an interpretation of the infinitesimal Douady-Earle extension in terms of three-dimensional geometry.
    \item Finally, we show that any continuous vector field on the circle can be extended to a harmonic Lagrangian vector field on the hyperbolic plane. Moreover, such an extension is unique if the associated mean surface in $\HP$ has bounded principal curvature. In this way, we relate the regularity of harmonic Lagrangian vector fields with Half-pipe geometry.
\end{enumerate}

\subsection{Motivation from AdS geometry and conformally natural extension}
Following the groundbreaking ideas of Mess \cite{Mess}, the relationship between Lorentzian space forms and two-dimensional hyperbolic geometry has become a powerful tool in Teichmüller theory. Several contributions have been made on this subject; see, for example, \cite{Note_on_paper_mess,flatspacetimes_bonsante,Barbot_flatspacetime,canorot,BS_flat_conical}.

Mess emphasized the significance of studying Anti-de Sitter geometry in dimension $3$, namely the Lorentzian geometry of constant curvature $-1$. Anti-de Sitter space $\ads$ can be identified with the Lie group \(\text{Isom}_0(\mathbb{H}^2)\) of orientation-preserving isometries of the hyperbolic plane \(\mathbb{H}^2\), endowed with its bi-invariant metric induced by its Killing form. The study of $\ads$ is often motivated by its similarities to three-dimensional hyperbolic geometry and its connections to the Teichmüller theory of hyperbolic surfaces.

A main idea of Mess is the \textit{Gauss map construction}, which associates to a spacelike surface \(S\) in Anti-de Sitter space, a map \(\Phi\) between domains of \(\mathbb{H}^2\). Mess then observed that the connected component of the boundary of the convex hull in Anti-de Sitter space provides an earthquake map of \(\mathbb{H}^2\), leading to a new proof of Thurston's Earthquake Theorem \cite{Thurston} (the construction works even though the convex hull boundary is not a smooth surface).

Since Mess's work, interest in three-dimensional Anti-de Sitter space has grown, and the Gauss map construction has been used to provide several interesting extensions of circle homeomorphisms to the hyperbolic plane; see \cite{Maximalsurface, Areapreserving, SEP19}. For instance, Bonsante and Schlenker used the Gauss map construction to prove that any quasisymmetric homeomorphism \(\phi:\mathbb{S}^1 \to \mathbb{S}^1\) of the circle is the extension of a unique \textit{minimal Lagrangian diffeomorphism} \(\Phi:\mathbb{H}^2 \to \mathbb{H}^2\). These are diffeomorphisms of \(\mathbb{H}^2\) for which the graph is a minimal Lagrangian surface in \(\mathbb{H}^2 \times \mathbb{H}^2\). The crucial observation is that minimal Lagrangian maps are precisely those associated, via the Gauss map construction, to \textit{maximal surfaces} in $\ads$ (i.e., smooth surfaces with zero mean curvature).

Minimal Lagrangian diffeomorphisms are a particular class of \textit{conformally natural} extensions. Specifically, if \(A\) and \(B\) are isometries of \(\mathbb{H}^2\) and \(\phi\) is a quasisymmetric homeomorphism with a minimal Lagrangian extension \(\Phi\), then \(A \circ \Phi \circ B^{-1}\) is the minimal Lagrangian diffeomorphism that extends \(A \circ \phi \circ B^{-1}\). Nevertheless, in general, minimal Lagrangian diffeomorphisms are not stable under composition. In fact, a general theorem by Epstein and Markovic \cite{Stopdreamingthm} states that it is not possible to extend, in a homomorphic fashion, each quasisymmetric homeomorphism of the circle to a quasiconformal homeomorphism of \(\mathbb{H}^2\). However, the infinitesimal situation is completely different. Indeed, denote by $\Gamma(\mathbb{S}^1)$ and $\Gamma(\mathbb{H}^2)$ the spaces of continuous vector fields on $\mathbb{S}^1$ and $\mathbb{H}^2$, respectively. We say that a linear map $L: \Gamma(\mathbb{S}^1) \to \Gamma(\mathbb{H}^2)$ is \textit{conformally natural} if 
$$L(A_* X) = A_* L(X),$$
for all vector fields $X$ on the circle and for all isometries $A$ of the hyperbolic plane. The infinitesimal Douady-Earle extension $L_0: \Gamma(\mathbb{S}^1) \to \Gamma(\mathbb{H}^2)$ is an example of such a linear map. According to a theorem in \cite{Uniqueness_of_operator_L}, this is the unique (up to a constant) continuous linear operator that is conformally natural.

\subsection{Spacelike surfaces in \(\HP\) and vector fields of \(\mathbb{H}^2\)}

The first goal of this paper is to give an infinitesimal version of the Anti-de Sitter Gauss map construction, now between smooth spacelike surfaces in the \textit{Half-Pipe space} \(\HP\) on the one hand, and vector fields on \(\mathbb{H}^2\) on the other hand. Such a construction has been investigated by the author in \cite{diaf2023Infearth} for convex hulls in \(\HP\) (which are not smooth surfaces), thus yielding infinitesimal earthquakes of \(\mathbb{H}^2\), similar to how convex hulls in Anti-de Sitter space lead to earthquake maps on \(\mathbb{H}^2\).

The Half-Pipe space, also known as the \textit{Co-Minkowski space}, is the space of all spacelike planes in Minkowski space \(\mathbb{R}^{2,1}\). Recall that Minkowski space is the flat model of Lorentzian geometry, which can be described as the three-dimensional vector space \(\mathbb{R}^3\) endowed with a bilinear form of signature \((-,+,+)\). The Half-Pipe space can be identified as the geometry of the infinite cylinder \(\mathbb{H}^2 \times \mathbb{R}\) with respect to projective transformations induced from isometries of \(\mathbb{R}^{2,1}\). Indeed, for each pair \((\eta,t) \in \mathbb{H}^2 \times \mathbb{R}\), one can associate a spacelike plane in \(\mathbb{R}^{2,1}\) for which the normal is given by \(\eta\), and the oriented distance from the origin through the normal direction is \(t\).

We now describe what we could call the \textit{Gauss map construction} in Half-Pipe space. A plane in the projective model of Half-Pipe space \(\mathbb{H}^2 \times \mathbb{R}\) is said to be \textit{spacelike} if it is not vertical. It turns out that there is a projective duality between spacelike planes in Half-Pipe space and points in Minkowski space. This correspondence can be viewed as the infinitesimal version of the projective duality between points and spacelike planes in \(\ads\). The key idea in the Gauss map construction in \(\HP\) is that one of the models of Minkowski space is the Lie algebra \(\mathfrak{isom}(\mathbb{H}^2)\) of the Lie group \(\text{Isom}_0(\mathbb{H}^2)\), where each element of \(\mathfrak{isom}(\mathbb{H}^2)\) corresponds to a Killing vector field on \(\mathbb{H}^2\). Consequently, we establish the following homeomorphism:

\begin{equation}\label{introduction_duality}
\mathcal{K}: \{\text{Spacelike planes in } \HP\} \cong \{\text{Killing vector fields in } \mathbb{H}^2\}.
\end{equation}

Based on the identification \eqref{introduction_duality}, we can associate to each properly embedded spacelike surface \(S \subset \HP\), which is the graph of some smooth function \(u: \mathbb{H}^2 \to \mathbb{R}\), a vector field \(V_S\) on \(\mathbb{H}^2\). Specifically, take \(p \in \mathbb{H}^2\) and consider \(\mathrm{P}_p\), the tangent plane at \((p,u(p))\) of \(S\). This plane is spacelike, and therefore, by duality \eqref{introduction_duality}, we define the vector field associated with \(S\) as: 
\begin{equation}\label{2}
    V_S(p):=\mathcal{K}(\mathrm{P}_p)(p).
\end{equation}

The first result of this paper characterizes the vector field \(V_S\) in terms of the geometry of the tangent bundle of \(\mathbb{H}^2\). It turns out that \(\mathrm{T}\mathbb{H}^2\) is endowed with a natural \textit{pseudo-Kähler} structure, namely a triple \((\mathbb{G}, \mathbb{J}, \Omega)\) such that \(\mathbb{G}\) is a pseudo-Riemannian metric of signature \((2,2)\), \(\mathbb{J}\) is an integrable almost complex structure, and \(\Omega = \mathbb{G}(\mathbb{J} \cdot, \cdot)\) is a symplectic form (a non-degenerate closed 2-form). We will come back to this in detail in Section \ref{The_geometry_of_the_tangent bundle}.
\begin{theorem}\label{intro_divergence}
Let \(V: \mathbb{H}^2 \to \mathrm{T}\mathbb{H}^2\) be a smooth vector field on \(\mathbb{H}^2\). The following are equivalent:
\begin{enumerate}
\item There exists a smooth function \(u: \mathbb{H}^2 \to \mathbb{R}\) such that \(V\) is the vector field associated with the surface \(S = \mathrm{gr}(u) \subset \HP\), that is \(V = V_S\).
\item \(V(\mathbb{H}^2)\) is a \textit{Lagrangian} surface in \(\mathrm{T}\mathbb{H}^2\) with respect to the symplectic form \(\Omega\). (See Theorem \ref{canonical_pseudo}).
\item \(V\) is a divergence-free vector field on \(\mathbb{H}^2\).
\end{enumerate}
\end{theorem}

Theorem \ref{intro_divergence} is a local result, meaning that one may replace $\mathbb{H}^2$ with any simply connected open set of $\mathbb{H}^2$ and the result still holds. It is worth noting that in the case of Anti-de Sitter geometry, an important feature of the Gauss map construction is the fact that the space of timelike geodesics of \(\ads\) is naturally identified with \(\mathbb{H}^2 \times \mathbb{H}^2\). In our case, the tangent bundle \(\mathrm{T}\mathbb{H}^2\) can be interpreted as the space of oriented timelike geodesics in Minkowski space \(\mathbb{R}^{2,1}\). Since the seminal paper of Hitchin \cite{Monopole_Hitchin}, who observed the existence of a natural complex structure on the space of oriented geodesics in Euclidean three-space, there has been a growing interest in the geometry of the space of geodesics of certain Riemannian and pseudo-Riemannian manifolds (see \cite{An_indi_Kaehler,Anciaux_space_of_geodesic,Salvai_Godoy_smooth_foliation,ML_in_tangent_bundle,Bon_Emam,Emam_Seppi1}).

\subsection{Harmonic Lagrangian extension}

The second goal of this paper is to use the Gauss map construction in Half-Pipe space to obtain interesting extensions of vector fields on the circle to the hyperbolic plane, similarly to how the Gauss map construction in Anti-de Sitter space provides extensions of circle homeomorphisms. Let us briefly explain how such a construction can be used to extend a vector field on the circle. Given that the tangent bundle of \(\mathbb{S}^1\) is trivial, each vector field \(X\) on the circle can be represented as a function \(\phi_X: \mathbb{S}^1 \to \mathbb{R}\), called the \textit{support function}, where \(X(z) = iz\phi_X(z)\) for every \(z \in \mathbb{S}^1\).

Identifying a vector field \(X\) with its support function \(\phi_X: \mathbb{S}^1 \to \mathbb{R}\), we can view the graph of \(\phi_X\) in \(\mathbb{S}^1 \times \mathbb{R} \cong \partial \mathbb{H}^2 \times \mathbb{R}\) as a curve on the boundary at infinity of \(\HP\). The key point is that certain surfaces in \(\HP\) that admit these curves as their "boundary at infinity" give rise to vector fields on the hyperbolic plane that extend \(X\).

In this paper, we are interested in \textit{harmonic Lagrangian} vector fields on \(\mathbb{H}^2\). These can be seen as the infinitesimal analogue of minimal Lagrangian maps of \(\mathbb{H}^2\). That is, if \(\Phi_t: \mathbb{H}^2 \to \mathbb{H}^2\) is a one-parameter family of minimal Lagrangian maps such that \(\Phi_t = \mathrm{Id}\), then \(V = \frac{d}{dt}\big\lvert_{t=0} \Phi_t\) is a harmonic Lagrangian vector field. To define this class of vector fields intrinsically, we use the geometry of the tangent bundle \(\mathrm{T}\mathbb{H}^2\), similarly to how minimal Lagrangian maps correspond to minimal Lagrangian surfaces in \(\mathbb{H}^2 \times \mathbb{H}^2\).

\begin{defi}\label{intro_def1}
We say that a vector field \(V: \mathbb{H}^2 \to \mathrm{T}\mathbb{H}^2\) is \textit{harmonic Lagrangian} if it satisfies the following conditions:
\begin{enumerate}
\item \(V: (\mathbb{H}^2, g^{\mathbb{H}^2}) \to (\mathrm{T}\mathbb{H}^2, \mathbb{G})\) is a harmonic map.
\item \(V(\mathbb{H}^2)\) is a Lagrangian surface in \(\mathrm{T}\mathbb{H}^2\) with respect to the symplectic structure \(\Omega\).
\end{enumerate}
\end{defi}
In Definition \ref{intro_def1}, \(g^{\mathbb{H}^2}\) denotes the hyperbolic metric on \(\mathbb{H}^2\) and \(\mathbb{G}\) and \(\Omega\) are defined in Theorem \ref{The_geometry_of_the_tangent bundle}. The harmonicity condition is defined as the critical points of an energy functional among compactly supported variations, and an equivalent analytic condition is given in Definition \ref{harmonic_def}. The next result gives characterizations of harmonic Lagrangian vector fields.

\begin{theorem}\label{intro_chara}
Let \(V: \mathbb{H}^2 \to \mathrm{T}\mathbb{H}^2\) be a smooth vector field on \(\mathbb{H}^2\). The following are equivalent:
\begin{enumerate}
\item There exists a smooth function \(u: \mathbb{H}^2 \to \mathbb{R}\) such that \(S\) is a mean surface in \(\HP\) and \(V\) is the vector field associated with the surface \(S\), that is \(V = V_S\).
\item \(V\) is harmonic Lagrangian.
\item The unique self-adjoint \((1, 1)\)-tensor \(\bb: \mathrm{T}\mathbb{H}^2 \to \mathrm{T}\mathbb{H}^2\) such that \(\mathcal{L}_V g^{\mathbb{H}^2} = g^{\mathbb{H}^2}(\bb \cdot, \cdot)\) satisfies the conditions:
\begin{equation}\label{intro_eq_traceless}
\mathrm{tr}(\bb) = 0 \quad \text{and} \quad \mathrm{d}^{\nabla} \bb = 0,
\end{equation}
where \(\nabla\) is the Levi-Civita connection of \(\mathbb{H}^2\).
\end{enumerate}
\end{theorem}
Mean surfaces in \(\HP\) are smooth spacelike surfaces of zero mean curvature. They can be thought as the infinitesimal analogue of maximal surfaces in Anti-de Sitter space \(\ads\). However, due to the lack of a natural pseudo-Riemannian metric in \(\HP\), mean surfaces are not local minima or maxima for the area functional, which is why they are called mean surfaces and not minimal or maximal surfaces. Observe that the third characterization in Theorem \ref{intro_chara} implies in particular that \(V\) is divergence-free (that is the third characterization in Theorem \ref{intro_divergence}); in fact, this is equivalent to the traceless condition of the tensor \(\bb\) in \eqref{intro_eq_traceless}.

As we mentioned, a harmonic Lagrangian vector field can be seen as the infinitesimal version of a minimal Lagrangian map of \(\mathbb{H}^2\). The graphs of these maps are by definition Lagrangian and minimal surfaces in \(\mathbb{H}^2 \times \mathbb{H}^2\). By analogy, one may wonder if for a harmonic Lagrangian vector field \(V\), the section \(V(\mathbb{H}^2)\) is a minimal surface in \((\mathrm{T}\mathbb{H}^2, \mathbb{G})\). It turns out that this is not possible according to a theorem in \cite[Proposition 2.2]{ML_in_tangent_bundle}, which shows that if \(V: \mathbb{H}^2 \to \mathrm{T}\mathbb{H}^2\) is a vector field (not necessarily harmonic) such that \(V(\mathbb{H}^2)\) is a Lagrangian surface in \(\mathrm{T}\mathbb{H}^2\), then \(V(\mathbb{H}^2)\) cannot be minimal in \(\mathrm{T}\mathbb{H}^2\). We can now state our third main result.
\begin{theorem}\label{intro_ex_HL}
Let \(X\) be a continuous vector field on \(\mathbb{S}^1\). There exists a harmonic Lagrangian vector field on \(\mathbb{H}^2\) which extends continuously to \(X\) on \(\mathbb{S}^1\).
\end{theorem}
Using mean surfaces in Half-Pipe space, we will give an explicit construction of the harmonic Lagrangian vector field extending \(X\). This works as follows: take a continuous vector field \(X(z) = iz\phi_X(z)\), for some continuous function \(\phi_X: \mathbb{S}^1 \to \mathbb{R}\). Then, based on a theorem from \cite{barbotfillastre}, there exists a unique mean surface \(S\) in \(\HP\) with "boundary at infinity" \(\partial S\) given by the graph of \(\phi_X\). We then define
\begin{equation}\label{intro_HL}
    \mathrm{HL}(X) := V_{S_X}.
\end{equation}
We will show in Proposition \ref{second_proof} that \(\mathrm{HL}(X)\) is a harmonic Lagrangian vector field which extends continuously to \(X\).

The next result, proved in Section \ref{subsec_inf_DEE}, shows that the infinitesimal Douady-Earle extension coincides with the vector field \(\mathrm{HL}(X)\).

\begin{prop}[Proposition \ref{HL_DE}]\label{intro_pro_DE}
    Let \(X\) be a continuous vector field on the circle and \(L_0: \Gamma(\mathbb{S}^1) \to \Gamma(\mathbb{H}^2)\) be the infinitesimal Douady-Earle extension. Then 
    \[
    \mathrm{HL}(X) = L_0(X).
    \]
\end{prop}
The precise statement of Proposition \ref{intro_pro_DE} is given in Proposition \ref{HL_DE} after recalling in Section \ref{subsec_inf_DEE} the explicit construction of the infinitesimal Douady-Earle extension.
\subsection{Zygmund and little Zygmund vector fields}

The correspondence between mean surfaces in \( \HP \) and the infinitesimal Douady-Earle extension (Proposition \ref{intro_pro_DE}) establishes a connection with the work of Fan and Hu \cite{Fan_Hu_Little}, where infinitesimal Douady-Earle extensions are studied in detail and used to characterize Zygmund and little Zygmund vector fields using the \(\overline{\partial}\)-operator. These classes of vector fields are related to the tangent space of the \textit{universal Teichmüller space} and the \textit{little Teichmüller space}, respectively. In this section, we state our main results concerning the characterization of these classes of vector fields using the width of the convex core in \( \HP \). Note that the characterization involving the width is only visible from the Half-Pipe perspective.

Roughly speaking, the \textit{width} of a vector field \( X \) measures the thickness of the convex core of \( \mathrm{gr}(\phi_X) \) in \( \HP \). Concretely, it is defined as follows: Denote by \( \partial_{-}\mathcal{C}(X) \) and \( \partial_+\mathcal{C}(X) \) the lower and upper boundary components of the convex hull of the graph of \( \phi_X \) in the cylinder \( \HP \). It turns out that each of these boundary components is the graph of a function defined on \( \mathbb{D}^2 \). We denote these functions by \( \phi_X^- \) and \( \phi_X^+ \): \( \mathbb{D}^2 \to \mathbb{R} \) such that
$$ \partial_{\pm}\mathcal{C}(X) = \mathrm{gr}(\phi_X^{\pm}), $$
see Figure \ref{width_picture}. Next, we define the function \( \omega_X:\mathbb{D}^2\to\mathbb{R} \), which measures the length along the degenerate fiber of the convex core of \( \phi_X \), as follows: for each \( \eta \in \mathbb{D}^2 \cong \mathbb{H}^2 \), the points \( (\eta,\phi_X^+(\eta)) \) and \( (\eta,\phi_X^-(\eta)) \) in \( \HP \) correspond to two parallel spacelike planes in Minkowski space. We then define \( \omega_X(\eta) \) as the timelike distance between these two parallel spacelike planes. The width of \( X \) is then defined by:
$$ \omega(X) := \sup_{\eta\in\mathbb{D}^2} \omega_X(\eta)\in[0,+\infty] $$
This is a meaningful quantity to work with since it is invariant under isometries of Half-Pipe space.

The next result concerns the uniqueness of the harmonic Lagrangian extension for vector fields on the circle having \textit{Zygmund} regularity and estimates relating to the width.
\begin{theorem}\label{intro_uniqu_HL}
Let $X$ be a continuous vector field on $\mathbb{S}^1$. Consider $\omega(X)$ the width of $X$ and $\mathrm{HL}(X)$ the harmonic Lagrangian vector field defined in \eqref{intro_HL}. The following are equivalent:
\begin{enumerate}
    \item \label{1_intro_uniqu} $X$ is a Zygmund vector field.
    \item \label{2_intro_uniqu} There exists a harmonic Lagrangian vector field $V$ on $\mathbb{H}^2$ which extends continuously to $X$ and such that $\lVert\overline{\partial} V \rVert_{\infty}$ is finite.
    \item \label{3_intro_uniqu} $\omega(X)<+\infty$.
    \item \label{4_intro_uniqu} $\lVert \overline{\partial}\mathrm{HL}(X)\rVert_{\infty}<+\infty$.
\end{enumerate}
Moreover,
\begin{enumerate}
    \item[(i)]\label{Moreover_i} The harmonic Lagrangian extension $V$ as in \eqref{2_intro_uniqu} is unique.
    \item[(ii)]\label{Moreover_ii} The following estimates hold:
\begin{equation}\label{Intro_left_estimate}
\frac{1}{6}\lVert \overline{\partial} \mathrm{HL}(X)\rVert_{\infty} \leq \omega(X) \leq 2\lVert \overline{\partial} \mathrm{HL}(X) \rVert_{\infty}.
\end{equation}
\end{enumerate}
\end{theorem}
The \(\overline{\partial}\)-operator is defined in Section \ref{Sec_dbar_operator} via the complex structure \(\mathbb{J}\) of \(\mathrm{T}\mathbb{H}^2\). In simpler terms, if we write \(V\) in the Poincaré disk model of \(\mathbb{H}^2\), then the norm \(\lVert\overline{\partial} V \rVert_{\infty}\) coincides with the sup-norm of the complex derivative with respect to \(\overline{z}\), see Corollary \ref{d_bar_isometry}. In Proposition \ref{prop_d_bar_less_width}, we actually obtain a stronger statement for the left estimate \eqref{Intro_left_estimate}. Indeed, we derive a pointwise estimate, which should be of independent interest. See Remark \ref{remark_weil_petersson}. The next theorem concerns little Zygmund vector fields.

\begin{theorem}\label{Intro_chara_little}
Let $X$ be a continuous vector field on $\mathbb{S}^1$. Consider $\omega(X)$ the width of $X$ and $\mathrm{HL}(X)$ the harmonic Lagrangian vector field defined in \eqref{intro_HL}. The following are equivalent:
\begin{enumerate}
    \item $X$ is little Zygmund.
    \item $\lVert \overline{\partial} \mathrm{HL}(X)_p\rVert$ tends to zero as $p \in \mathbb{H}^2$ tends to the boundary of $\mathbb{H}^2$.
    \item $\omega_X(z)$ tends to zero as $\lvert z \rvert \to 1$.
\end{enumerate}
\end{theorem}
\subsection{Outline of the proof}\label{outline_of_the_proof}

We will sketch some elements of the proof of Theorem \ref{intro_ex_HL}. In fact, we will give two independent proofs of this result. One proof is in the spirit of Half-Pipe geometry, which is essential for proving the uniqueness result stated in Theorem \ref{intro_uniqu_HL}. The other proof uses tools from Teichmüller Theory. In the end, we will relate the two approaches.

Let's start by explaining the proof from a Half-Pipe perspective. The first thing to observe is that if \( V \) is a harmonic Lagrangian vector field, then by Definition \ref{intro_def1}, \( V(\mathbb{H}^2) \) is a Lagrangian surface in \( \mathbb{H}^2 \). According to Theorem \ref{intro_divergence}, \( V \) corresponds to a surface \( S \subset \HP \), which is the graph of some smooth function \( u: \mathbb{H}^2 \to \mathbb{R} \). The remaining part is to translate the harmonicity condition in terms of the surface \( S \). Theorem \ref{intro_chara} implies that \( S \) should be a mean surface in \(\HP\). At the price of analytical technicality, the choice of the mean surface \( S \) with boundary at infinity given by the graph of \( \phi_X \) will force the harmonic vector field \(\mathrm{HL}(X)\) defined in \eqref{intro_HL} to extend to \( X \) on \(\mathbb{S}^1\) (see Proposition \ref{second_proof}). We will provide a far more detailed discussion of the main ideas used in this step of the proof in Section \ref{section_5.2_extension_mean_surface}.

The proof of the uniqueness result in Theorem \ref{intro_uniqu_HL} (i.e., (i) in the "moreover" part) follows from the following result, which establishes a relationship between the boundedness of \( \overline{\partial} \) and the asymptotic behavior of mean surfaces.

\begin{prop}[Propositions \ref{zygumund_imply_extension} and  \ref{second_proof}]\label{intro_pro_bounded_imple_extension}
    Let \( u: \mathbb{H}^2 \to \mathbb{R} \) be a smooth function and \( S \subset \HP \) be its graph so that \( V_S \) is a harmonic Lagrangian vector field of \(\mathbb{H}^2\). Assume that \(\lVert \overline{\partial} V_S \rVert_{\infty}\) is finite, then the boundary at infinity of \( S \) is the graph of a continuous function \(\phi: \mathbb{S}^1 \to \mathbb{R}\). Moreover $V_S=\mathrm{HL}(X)$ extends continuously to $X(z)=iz\phi(z)$ on $\mathbb{S}^1$.
\end{prop}

Since \( S \) has zero mean curvature, the principal curvatures of \( S \) are \(\lambda\) and \(-\lambda\). The idea of the proof of Proposition \ref{intro_pro_bounded_imple_extension} is to relate \(\lVert \overline{\partial} V_S \rVert\) to the sup norm of \(\lvert \lambda \rvert\). Indeed, we show in Corollary \ref{d_bar_and_shape_operator} (see also Remark \ref{principal_curvature}):
$$\lVert \overline{\partial} V_S \rVert_{\infty} = \sup_{p \in \mathbb{H}^2} \lvert \lambda(p) \rvert.$$
Since the construction that associates to a spacelike surface in \(\HP\), a vector field of \(\mathbb{H}^2\) is invariant under normal evolution, one can push the surface \( S \) by normal evolution to obtain a convex or concave surface in \(\HP\) (depending on the direction of the evolution); this can be done when the principal curvatures are bounded. Then we use classical convexity results to show that the new surface has a boundary in \(\HP\) which is the graph of a continuous function, thus concluding the proof since the surface \( S \) has the same boundary at infinity as the pushed surfaces. Then one may use the uniqueness of the mean surface with prescribed data at infinity to conclude the uniqueness of the harmonic Lagrangian vector field with bounded \(\overline{\partial}\). This kind of argument would also be useful to prove the right estimate \eqref{Intro_left_estimate} in Theorem \ref{intro_uniqu_HL}.

The second way to prove the existence of harmonic Lagrangian extension follows by showing that if \( X \) is a continuous vector field on the circle, then the infinitesimal Douady-Earle extension \( L_0(X) \) is a harmonic Lagrangian vector field (see Proposition \ref{L_0_harmonic_lagrangian}). This can be done using the next result on an appropriate conformally natural continuous linear map from \(\Gamma(\mathbb{S}^1)\) to \(\Gamma(\mathbb{H}^2)\).

\begin{theorem}\label{intro_earle_unique} \cite{Uniqueness_of_operator_L}
    Up to multiplication by a constant, there is exactly one conformally natural continuous linear map from \(\Gamma(\mathbb{S}^1)\) to \(\Gamma(\mathbb{H}^2)\).
\end{theorem}

Based on a theorem of Reich and Chen \cite{Extension_with_bounded_derivative}, who proved that \( L_0(X) \) is an extension of \( X \) to the hyperbolic plane, we obtain another proof of Theorem \ref{intro_ex_HL}. Let us briefly explain why the infinitesimal Douady-Earle extension coincides with the vector field corresponding to a mean surface in \(\HP\) with prescribed data at infinity (i.e., $\mathrm{HL}(X)$). Given a continuous vector field \( X \) on the circle and consider \( u_X: \mathbb{H}^2 \to \mathbb{R} \) the function for which the graph is a mean surface with boundary at infinity given by \(\phi_X\), then \( u_X \) satisfies a particular linear elliptic equation (see Proposition \ref{BF_platau_HP}). This implies the following observation: For any vector fields \( X \) and \( Y \) on the circle and for any \(\lambda \in \mathbb{R}\),
$$\mathrm{HL}(X + \lambda Y) = \mathrm{HL}(X) + \lambda \mathrm{HL}(Y).$$
Namely, \(\mathrm{HL}: \Gamma(\mathbb{S}^1) \to \Gamma(\mathbb{H}^2)\) is a linear operator. Using classical results of elliptic equations, we show that \(\mathrm{HL}\) defines a continuous linear operator. By showing the conformal naturality of the operator \(\mathrm{HL}\) and applying Earle's Theorem \ref{intro_earle_unique}, we establish Proposition \ref{intro_pro_DE}.\\

It is worth noting that the uniqueness result in Theorem \ref{intro_earle_unique} differs from the uniqueness results for the harmonic Lagrangian extension of a given continuous vector field on the circle. Therefore, our uniqueness result in Theorem \ref{intro_uniqu_HL} is not a direct consequence of Earle's Theorem \ref{intro_earle_unique}. Furthermore, we expect that the uniqueness result in Theorem \ref{intro_uniqu_HL} may not hold if we remove the condition \(\lVert \overline{\partial} V \rVert_{\infty} < +\infty\). This condition can be viewed as a quasiconformal condition in the non-infinitesimal setting. In the theory of quasiconformal maps, it is known that for any given quasisymmetric homeomorphism on \(\mathbb{S}^1\), there is a unique quasiconformal harmonic extension to \(\mathbb{H}^2\) (see \cite{LI_Tam1, Schoen_conjecture}). Such a uniqueness result does not hold for non-quasiconformal harmonic extensions. In \cite{LI_Tam1, Li_Tam2}, Li and Tam constructed families of harmonic maps on \(\mathbb{H}^2\) with the same boundary values. By analogy with this result, one might expect that a family of harmonic Lagrangians vector fields with the same boundary values can be found.

\subsection{Organization of the paper}
In Section \ref{section3}, we collect some preliminary results to work towards the proofs of our main results introduced above. Section \ref{sec4_correpondance} explains the correspondence between spacelike surfaces in \( \HP \) and vector fields on \( \mathbb{H}^2 \), where we prove Theorem \ref{intro_divergence}. In Section \ref{sec5_harmo}, we introduce harmonic Lagrangian vector fields and prove Theorem \ref{charac_HL}. Additionally, we provide details on the \(\overline{\partial}\)-operator. Section \ref{sec6_extension} is devoted to proving Theorem \ref{intro_ex_HL}. We begin by proving it through the infinitesimal Douady-Earle extension and then provide the proof from the Half-Pipe perspective. Finally, we prove Proposition \ref{intro_pro_DE}, which connects the infinitesimal Douady-Earle extension to the mean surface in \( \HP \). In Section \ref{sec_uniquness}, we focus on proving the uniqueness part of Theorem \ref{intro_uniqu_HL}. In the last Section \ref{sec8_reg}, we complete the proof of Theorem \ref{intro_uniqu_HL}. Finally, we conclude by proving Theorem \ref{Intro_chara_little} concerning little Zygmund vector fields.

\subsection{Acknowledgments}

I would like to thank my advisor, Andrea Seppi, for helpful discussions, for his support, patience, and careful readings of this article. I would also like to thank Francesco Bonsante and François Fillastre for related discussions.

The author is funded by the European Union (ERC, GENERATE, 101124349). Views and opinions expressed are however those of the author(s) only and do not necessarily reflect those of the European Union or the European Research Council Executive Agency. Neither the European Union nor the granting authority can be held responsible for them.

\section{Preliminaries}\label{section3}
\subsection{Minkowski geometry} 
The Minkowski space is the vector space $\mathbb{R}^3$ endowed with a non degenerate bilinear form of signature $(-,+,+)$, it can be defined as
$$\minko:=\left( \mathbb{R}^3,  \inner{(x_0,x_1,x_2),(y_0,y_1,y_2)}_{1,2}=-x_0y_0+x_1y_1+x_2y_2\right).$$

The group of isometries of $\minko$ that preserve both the orientation and time orientation is identified as:
$$\mathrm{Isom}_0(\minko)=\mathrm{O}_0(1,2)\ltimes\minko,$$
where $\mathrm{O}(1,2)$ is the group the linear transformations that preserve the Lorentzian form $\inners_{1,2}$, $\mathrm{O}_0(1,2)$ denotes the identity component of $\mathrm{O}(1,2)$ and $\minko$ acts by translation on itself. In Minkowski space, there are three types of planes $\mathrm{P}$: \textit{spacelike} when the restriction of the Lorentzian metric to $\mathrm{P}$ is positive definite, \textit{timelike} when the restriction is still Lorentzian or \textit{lightlike} when the restriction is a degenerate bilinear form.

We define the hyperbolic plane as the upper connected component of the two-sheeted hyperboloid,
namely:
\begin{equation}\label{hyperboloid}\mathbb{H}^2:=\{(x_0,x_1,x_2)\in \mathbb{R}^{1,2}\mid -x_0^2+x_1^2+x_2^2=-1, \ x_0>0\}.\end{equation} The restriction of the Lorentzian bilinear form of $\minko$ on $\mathbb{H}^2$ induces a complete Riemannian metric $g^{\mathbb{H}^2}$ of sectional curvature $-1$. The group $\isom(\mathbb{H}^2)$ of orientation preserving isometries of $\mathbb{H}^2$ is thus identified with $\mathrm{O}_0(1,2)$. Consider the \textit{radial projection} $\Pi$ defined on $\{(x_0,x_1,x_2)\in\mathbb{R}^3\mid x_0\neq 0\}$ by:
\begin{equation}\label{radial}
\Pi(x_0,x_1,x_2)=\left( \frac{x_1}{x_0},\frac{x_2}{x_0}   \right),\end{equation} 
The projection $\Pi$ identifies the hyperboloid $\mathbb{H}^2$ with the unit disk $\mathbb{D}^2$ which is the \textit{Klein projective model} of the hyperbolic plane. The boundary at infinity $\partial\mathbb{H}^2$ of the hyperbolic plane is then identified with the unit circle $\mathbb{S}^1$. We now recall the definition of the Minkowski cross product.
\begin{defi}
    Let $x$, $y\in \minko$. The \textit{Minkowski cross product} of $x$ and $y$ is the unique vector $x\boxtimes y$ in $\minko$ such that:
    $$\inner{x\boxtimes y,v}_{1,2}=\det(x,y,v),$$
    for all $v\in \minko.$
\end{defi}
Next, we highlight two important features of the Minkowski cross product. The first one is that we can write the almost complex structure on $\mathbb{H}^2$ as follows. For $v\in\mathrm{T}_p\mathbb{H}^2$ :
\begin{equation}\label{complex_structure_H2}
    \J_p(v)=p\boxtimes v
\end{equation}
Notice that $\J$ is compatible with $g^{\mathbb{H}^2}$. Namely, 
$$g^{\mathbb{H}^2}(\J\cdot,\J\cdot)=g^{\mathbb{H}^2}(\cdot,\cdot).$$
This implies that $\omega^{\mathbb{H}^2}=g^{\mathbb{H}^2}(\J\cdot,\cdot)$ is a differential 2-form which is moreover closed. As a result, the triple $(g^{\mathbb{H}^2}, \J, \omega^{\mathbb{H}^2})$ is a Kähler structure on $\mathbb{H}^2$. We will use this structure in Section \ref{The_geometry_of_the_tangent bundle} to describe the geometry of the tangent bundle of $\mathbb{H}^2$. 

The second property of the Minkowski cross product is that it gives rise to an isomorphism $\Lambda: \mathbb{R}^{1,2}\to \mathfrak{isom}(\mathbb{H}^2)$ between the Minkowski space $\minko$ and the Lie algebra of $\mathrm{Isom}(\mathbb{H}^2)$. Since $\mathrm{O}_0(1,2)$ is the isometry group of $\mathbb{H}^2$, $\mathfrak{isom}(\mathbb{H}^2)$ is the algebra of skew-symmetric matrices with respect to $\inner{\cdot,\cdot}_{1,2}$. More precisely, we have:
\begin{equation}\label{lievskilling}
    \Lambda(x)(y):=y\boxtimes x.\end{equation}
The isomorphism $\Lambda$ is equivariant with respect to the linear action $\mathrm{O}_0(1,2)$ on $\minko$ and the adjoint action of $\mathrm{O}_0(1,2)$ on $\mathfrak{isom}(\mathbb{H}^2)$, namely for all $x\in \mathbb{R}^{2,1}$, $\mathrm{A}\in\mathrm{O}_0(1,2)$, we have:
\begin{equation}\label{adequivariance}
\Lambda(\mathrm{A}\cdot x)=\mathrm{A}\Lambda(x)\mathrm{A}^{-1}.\end{equation}

Recall that the Lie algebra $\mathfrak{isom}(\mathbb{H}^2)$ can be seen as the space of all \textit{Killing vector fields} on $\mathbb{H}^2$, where a Killing field $X$ is by definition a vector field whose flow is a one-parameter group of isometries of $\mathbb{H}^2$. Indeed each $X\in \mathfrak{isom}(\mathbb{H}^2)$ defines a Killing field on $\mathbb{H}^2$ given by:
$$\overline{X}(p)=\frac{d}{dt}\bigg\lvert_{t=0} (e^{tX}\cdot p)$$ and any Killing field on $\mathbb{H}^2$ is of this form for a unique $X \in \mathfrak{isom}(\mathbb{H}^2)$. 
\subsection{Half-pipe geometry as dual of Minkowski geometry}
In this section, we will present the so called \textit{Half-pipe} space. Following 
\cite{danciger_thesis}, it is defined as
$$\HP :=\{[x_0,x_1,x_2,x_3]\in\mathbb{RP}^3, \ -x_0^2+x_1^2+x_2^2<0\}.$$
The boundary at infinity $\partial\HP$ of $\HP$ is given by: 
$$\partial\HP=\{[x]\in \mathbb{RP}^3, \ -x_0^2+x_1^2+x_2^2=0\}.$$ The Half-pipe space has a natural identification with the dual of Minkowski space, namely the space of spacelike planes of Minkowski space. More precisely, we have the homeomorphism
\begin{equation}\label{hpminko}
    \mathcal{D}:\ \HP\cong \{\mathrm{Spacelike}\ \mathrm{planes}\ \mathrm{in}\ \minko\}
\end{equation}
which associates to each point $[x,t]$ in $\HP$, the spacelike plane of $\minko$ defined as:
 \begin{equation}\label{4}
    \mathrm{P}_{[x,t]}=\{  y \in \minko : \langle x,y \rangle_{1,2} = t         \}
\end{equation}
The homeomorphism $\mathcal{D}$ extends to a homeomorphism between $\partial\HP\setminus[0,0,0,1]$ and the space of \textit{lightlike} planes in Minkowski space $\minko$ using the same formula \eqref{4}. Another interesting model of $\HP$ derived from this duality is given by the diffeomorphism: $\HP\to\mathbb{H}^2\times\mathbb{R}$ defined by: 
\begin{equation}\label{L}
[x,t]\to\left(\frac{x}{\sqrt{-\langle x,x\rangle}_{1,2}},\mathrm{L}([x,t])\right),\end{equation}
where $\mathrm{L}([x,t])$ is the \textit{height function}, which is defined as the signed distance of the spacelike plane $\mathrm{P}_{[x,t]}$ to the origin along the future normal direction. It can be checked by elementary computation that: \begin{equation}\label{formule_L}
    \mathrm{L}([x,t])=\frac{t}{\sqrt{-\langle x,x\rangle}}_{1,2}.\end{equation}
We will call \textit{geodesics} (resp. \textit{planes}) of $\HP$ the intersection of lines (resp. planes) of $\mathbb{RP}^3$ with $\HP.$ We will also use the following terminology:
\begin{itemize}
\item[\textbullet] A geodesic in $\HP$ of the form $\{*\}\times \mathbb{R}$ is called a \textit{fiber}.
\item[\textbullet] A geodesic in $\HP$ which is not a fiber is called a \textit{spacelike} geodesic.
\item[\textbullet] A plane in $\HP$ is \textit{spacelike} if it does not contain a fiber.
\end{itemize}
From this, we can define a dual correspondence to the identification \eqref{hpminko} as follows:
\begin{equation}\label{D_star}
  \mathcal{D}^*: \ \minko\cong\{ \mathrm{Spacelike}\ \mathrm{planes}\ \mathrm{in}\ \HP  \}
\end{equation}
which associates to each vector $v$ in $\minko$, the spacelike plane in $\HP$ given by: 
$$\mathrm{P}_v:=\{[x,t]\in\HP \mid \inner{x,v}_{1,2}=t   \}.$$
Now, let us denote by $\isom(\HP)$ the group of transformations given by: 
$$\begin{bmatrix}
\mathrm{A} & 0  \\
v & 1 
\end{bmatrix},$$
where $\mathrm{A}\in \mathrm{O}_{0}(1,2)$ and $v\in \mathbb{R}^2$. Observe that the group $\isom(\HP)$ preserves the orientation of $\HP$ and sends oriented fibers to oriented fibers. The map $\mathcal{D}$ in \eqref{hpminko} induces an isomorphism between 
$\isom(\mathbb{R}^{1,2})$ and $\isom(\HP)$ given by (See \cite[Section 2.8]{riolo_seppi})

\begin{equation}\label{groupeduality}  
\begin{array}{ccccc}
\mathrm{Is}: &  & \isom(\mathbb{R}^{1,2}) & \to & \isom(\HP) \\
 & & (\mathrm{A},v) & \mapsto & \begin{bmatrix}
\mathrm{A} & 0  \\
^T v\mathcal{J}\mathrm{A} & 1 
\end{bmatrix}, \\
\end{array}
\end{equation}
where $\mathcal{J}=\mathrm{diag}(-1,1,1).$ 
Now, we move on to describing the $\textit{Klein model}$ of the Half-pipe space which will be useful in this paper. It is defined as the cylinder $\mathbb{D}^2\times\mathbb{R}$ obtained by projecting $\HP$ in the affine chart  $\{x_0\neq 0\}$: \begin{equation}\label{7}
    \begin{array}{ccccc}
&  & \HP & \to & \mathbb{D}^2\times\mathbb{R} \\
 & & [x_0,x_1,x_2,x_3] & \mapsto & (\frac{x_1}{x_0},\frac{x_2}{x_0},\frac{x_3}{x_0}). \\
\end{array}\end{equation}
By the above discussion, it is clear that in this model the correspondence $\eqref{D_star}$ between Minkowski space and spacelike planes of $\HP$ associates to each vector $v\in\minko$, the non-vertical plane
\begin{equation}\label{formule_dual_plan}
    \mathrm{P}_v:=\{(\eta,t)\in\mathbb{D}^2\times\mathbb{R} \mid \langle (1,\eta),v  \rangle_{1,2}=t  \}\end{equation} and this plane corresponds to the graph of an affine function over $\mathbb{D}^2$.
The height function $\mathrm{L}$ is expressed in the Klein model as follows 
    \begin{equation}\label{formule_L_Klein}
    \mathrm{L}(\eta,t)=\frac{t}{\sqrt{1-\lvert\eta\rvert^{2}}},\end{equation} where $\lvert \eta\vert$ is the standard euclidean norm of $\eta$. 
The next Lemma describes the action of the isometries of $\HP$ in the Klein model. 
\begin{lemma}\cite[Lemma 2.26]{barbotfillastre}\label{rotationHP}
    Let $(\eta, t)\in \mathbb{D}^2\times\mathbb{R}$, $\mathrm{A}\in \mathrm{O}_{0}(1,2)$ and $v\in \mathbb{R}^{1,2}$. Then the isometry of
Half-pipe space defined by $\mathrm{Is}(\mathrm{A},v)$ acts on the Klein model $\mathbb{D}^2\times\mathbb{R}$ as follows:
$$\mathrm{Is}(\mathrm{A},v)\cdot(\eta,t)=\left( \mathrm{A}\cdot \eta, \frac{t}{-\langle \mathrm{A}\begin{pmatrix}
1 \\
\eta
\end{pmatrix} ,(1,0,0)\rangle_{1,2}}+\langle (1,\mathrm{A}\cdot \eta), v\rangle_{1,2} \right),$$
where $\mathrm{A}\cdot \eta$ is the image of $\eta$ by the isometry
of $\mathbb{D}^2$ induced by $\mathrm{A}.$
\end{lemma}

We finish these preliminaries on Half-pipe geometry by recalling the notion of width. To do this, we need to recall some notions in convex analysis. For more detailed discussions, readers can refer to \cite{Rockconvex}.
Given a convex (resp. concave) function $u:\mathbb{D}^2\to \mathbb{R}$, the \textit{boundary value} of $u$ is the function on $\mathbb{S}^1$ whose value at $z\in\mathbb{S}^1$ is given by: \begin{equation}\label{boundaryvalue}
    u(z) = \lim_{s \to 0^+}u((1-s)z+sx),\ \mathrm{for} \ \mathrm{any}\ x\in\mathbb{D}^2.\end{equation}
The limit defined in \eqref{boundaryvalue} exists and does not depend on the choice of $x\in \mathbb{D}^2$ (see \cite[Section 4]{ConvexAnal}). Moreover, if $u:\mathbb{D}^2\to\mathbb{R}$ is convex (resp. concave), then the extension of $u$ to $\overline{\mathbb{D}^2}$ through the boundary value \eqref{boundaryvalue} is a lower semicontinuous (resp. upper semicontinuous) function. We also need the following result.
\begin{prop}\cite[Proposition 4.2]
{ConvexAnal}\label{boundaryvalueextension}
    Let $u: \mathbb{D}^2\to\mathbb{R}$ be a convex function (or concave). Then the boundary value of $u$ is a continuous function on $\mathbb{S}^1$ if and only if $u$ has a continuous extension to $\overline{\mathbb{D}^2}$.   
\end{prop}
For any function $\phi: \mathbb{S}^1 \to \mathbb{R}$, define two functions $\phi^-, \phi^+: \overline{\mathbb{D}^2} \to \mathbb{R}$ as follows:
\begin{equation}\label{phi-}
\phi^-(z)=\sup\{ a(z)\mid \ a:\mathbb{R}^2\to\mathbb{R} \ \mathrm{is}\ \mathrm{an} \ \mathrm{affine}\ \mathrm{function}\ \mathrm{with}  \ a_{\mid\mathbb{S}^1}\leq\phi   \},\end{equation}
\begin{equation}\label{phi+}
\phi^+(z)=\inf\{ a(z)\mid \ a:\mathbb{R}^2\to\mathbb{R} \ \mathrm{is}\ \mathrm{an} \ \mathrm{affine}\ \mathrm{function}\ \mathrm{with}  \ \phi\leq a_{\mid\mathbb{S}^1}   \}.\end{equation}
We need the following properties of the maps $\phi^-$ and $\phi^+$:\label{page 8}
\begin{enumerate}
    \item[(P1)]\label{P1} 
    We have $\phi^-\leq\phi^+$, moreover, if $\phi$ is continuous on $\mathbb{S}^1$ then $\phi^{\pm}$ is continuous in $\mathbb{D}^2$ with $\phi^{\pm}|_{\mathbb{S}^1}$ (see for instance \cite[Lemma 4.6]{ConvexAnal}).
\item[(P2)]\label{P2} If $u:\overline{\mathbb{D}^2}\to\mathbb{R}$ is a convex (resp. concave) function such that $u |_{\mathbb{S}^1}\leq \phi$ (resp. $\phi\leq u |_{\mathbb{S}^1} $), then $u\leq\phi^-$ (resp. $\phi^+\leq u$) in $\mathbb{D}^2$, see \cite[Corollary 4.5]{ConvexAnal}.
\end{enumerate}
Let $\phi:\mathbb{S}^1\to\mathbb{R}$ be a continuous map and let $\mathrm{gr}(\phi)$ be its graph. Denote by $\mathcal{C}(\phi)$ the convex hull of $\mathrm{gr}(\phi)$. That is, the smallest convex set of $\mathbb{R}^3$ containing $\mathrm{gr}(\phi)$. Note that since $\overline{\mathbb{D}}^2\times\mathbb{R}$ is a convex set containing $\mathrm{gr}(\phi)$, we have $\mathcal{C}(\phi)\subset\overline{\mathbb{D}}^2\times\mathbb{R}.$

We denote by $\partial\mathcal{C}(\phi)$ the boundary of $\mathcal{C}(\phi)$ in $\mathbb{D}^2 \times \mathbb{R}$. If $\phi$ is continuous and is not the restriction of an affine map of $\mathbb{R}^2$, then by the Jordan-Brouwer separation theorem, $\partial\mathcal{C}(\phi) \setminus \mathrm{gr}(\phi)$ has two connected components, which we denote as $\partial_+\mathcal{C}(\phi)$ and $\partial_-\mathcal{C}(\phi)$, respectively. It turns out that the connected components $\partial_{\pm}\mathcal{C}(\phi)$ are exactly the graphs of the maps $\phi^{\pm}$; see \cite[Lemma 2.41]{barbotfillastre}.
Following the work of \cite{barbotfillastre}, we have
\begin{defi}\label{width}
 Let $\phi:\mathbb{S}^1\to\mathbb{R}$ be a continuous function and $\mathrm{L}:\HP\to\mathbb{R}$ be the height function defined in \eqref{formule_L}. Then the \textit{width} of $\phi$ is defined as:
$$w(\phi):=\underset{(\eta,\phi^{+}(\eta)) \in \partial_- \mathcal{C}(\phi), \ (\eta,\phi^-(\eta))\in \partial_+\mathcal{C}(\phi)}{\sup} \vert \mathrm{L}(\eta,\phi^{-}(\eta))-\mathrm{L}(\eta,\phi^{+}(\eta))\vert  \in [0,+\infty].$$ In other words, by \eqref{formule_L_Klein}
$$w(\phi):=\underset{\eta\in\mathbb{D}^2}{\sup} \frac{\phi^+(\eta)-\phi^-(\eta)}{\sqrt{1-\lvert\eta\rvert^2}}.$$ 
\end{defi}
The quantity $\mathrm{L}(\eta,\phi^{-}(\eta))-\mathrm{L}(\eta,\phi^{+}(\eta))$ can be interpreted as the timelike distance between the parallel spacelike planes $\mathrm{P}_{(\eta,\phi^{-}(\eta))}$ and $\mathrm{P}_{(\eta,\phi^{+}(\eta))}$ of Minkowski space, dual to $(\eta,\phi^{-}(\eta))$ and $(\eta,\phi^{+}(\eta))$ respectively (see Equation \eqref{4}). As consequence, the function $$\mathrm{L}(\eta,\phi^{-}(\eta))-\mathrm{L}(\eta,\phi^{+}(\eta)): \eta\mapsto\frac{\phi^+(\eta)-\phi^-(\eta)}{\sqrt{1-\lvert\eta\rvert^2}}$$ is invariant by isometries of $\HP$ in the sense that for any $A\in\mathrm{O}_0(1,2)$ and $v\in\minko$.
\begin{equation}\label{invariance_hight_function}
    \mathrm{L}\big( \mathrm{Is}(A,v)\cdot (\eta,\phi^+(\eta))  \big)-\mathrm{L}\big( \mathrm{Is}(A,v)\cdot (\eta,\phi^-(\eta))  \big)=\mathrm{L}\big( (\eta,\phi^+(\eta))  \big)-\mathrm{L}\big((\eta,\phi^-(\eta))  \big).
\end{equation}
 \begin{figure}[h]
\centering
\includegraphics[width=.5\textwidth]{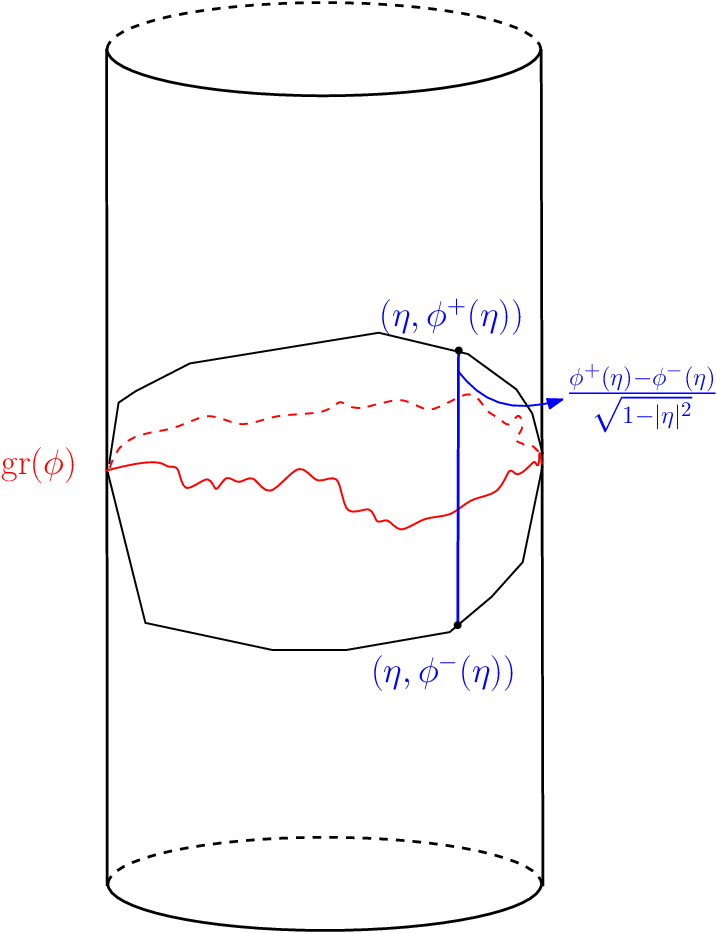}
\caption{An illustrative picture on the width of a continuous map $\phi.$}\label{width_picture}
\end{figure}  
\subsection{Vector fields on the circle and Half-pipe geometry}
The goal of this section is to explain how one can interpret a vector field on the circle as data at the boundary at infinity in $\HP.$ Consider 
$$N=\{(x_0,x_1,x_2)\in\minko, \ -x_0^2+x_1^2+x_2^2=0\}\setminus\{0\},$$ so that $\mathbb{P}(N)\cong\mathbb{S}^1$. The following lemma, as shown in \cite{BS17}, establishes a natural identification between vector fields on $\mathbb{S}^1$ and $1$-homogeneous functions on $N$.

\begin{lemma}\label{homogene}\cite[Lemma $2.23$]{BS17}
    There is a $1-$to$-1$ correspondence between vector fields $X$ on $\mathbb{S}^1$ and $1$-homogeneous functions $\Phi:N\to\mathbb{R}$ satisfying the following property: For any $\mathcal{C}^1$ spacelike section $s:\mathbb{S}^1\to N$ of the radial projection $\Pi:N\to\mathbb{S}^1$ and for $v$ the unit tangent vector field to
$s$ which is positively oriented, we have 
\begin{equation}\label{secctionhomogene}
s_*(X(z))=\Phi(s(z))v(s(z))
\end{equation}
\end{lemma}
To clarify the terminology in the Lemma: Given a $\mathcal{C}^1$ spacelike section $s:\mathbb{S}^1\to N$ of the radial projection $\Pi:N\to\mathbb{S}^1$ defined in \eqref{radial}, a unit vector field $v$ to $s$ is a map from the image of the section $s(\mathbb{S}^1)$ to the tangent bundle $\mathrm{T}(s(\mathbb{S}^1))$ such that the norm of $v$ is $1$ with respect to the Minkowski norm (this is possible because the section is $\mathcal{C}^1$). Note that the section $s$ has two unit tangent vectors $v$; we choose the one that is \textit{positively oriented}, meaning for each $z\in \mathbb{S}^1$ $$\det\biggl(\Pi(s(z)),\Pi_*\bigl(v(s(z))\bigr)\biggr)>0,$$ where $\Pi_*\bigl(v(s(z))\bigr)=\mathrm{d}_{s(z)}\Pi\bigl(v(s(z))\bigr)\in\mathbb{R}^2$. For example, if $s(z)=(1,z)$ is the horizontal section, then for $z=(x,y)$ we have $v(1,x,y)=(0,-y,x)$.

In this paper, we will primarily work within the Klein model $\mathbb{D}^2\times\mathbb{R}$ of the Half-pipe space, so the following definition is essential.

\begin{defi}\label{supportfunction}
Given a vector field $X$ on $\mathbb{S}^1$ and $\Phi:N\to \mathbb{R}$ a one-homogeneous function as in \eqref{secctionhomogene}, the \textit{support function} of $X$ is the function $\phi_X:\mathbb{S}^1\to \mathbb{R}$ defined by:
$$\phi_X(z)=\Phi(1,z).$$
Therefore, the vector field $X$ can be written as: \begin{equation}\label{supportfunction_vs_vectorfield}
    X(z)=iz\phi_{X}(z)
\end{equation}
\end{defi}
For instance, one can show that the support function of the restriction on $\mathbb{S}^1$ of a Killing vector field is the restriction of an affine map on $\mathbb{S}^1.$
\begin{cor}\cite[Corollary 2.10]{diaf2023Infearth}\label{support_function_of_killing} 
Let $\sigma\in \mathbb{R}^{1,2}$. The support function $\phi_{\Lambda(\sigma)}:\mathbb{S}^1\to \mathbb{R}$ of the Killing vector field $\Lambda(\sigma)$ is given by: 
$$\phi_{\Lambda(\sigma)}(z)=\inner{(1,z),\sigma}_{1,2}.$$
\end{cor}
The next lemma describes the behavior of a support function under the action of $\mathrm{O}_0(1,2)\ltimes\minko$ on vector fields of the circle. Recall that the linear part acts on the space of vector fields by pushforward, and the translation part acts by adding a Killing vector field.\begin{lemma}\cite[Lemma 2.12]{diaf2023Infearth}\label{equii}
    Let $X$ be a vector field of $\mathbb{S}^1$ and $\phi_X$ be its support function. Let $\mathrm{A}\in\mathrm{O}_0(1,2)$ and $\sigma\in\minko$. The support function $\phi_{\mathrm{A}_*X+\Lambda(\sigma)}$ of the vector field $\mathrm{A}_*X+\Lambda(\sigma)$ satisfies:\begin{equation}
          \mathrm{gr}(\phi_{\mathrm{A}_*X+\Lambda(\sigma)})=\mathrm{Is}(\mathrm{A},\sigma)\mathrm{gr}(\phi_X).
    \end{equation}
    
\end{lemma}
We now define the \textit{width} of a vector field.

\begin{defi}\label{widthX}
 Let $X$ be a continuous vector field of $\mathbb{S}^1$. Then the \textit{width} of $X$, denoted by $\omega(X)$, is defined as the width of $\phi_X$ (see Definition \ref{width}).
\end{defi}
It follows from Lemma \ref{equii} and the invariance of the height function in \eqref{invariance_hight_function} that the width is invariant under isometries of $\HP.$ Namely for any $A\in \mathrm{O}_{0}(1,2)$ and $\sigma\in\minko$, we have
$$\omega(A_*X+\Lambda(\sigma))=\omega(X).$$

\subsection{Mean surfaces in $\HP$}
The goal of this section is to introduce mean surfaces in $\HP.$ For a more detailed exposition, we refer the reader to \cite[Section 5.2-6.3]{andreafrancois} and \cite[Section 2.3.2]{barbotfillastre}.

Since there does not exist any pseudo-Riemannian metric on $\HP$ invariant under the isometry group of $\HP$ \cite[Fact 3.11]{fillastreseppi}, there is no notion of a unit normal vector of a surface in $\HP$. Nevertheless, one may use the existence of a vertical fiber to mimic classical affine differential geometry \cite{affine_diff_geometry}. Indeed, let $\mathcal{HP}^3$ be the lift of $\HP$ in $\mathbb{R}^4$:

$$\mathcal{HP}^3=\{(x_0,x_1,x_2,x_3)\in\mathbb{R}^4\mid -x_0^2+x_1^2+x_2^2=-1, \ x_0>0  \}.$$
Consider the transverse vector field $N_x=x$ in $\mathbb{R}^4$. This vector field $N$ is invariant under the isometries of $\HP$, which are naturally identified with $\mathrm{O}_{0}(2,1)\ltimes\minko$, as seen from the structure inside the brackets of \eqref{groupeduality}. This implies that for any isometry $\gamma\in\mathrm{O}_{0}(1,2)\ltimes\minko$
$$\gamma_{*}( N_x) = N_{\gamma(x)}.$$ Thus, one can use the vector field $N$ to decompose the ambient canonical connection $D$ of $\mathbb{R}^4$ into tangential and normal parts. Specifically, denote by $\inner{\cdot,\cdot}_{2,1,0}$ the bilinear form on $\mathbb{R}^4$ induced by the quadratic form:
$$q_{1,2,0}(x_0,x_1,x_1,x_2)=-x_0^2+x_1^2+x_2^2,$$ then we have the following definition.

\begin{defi} Let $V$ and $W$ be two vector fields on $\mathcal{HP}^3$. The connection
$\nabla^{\mathcal{HP}^3}$ is defined by:
$$D_VW=\nabla^{\mathcal{HP}^3}_VW+\inner{V,W}_{1,2,0}N$$
The \textit{Half-pipe} connection is the connection $\nabla^{\HP}$
induced on $\HP$ by $\nabla^{\mathcal{HP}^3}.$ 
\end{defi}
We will denote by $T$ the vector field on $\HP$ defined by $(0, 0, 0, 1)$ in $\mathcal{HP}^3$. It is a degenerate vector field invariant under $\isom(\HP).$

We can now define the second fundamental form of any spacelike surface in $\HP\cong\mathbb{H}^2\times\mathbb{R}$. Recall that a surface $S$ is said to be \textit{spacelike} if it is locally the graph of a function $u : \Omega \to\mathbb{R}$, for a domain $\Omega\subset \mathbb{H}^2$. Denote by $I$ the \textit{first fundamental} form of $S$ which is just the hyperbolic metric on the base $\mathbb{H}^2$. For a spacelike surface $S\subset\HP$ we have the splitting. 
\begin{equation}\label{splittingconnection}
    \mathrm{T}_{(x,t)}\HP=\mathrm{T}_{(x,t)}S\oplus T_{(x,t)}
\end{equation}
Following \cite{andreafrancois}, we define:
\begin{defi}\cite{andreafrancois}
Given a spacelike $S$ in $\HP$ and $v$, $w\in\mathrm{T}_{(x,t)}S$. Then we may define:
\begin{enumerate}
    \item The \textit{second fundamental form} of $S$:
$$\nabla_V^{\HP}W=\nabla_V^{I}W+\II(v,w) T$$
 where $V$ and $W$ are vector fields extending $v$ and $w$ and $\nabla^{I}$ is the Levi-Civita connection of the first fundamental form $I.$
 \item The \textit{shape operator} of $S$ is the self-adjoint $(1, 1)$-tensor given by
$$\II(v,w)=\I(\B(v),w).$$
\end{enumerate}
\end{defi}
\begin{defi}\label{def_mean_surface}
    A spacelike surface $S$ in $\HP$ is called \textit{mean} if $\mathrm{tr}_{\I}(\B) = 0$.
\end{defi}
In this paper, we focus on mean surfaces, which globally are graphs of functions on $\mathbb{H}^2$. Given a function $u:\mathbb{H}^2\to \mathbb{R}$, We define the \textit{hyperbolic Hessian} of $u$ as the $(1, 1)$-tensor
$$\hess u(v)=\nabla^{\mathbb{H}^2}_v\mathrm{grad}^{\mathbb{H}^2} u$$
where $\nabla^{\mathbb{H}^2}$ is the Levi-Civita connection of $\mathbb{H}^2.$ Finally, denote by $\Delta^{\mathbb{H}^2}$ the Laplace-Beltrami operator of $\mathbb{H}^2$, which can be defined as: $$\Delta^{\mathbb{H}^2}u=\mathrm{tr}_{g^{\mathbb{H}^2}}(\hess(u)).$$
The next Lemma provides a formula for the shape operator of the graph of $u$.

\begin{lemma}\cite[Lemma 9.1.7]{Seppi_thesis}\label{Shapeoperator_of_graph}
Let $u:\mathbb{H}^2\to\mathbb{R}$ be a smooth function, and let $S\subset\HP$ be the graph of $u$. Then, the shape operator of $S$ is given by 
  $$\B=\hess(u)-u\mathbbm{1},$$ where $\mathbbm{1}$ is the identity operator. Thus $S$ is a mean surface if and only if $\Delta^{\mathbb{H}^2} u-2u=0.$
\end{lemma}
The problem of existence for mean surfaces with a prescribed curve at infinity is established by Barbot and Fillastre \cite{barbotfillastre}; see also Seppi's thesis \cite[Chapter 9]{Seppi_thesis}. To explain this, it is convenient to work in the Klein model $\mathbb{D}^2\times\mathbb{R}$. Hereafter, for a function $u:\mathbb{H}^2\to\mathbb{R}$, we denote by $\overline{u}:\mathbb{D}^2\to\mathbb{R}$ the function given by
\begin{equation}\label{u_et_u_bar}
u(\Pi^{-1}(\eta))=\frac{\overline{u}(\eta)}{\sqrt{1-\lvert\eta\rvert^2}},\end{equation} where $\Pi$ is the radial projection defined in \eqref{radial}. The relationship between $u$ and $\overline{u}$ should be understood as follows: Let $\I^+(0)$ be the \textit{future cone} over $0\in\minko$, namely
$$\I^+(0)=\{(x_0,x_1,x_2)\in\minko\mid -x_0^2+x_1^2+x_2^2<0    \}.$$
Consider $U:\I^{+}(0)\to\mathbb{R}$ to be the unique $1$-homogeneous function such that $U|_{\mathbb{H}^2}=u$; then it is immediate to check that $\overline{u}(\eta)=U(1,\eta).$ We refer the reader to \cite{françois_vernolli} for more properties about the $1$-homogeneous function $U$.

\begin{remark}
Let $U$ be a $1$-homogeneous function on $\I^{+}(0)$, and consider $U|_{\mathbb{H}^2}=u$ and $\overline{u}(\eta)=U(1,\eta)$. Given an isometry $(A,v)\in\mathrm{O}_0(1,2)\ltimes\minko$ and the $1$-homogeneous function $H$ defined by:
$$H(x)=U(A^{-1}x)+\inner{x,v}_{1,2}.$$ Let $H|_{\mathbb{H}^2}=h$, then using Lemma \ref{rotationHP}, it is not difficult to verify that the graph of $\overline{h}$ and $\overline{u}$ are related by the following relation:
\begin{equation}\label{h_et_h_bar}
\mathrm{gr}(\overline{h})=\mathrm{Is}(A,v)(\mathrm{gr}(\overline{u})).
\end{equation}
\end{remark}

\begin{prop}\cite[Page 668, Proposition 16.2.37, Lemma 16.2.42]{barbotfillastre}\label{BF_platau_HP}
    Let $\phi:\mathbb{S}^1\to\mathbb{R}$ be a continuous map. Then there exists a unique smooth mean surface $S$ in $\HP\cong\mathbb{H}^2\times\mathbb{R}$ with boundary at infinity $\mathrm{gr}(\phi).$ More precisely, there exists a unique function $u:\mathbb{H}^2\to\mathbb{R}$ which satisfies 
    \begin{equation}\label{platau_HP}
    \begin{cases}
\Delta^{\mathbb{H}^2} u-2u=0  \\
\overline{u}|_{\mathbb{S}^1}=\phi.
\end{cases}\end{equation}
Equivalently if we denote by $\Delta^{\mathbb{E}^2}$ and $\mathrm{Hess}^{\mathbb{E}^2}$ the Laplacian and Hessian on $\mathbb{R}^2$ respectively, then $\overline{u}$ is the unique solution of \begin{equation}\label{platau_HPE}
    \begin{cases}
\Delta^{\mathbb{E}^2}\overline{u}(\eta)-\mathrm{Hess}_{\eta}^{\mathbb{E}^2}(\eta,\eta)=0  \\
\overline{u}|_{\mathbb{S}^1}=\phi.
\end{cases}\end{equation}
In addition, we have $\phi^-\leq\overline{u}\leq \phi^+$, namely the graph of $\overline{u}$ is contained in the convex core of $\mathrm{gr}(\phi).$ 
\end{prop}

\section{Surfaces in $\HP$ and vector fields of $\mathbb{H}^2$}\label{sec4_correpondance}
The goal of this section is to establish a correspondence between spacelike surfaces in $\HP$ and vector fields in $\mathrm{T}\mathbb{H}^2$. Subsequently, we will characterize the vector fields constructed from embedded spacelike surfaces in $\HP$ in terms of the geometry of the tangent bundle of $\mathbb{H}^2$, see Theorem \ref{Lagrange_vs_divergence_free}.

\subsection{The geometry of the tangent bundle of $\mathbb{H}^2$}\label{The_geometry_of_the_tangent bundle}
We briefly recall how the tangent bundle of $\mathbb{H}^2$ is endowed with a natural pseudo-Kähler structure and refer the reader to \cite{canonical_pseudo} for a more general construction on the tangent bundle of a general Kähler manifold. Let $\mathrm{T}\mathbb{H}^2$ be the tangent bundle of $\mathbb{H}^2$, and we denote by $p:\mathrm{T}\mathbb{H}^2\to\mathbb{H}^2$ the canonical projection. The sub-bundle $\mathrm{ker}(\mathrm{d}p) := \mathcal{V}\mathbb{H}^2$ of $\mathrm{T}\mathrm{T}\mathbb{H}^2$ will be called the \textit{vertical bundle}. For $(x,v)\in \mathrm{T}\mathbb{H}^2$ and $w\in\mathrm{T}_x\mathbb{H}^2$, we define:
$$\begin{array}{ccccl}
F^{\mathcal{V}}_{(x,v)} & : & \mathrm{T}_x\mathbb{H}^2 & \to & \mathrm{T}_{(x,v)}\mathrm{T}\mathbb{H}^2 \\
 & & w &\mapsto &  w^{\mathcal{V}}:=\frac{d}{dt}\big\lvert_{t=0} (x,v+tw) \\
\end{array}.$$
We call $w^{\mathcal{V}}$ the vertical lift of $w$. The map $F^{\mathcal{V}}_{(x,v)}$ is injective with image $\mathcal{V}\mathbb{H}_{(x,v)}$.
The Levi-Civita connection $\nabla^{\mathbb{H}^2}$ of $\mathbb{H}^2$ allows us to define a horizontal bundle $\mathcal{H}\mathbb{H}^2$ in such a way that we have the splitting:
$$\mathrm{T}\mathrm{T}\mathbb{H}^2=\mathcal{H}\mathbb{H}^2\oplus \mathcal{V}\mathbb{H}^2.$$
For $(x,v)\in \mathrm{T}\mathbb{H}^2$ and $w\in\mathrm{T}_x\mathbb{H}^2$, we set

$$\begin{array}{ccccl}
F_{(x,v)}^{\mathcal{H}} & : &\mathrm{T}_x\mathbb{H}^2 & \to & \mathrm{T}_{(x,v)}\mathrm{T}\mathbb{H}^2  \\
 & & w &\mapsto &  w^{\mathcal{H}}:=\frac{d}{dt}\big\lvert_{t=0} (c(t),V(t)) \\
\end{array}$$
where $c:\mathbb{R}\to\mathbb{H}^2$ is the parameterized geodesic with $c(0)=x$ and $c'(0)=w$, and $V$ is the parallel vector field along $c$ such that $V(0)=v$. The map $F_{(x,v)}$ is injective, and we call $w^{\mathcal{H}}$ the \textit{horizontal lift} of $w.$ The fiber on $(x,v)$ of the horizontal bundle $\mathcal{H}\mathbb{H}^2$ is then just the image of $F_{(x,v)}$. In conclusion, we have the identification

\begin{equation}\label{splittangent}
 \begin{array}{ccccl}
 &  & \mathrm{T}_x\mathbb{H}^2\times \mathrm{T}_x\mathbb{H}^2 & \cong & \mathrm{T}_{(x,v)}\mathrm{T}\mathbb{H}^2=\mathcal{H}\mathbb{H}^2_{(x,v)}\oplus \mathcal{V}\mathbb{H}^2_{(x,v)} \\
 & & (w_1,w_2) & \mapsto & w_1^{\mathcal{H}}+w_2^{\mathcal{V}} \\
\end{array}.
\end{equation}

\begin{theorem}\cite{canonical_pseudo}\label{canonical_pseudo}
Under the identification \eqref{splittangent}, we define:
\begin{itemize}
\item A pseudo-Riemannian metric $\mathbb{G}$ on $\mathrm{T}\mathbb{H}^2$ of signature $(2,2)$ given by:
$$\mathbb{G}\big((X_1,X_2),(Y_1,Y_2)\big)=g_{\mathbb{H}^2}(X_1,Y_2)+g_{\mathbb{H}^2}(X_2,Y_1).$$
\item An almost complex structure $\mathbb{J}$ on $\mathrm{T}\mathbb{H}^2$ (i.e., $\mathbb{J}^2=-\mathrm{Id}$):
$$\mathbb{J}(X_1,X_2)=(\mathrm{J}X_1,\mathrm{J}X_2).$$
\item A differential $2$-form $\Omega$ on $\mathrm{T}\mathbb{H}^2$ given by:
$$\Omega((X_1,X_2),(Y_1,Y_2))=\mathrm{g}^{\mathbb{H}^2}(\mathrm{J}X_1,Y_2)-g^{\mathbb{H}^{2}}(X_2,\mathrm{J}Y_1).$$
\end{itemize}
Then the triple $(\mathbb{G},\mathbb{J},\Omega)$ defines a pseudo-Kähler structure on $\mathrm{T}\mathbb{H}^2$. That is, $\mathbb{J}$ is an integrable almost complex structure such that $\Omega=\mathbb{G}(\mathbb{J}\cdot,\cdot)$ is a symplectic form (a non-degenerate closed 2-form).
\end{theorem}
\begin{remark}
The original theorem in \cite{canonical_pseudo} is stated differently. Instead of taking $(\mathbb{G},\mathbb{J},\Omega)$, the authors consider the pullback of $(\mathbb{G},\mathbb{J},\Omega)$ by the diffeomorphism $$
\begin{array}{ccccl}
& & \mathrm{T}\mathbb{H}^2 & \to & \mathrm{T}\mathbb{H}^2 \\
& & X & \mapsto & \mathrm{J}X \\
\end{array}.$$
\end{remark}

\subsection{Divergence free vector fields of $\mathbb{H}^2$ and surfaces in $\HP$}
The purpose of this section is to explain the construction which associates to a smooth spacelike surface in $\HP$ a vector field of $\mathbb{H}^2.$ As discussed in Section \ref{section3}, each element of $\mathfrak{isom}(\mathbb{H}^2)$ corresponds to a Killing vector field of $\mathbb{H}^2$. Consequently, we establish a correspondence:

\begin{equation}\label{duality_map}
    \begin{array}{ccccc}
 &  & \{\mathrm{Spacelike}\ \mathrm{planes}\ \mathrm{in}\ \mathbb{D}^2\times\mathbb{R}\} & \cong & \{ \mathrm{Killing}\ \mathrm{vector} \ \mathrm{fields}\ \mathrm{of}\ \mathbb{H}^2\} \\
 & & \mathrm{P}_v& \mapsto & \Lambda(v) \\
\end{array}
\end{equation}
where we recall that for $v\in\minko$,  $$\mathrm{P}_v:=\{(\eta,t)\in\mathbb{D}^2\times\mathbb{R} \mid \langle (1,\eta),v  \rangle_{1,2}=t  \},$$
and $\Lambda(v)$ is the Killing vector field of $\mathbb{H}^2$ given by $\Lambda(v)(p)=p\boxtimes v$. 
Let $u:\mathbb{H}^2\to \mathbb{R}$ be a smooth function so that the graph $S:=\mathrm{gr}(u)$ of $u$ is an embedded spacelike surface in $\HP.$ We now describe how to construct a vector field $V_S$ using the identification \eqref{duality_map}. Let $p_0\in\mathbb{H}^2$ and $\eta_0=\Pi(p_0)\in \mathbb{D}^2$ and consider the affine function $a:\mathbb{D}^2\to \mathbb{R}$ given by:
$$a(\eta)=\overline{u}(\eta_0)+\mathrm{d}_{\eta_0}\overline{u}(\eta-\eta_0),$$ where $\overline{u}: \mathbb{D}^2 \to \mathbb{R}$ is the function defined in \eqref{u_et_u_bar}.
The graph of $a$ is the tangent plane of $\mathrm{gr}(\overline{u})$ at $(\eta_0,\overline{u}(\eta_0))$. Since this plane is spacelike, there is $\sigma\in\minko$ such that $\mathrm{gr}(a)=\mathrm{P}_{\sigma}$. We define the vector field associated to $S$ as follows:
\begin{equation}\label{XS}
    V_S(p_0)=\Lambda(\sigma)= p_0\boxtimes \sigma.
\end{equation}
The next lemma provides an explicit description of $\sigma$ in terms of the gradient of $u$.

\begin{lemma}\label{sigma_lemma}
  Let $u:\mathbb{H}^2\to\mathbb{R}$ be a smooth function. Fix $p_0\in\mathbb{H}^2$ and $\eta_0=\Pi(p_0)$. Let $a:\mathbb{D}^2\to \mathbb{R}$ be the affine function given by:
  $$a(\eta)=\overline{u}(\eta_0)+\mathrm{d}_{\eta_0}\overline{u}(\eta-\eta_0).$$
   Let $\sigma=\grad_{p_0}u-u(p_0)p_0$. Then $$a(\eta)=\inner{(1,\eta),\sigma}_{1,2}.$$ In other words, $\mathrm{P}_{\sigma}$ is the graph of the affine function $a$.
\end{lemma}

\begin{proof}
Let $L:\mathbb{D}^2\to\mathbb{R}$ be the function defined by $$L(\eta)=\sqrt{1-\lvert\eta\rvert^2}.$$ Recall that the radial projection defined in \eqref{radial} induces an isometry between $\mathbb{D}^2$ and $\mathbb{H}^2$, so consider $\eta_0=(x_0,y_0)\in \mathbb{D}^2$ and $p_0=\Pi^{-1}(\eta_0)$. By \eqref{u_et_u_bar}, \(u \circ \Pi^{-1} = L^{-1}\overline{u}\), and thus taking the differential at \(\eta_0\), we obtain that for all \(h \in \mathbb{R}^2\):
\begin{equation}\label{eq_sigma_27}
\langle \mathrm{d}_{\eta_0}\Pi^{-1}(h), \grad_{p_0} u \rangle_{1,2} = \mathrm{d}_{\eta_0}L^{-1}(h) \overline{u}(\eta_0) + L^{-1}(\eta_0) \mathrm{d}_{\eta_0}\overline{u}(h).
\end{equation} By an elementary computation, the differential of $\Pi$ is given by:
$$\mathrm{d}_{p_0}\Pi=\sqrt{1-\lvert \eta_0\rvert^2}\begin{pmatrix}
-x_0 & 1  & 0  \\
-y_0 & 0  & 1
\end{pmatrix},$$
In particular, 
$$\mathrm{d}_{p_0}\Pi((1,\eta))=\sqrt{1-\lvert\eta_0\rvert^2}(\eta-\eta_0).$$
Let us write \((1, \eta) = v_1 + v_2\), where \(v_1 \in \mathrm{T}_{p_0}\mathbb{H}^2\) and \(v_2 \in \mathbb{R} \cdot p_0\). Since \(\mathrm{d}_{p_0}\Pi(p_0) = 0\), we get
$$\mathrm{d}_{p_0}\Pi(v_1) = \sqrt{1 - \lvert \eta_0 \rvert^2} (\eta - \eta_0).$$
Since \(\mathrm{d}_{\eta_0}\Pi\) is an isomorphism from \(\mathrm{T}_{p_0}\mathbb{H}^2\) to \(\mathrm{T}_{\eta_0}\mathbb{D}^2 \cong \mathbb{R}^2\), we deduce that
\begin{equation}\label{eq_sigma_277}
    \mathrm{d}_{\eta_0}\Pi^{-1}(\eta - \eta_0) = \frac{v_1}{L(\eta_0)}.
\end{equation}
For \(h = \eta - \eta_0\), we get from \eqref{eq_sigma_277} and \eqref{eq_sigma_27}:
$$\langle \frac{v_1}{L(\eta_0)}, \grad_{p_0} u \rangle_{1,2} = \mathrm{d}_{\eta_0}L^{-1}(\eta - \eta_0) \overline{u}(\eta_0) + L^{-1}(\eta_0) \mathrm{d}_{\eta_0}\overline{u}(\eta - \eta_0).$$
Equivalently,
\begin{equation}\label{eq_sigma_28}
\langle v_1, \grad_{p_0} u \rangle_{1,2} = L(\eta_0) \mathrm{d}_{\eta_0}L^{-1}(\eta - \eta_0) \overline{u}(\eta_0) + \mathrm{d}_{\eta_0}\overline{u}(\eta - \eta_0).
\end{equation}
Using the fact that \((1, \eta_0) = v_1 + v_2\) and \(\langle \grad_{p_0} u, v_2 \rangle_{1,2} = 0\), we can rewrite \eqref{eq_sigma_28} as:
\begin{equation}\label{eq_sigma_29}
\langle (1, \eta), \grad_{p_0} u \rangle_{1,2} = L(\eta_0) \mathrm{d}_{\eta_0}L^{-1}(\eta - \eta_0) \overline{u}(\eta_0) + \mathrm{d}_{\eta_0}\overline{u}(\eta - \eta_0).
\end{equation}
Note that for all \(h \in \mathbb{R}^2\),
$$\mathrm{d}_{\eta_0}L^{-1}(h) = L^{-3}(\eta_0) \langle \eta_0, h \rangle_2,$$ where \(\langle \cdot, \cdot \rangle_2\) is the standard Euclidean metric on \(\mathbb{R}^2\). It follows from \eqref{eq_sigma_29}:
\begin{align}\label{eq_sigma_30}
\begin{split}
\langle (1, \eta), \grad_{p_0} u \rangle_{1,2} &= L^{-2}(\eta_0) \langle \eta - \eta_0, \eta_0 \rangle_2 \overline{u}(\eta_0) + \mathrm{d}_{\eta_0}\overline{u}(\eta - \eta_0) \\
&= L^{-2}(\eta_0) (\langle \eta, \eta_0 \rangle_2 - \lvert \eta_0 \rvert^2) \overline{u}(\eta_0) + \mathrm{d}_{\eta_0}\overline{u}(\eta - \eta_0).
\end{split}
\end{align}
On the other hand, we have:
\begin{equation}\label{eq_sigma_31}
\langle (1, \eta), u(p_0) p_0 \rangle_{1,2} = L^{-2}(\eta_0) (-1 + \langle \eta, \eta_0 \rangle_{2}) \overline{u}(\eta_0).
\end{equation}
Thus, we deduce from \eqref{eq_sigma_30} and \eqref{eq_sigma_31}:
\begin{align*}
\langle (1, \eta), \grad_{p_0} u - u(p_0) p_0 \rangle_{1,2} &= L^{-2}(\eta_0) \big(\langle \eta, \eta_0 \rangle_2 - \lvert \eta_0 \rvert^2 + 1 - \langle \eta, \eta_0 \rangle_{2}\big) \overline{u}(\eta_0) + \mathrm{d}_{\eta_0}\overline{u}(\eta - \eta_0) \\
&= \overline{u}(\eta_0) + \mathrm{d}_{\eta_0}\overline{u}(\eta - \eta_0) \\
&= a(\eta).
\end{align*}
This finishes the proof.
\end{proof}
As consequence, we get the following corollary.
\begin{cor}\label{cor_formule_XS}
    Let $u:\mathbb{H}^2\to\mathbb{R}$ be a smooth function and $S$ be the graph of $u$. Then the vector field associated to $S$ is given by:
    $$V_S=\J\grad(u).$$
\end{cor}
\begin{proof}
 Let $p\in\mathbb{H}^2$. It follows from Lemma \ref{sigma_lemma} and from the definition of $V_S$ (see \eqref{XS}) that $$V_S(p)=p\boxtimes\sigma,$$ where
 $\sigma=\grad_{p}u-u(p)p$. Since $p\boxtimes p=0$, we have 
 $$V_S(p)=p\boxtimes\grad_{p}u.$$ The conclusion follows from the interpretation of $\J$ in terms of the Minkowski cross product (see \eqref{complex_structure_H2}).
\end{proof}

\begin{remark}
It is worth noting the following:
 \begin{enumerate}
     \item The above construction of the vector field is invariant under the normal evolution of a surface along the vertical fiber on $\HP$. More precisely, let $u:\mathbb{H}^2\to\mathbb{R}$ be a smooth function and $S=\mathrm{gr}(u)$. For $t\in\mathbb{R}$, define the parallel surface $S_t$ as the graph of $u+t$. Then it follows for Corollary \ref{cor_formule_XS} that $V_{S+t}=V_S.$
     \item Similarly, one may define a vector field from a convex function $u:\mathbb{H}^2\to\mathbb{R}$ (i.e. $\overline{u}$ convex) which is possibly discontinuous. One needs to replace the tangent plane by a choice \textit{support plane}. For instance, this construction is used in \cite{diaf2023Infearth} to study \textit{infinitesimal earthquakes}; see Section \ref{appendix6.1}.
 \end{enumerate}
\end{remark}
The next goal is to understand which vector fields on $\mathbb{H}^2$ can be obtained from an embedded spacelike surface in $\HP$ as in \eqref{XS}. Let $\mathrm{div}$ denote the divergence operator on $\mathbb{H}^2$ defined by

$$\mathrm{div}(V)=\mathrm{tr}(\nabla V).$$ In explicit terms, this is 
\begin{equation}\label{formula_divergence}
    \mathrm{div}(V)=g^{\mathbb{H}^2}(\nabla^{\mathbb{H}^2}_{e_1}V,e_1)+g^{\mathbb{H}^2}(\nabla^{\mathbb{H}^2}_{e_2}V,e_2),
\end{equation} where $e_i$ is any orthonormal basis. A vector field $V$ on $\mathbb{H}^2$ is said to be \textit{divergence-free} if $\mathrm{div}(V)=0.$ For the sake of clarity, we will henceforth denote the connection of Levi-Civita of $\mathbb{H}^2$ by $\nabla$ instead of $\nabla^{\mathbb{H}^2}$.

We can now state the principal result of this section which is the content of Theorem \ref{intro_divergence}.
\begin{theorem}\label{Lagrange_vs_divergence_free}
Let $V:\mathbb{H}^2\to\mathrm{T}\mathbb{H}^2$ be a smooth vector field on $\mathbb{H}^2.$ The following are equivalent:
\begin{enumerate}
\item There exists a smooth function $u:\mathbb{H}^2\to\mathbb{R}$ such that $V=\J\grad(u).$ In other words $V=V_S$, which is the vector field associated to the surface $S=\mathrm{gr}(u)\subset\HP$ (see Corollary \ref{cor_formule_XS}).
\item $V(\mathbb{H}^2)$ is a \textit{Lagrangian} surface in $\mathrm{T}\mathbb{H}^2$ with respect to the symplectic form $\Omega$.
\item $V$ is divergence-free vector field of $\mathbb{H}^2$.
\end{enumerate}
\end{theorem}
To prove this result, we need the following proposition which expresses the differential of a vector field under the decomposition \eqref{splittangent}:
\begin{prop}\cite[Proposition 1.6]{on_harmonic_vector_field}\label{dV(X)}
Let $V:\mathbb{H}^2\to\mathrm{T}\mathbb{H}^2$ be a vector field on $\mathbb{H}^2.$ Then for each vector field $X$ we have:
      $$\mathrm{d}V(X)=(X,\nabla_XV).$$
\end{prop}
\begin{proof}[Proof of Theorem \ref{Lagrange_vs_divergence_free}]
Let $X$ and $Y$ be two vector fields on $\mathbb{H}^2$, then
\begin{align*}
(\Omega^{*}V)(X,Y)&=\Omega(\mathrm{d}V(X),\mathrm{d}V(Y))\\
&=\Omega((X,\nabla_XV),(Y,\nabla_YV))\\
&= g^{\mathbb{H}^2}(\mathrm{J}X,\nabla_YV)-g^{\mathbb{H}^{2}}(\nabla_XV,\J Y). 
\end{align*}
On the other hand, if we consider the one-form given by
\begin{equation}\label{dualform}
\alpha := g^{\mathbb{H}^2}(\J V,\cdot ),
\end{equation}
then by simple computation, $\alpha$ satisfies
\begin{align*}
\mathrm{d}\alpha(X,Y)&=\mathrm{d}_X \alpha(Y)-\mathrm{d}_Y\alpha(X)-\alpha([X,Y])\\
&=g^{\mathbb{H}^2}(\nabla_X \J V,Y)-g^{\mathbb{H}^2}(\nabla_Y \J V,X).
\end{align*}
Since $\J$ is parallel with respect to $\nabla$ (i.e., $\nabla_X \J Y = \J \nabla_X Y$ for all vector fields $X$ and $Y$ on $\mathbb{H}^2$), then 
\begin{align*}
\mathrm{d}\alpha(X,Y)&=g^{\mathbb{H}^2}(\J\nabla_X V,Y)-g^{\mathbb{H}^2}(\J\nabla_Y V,X)\\
&=-g^{\mathbb{H}^2}(\nabla_X V,\J Y)+g^{\mathbb{H}^2}(\nabla_Y V,\J X).\\
&=(\Omega^{*}V)(X,Y).
\end{align*}
Now, let $e_1, e_2$ be an oriented orthonormal local frame of $g^{\mathbb{H}^2}$ so that $\J e_1=e_2$ and $\J e_2=-e_1$. Then we get
\begin{equation}\label{eq_div_free_closed}
\mathrm{div}(V)=\mathrm{d}\alpha(e_1,e_2)=(\Omega^{*}V)(e_1,e_2).
\end{equation}
The last equation \eqref{eq_div_free_closed} shows the equivalence between $(2)$ and $(3)$. To get the equivalence between $(1)$ and $(3)$, observe that by \eqref{eq_div_free_closed} $V$ is divergence-free if and only if $\alpha$ is a closed $1$-form, but $\mathbb{H}^2$ is simply connected, so we conclude that there exists a function $f: \mathbb{H}^2 \rightarrow \mathbb{R}$ such that $\alpha = \mathrm{d}f$. It follows from \eqref{dualform} that $\J V = \grad(f)$, and hence $V = \J\grad(u)$ with $u=-f$.
\end{proof}
                          
\subsubsection{Normal congruence in Minkowski space}
The goal of this part is to briefly explain how the vector field constructed in \ref{cor_formule_XS} can be obtained from the \textit{normal congruence} of a spacelike surface in Minkowski space. The general idea is that the space of geodesics of certain pseudo-Riemannian manifolds $M$ has a rich geometric structure. The study of such structures goes back to the seminal work of Hitchin \cite{Monopole_Hitchin}, who described a natural complex structure of the space of oriented straight lines in Euclidean $3$-space. An important aspect of this is the study of the lift of a submanifold in $M$ into the space of geodesics of $M$ (this lift is generally called the normal congruence). The reader can also consult \cite{ML_in_tangent_bundle, Anciaux_space_of_geodesic, Emam_Seppi1} for related works on normal congruences.

In this paper, we are interested in $\mathrm{T}\mathbb{H}^2$, which has a natural identification with the space of oriented timelike geodesics in $\minko$, denoted by $\mathbb{L}^{-1}(\minko)$. Following \cite[Section 6.6]{DGK_complete_lorentz}, the identification is given by
\begin{equation}
    \begin{array}{ccccc}\label{Lpv}
 &  & \mathrm{T}\mathbb{H}^2 & \to & \mathbb{L}^{-1}(\minko) \\
 & & (p,w) & \mapsto & \mathrm{L}_{p,w}:=\{\sigma\in\minko \mid p\boxtimes\sigma=w\}. \\
\end{array}\end{equation}
Moreover, we have the following equivariance properties; for $A\in \mathrm{O}_0(1,2)$ and $v\in\minko$
\begin{equation}
    \mathrm{L}_{Ap,Aw}=A(\mathrm{L}_{p,w}),\ \ \ \mathrm{L}_{p,p\boxtimes v+w}=\tau_v(\mathrm{L}_{p,w}),
\end{equation} where $\tau_v$ is the translation by $v$ in $\minko.$
If $c:\mathbb{R}\to\minko$ is a future timelike geodesic which is parameterized by arc length, then it is not difficult to check that
\begin{equation}\label{c(R)}
    c(\mathbb{R})=\mathrm{L}_{c'(0),c'(0)\boxtimes c(0)}
\end{equation}
Given a spacelike surface $\Sigma$ in $\minko$, i.e. for every $x\in\Sigma$, the tangent plane $\mathrm{T}_x\Sigma$ is a spacelike plane of $\minko$. There is a notion of \textit{Gauss map}:
  $$G_{\Sigma}:\Sigma\to\mathbb{H}^2$$
  which maps $x\in\Sigma$ to the normal of $S$ at $x$, i.e. the future unit timelike vector orthogonal to $T_x\Sigma$. For a spacelike surface $\Sigma$ of $\minko$ for which the Gauss map $G_{\Sigma}$ is invertible, one may lift $\Sigma$ in $\mathbb{L}^{-1}(\minko)$ to define a \textit{normal congruence} $\overline{\Sigma}$ as follows:
\begin{equation}
    \overline{\Sigma}=\{\mathrm{L}_{p,p\boxtimes G^{-1}_{\Sigma}(p)}\mid p\in\mathbb{H}^2\}.
\end{equation}
By \eqref{c(R)}, $\mathrm{L}_{p,p\boxtimes G^{-1}_{\Sigma}(p)}$ is the timelike geodesic in $\minko$ going through $G_{\Sigma}^{-1}(p)$ with speed $p.$
Under the identification \eqref{Lpv}, the surface $\overline{\Sigma}$ gives rise to the following vector field on $\mathbb{H}^2$ $$X(p)=p\boxtimes G_{\Sigma}^{-1}(p).$$
Consider $u_{\Sigma}:\mathbb{H}^2\to\mathbb{R}$ the \textit{support function} of $\Sigma$:
$$u_{\Sigma}(p)=\sup_{x\in\Sigma}\inner{x,p}_{1,2}.$$ It turns out that the $G_{\Sigma}^{-1}$ is related to $u_{\Sigma}$ by the formula:
$$G_{\Sigma}^{-1}(p)=\grad_p(u)-u(p)p,$$ and hence $$V(p)=p\boxtimes \grad_p(u)=\J_p\grad_p(u).$$ Then, one may notice that in the notation of Corollary \ref{cor_formule_XS}, we have $V=V_S$ where $S$ is the graph of the support function $u_{\Sigma}$.
We refer the reader to \cite{BS17} and the references therein for a more detailed exposition on the support function of surfaces in $\minko$.

\section{Harmonic Lagrangian Vector Fields}\label{sec5_harmo}

In this section, we introduce harmonic Lagrangian vector fields and prove several characterizations of them, as detailed in Theorem \ref{charac_HL}.

\subsection{Definition and Characterizations}
We begin this section by defining harmonic Lagrangian vector fields using the pseudo-Riemannian metric and the symplectic form of $\mathrm{T}\mathbb{H}^2$ described in the previous section. To do this, we need to recall the notion of a \textit{harmonic map} between pseudo-Riemannian manifolds.
Let $F:(M,g^M)\to(N,g^N)$ be a smooth map between two pseudo-Riemannian manifolds. Denote by $\nabla^{M}$ and $\nabla^{N}$ the Levi-Civita connections of $g^M$ and $g^N$ respectively. If $F^*(\mathrm{T}N)$
denotes the pullback bundle, we consider $\nabla^2F$ to be the section of $\mathrm{T}^*M\otimes \mathrm{T}^*M\otimes F^*(\mathrm{T}N)$ defined by 
$$\nabla^2F(X,Y)=\nabla_{\mathrm{d}F(X)}^{N}\mathrm{d}F(Y)-\mathrm{d}F(\nabla^{M}_XY),$$ for all vector fields $X$ and $Y$ on $M$.
The \textit{tension field} of $F$ is the section of $F^*(\mathrm{T}N)$ given by: 
 $$\tau(F):=\mathrm{tr}_{g^{M}}\big(\nabla^2F   \big).$$
Namely, if we fix $(e_i)$, a local orthonormal local frame for $g_M$, then 
 $$\tau(F)=\sum_{i}g^M(e_i,e_i)\nabla^2F(e_i,e_i).$$
Note that if $M$ is a Riemannian manifold, then $g^{M}(e_i,e_i)$ is always equal to $1$.
\begin{defi}\label{harmonic_def}
Let $F:(M,g^M)\to(N,g^N)$ be a smooth map between two pseudo-Riemannian manifolds. Then $F$ is called \textit{harmonic} if $\tau(F)=0.$
\end{defi}
We can now state the Definition of harmonic Lagrangian vector field.
\begin{defi}\label{def1}
We say that a vector field $V:\mathbb{H}^2\to\mathrm{T}\mathbb{H}^2$ is Harmonic Lagrangian if it satisfies the following conditions:
\begin{itemize}
\item $V:(\mathbb{H}^2,g_{\mathbb{H}^2})\to(\mathrm{T}\mathbb{H}^2,\mathbb{G})$ is an harmonic map.
\item $V(\mathbb{H}^2)$ is a Lagrangian surface in $\mathrm{T}\mathbb{H}^2$ with respect to the symplectic structure $\Omega$. (See Theorem \ref{canonical_pseudo}).
\end{itemize}
\end{defi}
Konderak \cite{on_harmonic_vector_field} established necessary and sufficient conditions under which a vector field $V:(\mathbb{H}^2,g_{\mathbb{H}^2})\to(\mathrm{T}\mathbb{H}^2,\mathbb{G})$ is an harmonic map, which we will explain now. For clarity, we will use the notation "$\mathrm{tr}$" to denote the trace of a $2-$tensor on $\mathbb{H}^2$, without explicitly mentioning the hyperbolic metric $g^{\mathbb{H}^2}$.

\begin{defi}
The \textit{rough Laplacian} of a vector field $V$ on $\mathbb{H}^2$ is defined by
$$\overline{\Delta}V=\mathrm{tr}\big((X,Y)\mapsto \nabla_X\nabla_Y V-\nabla_{\nabla_X Y}V    \big).$$
\end{defi}
We denote by $\mathrm{R}$ the Riemann tensor of $(\mathbb{H}^2,g^{\mathbb{H}^2})$ of type $(3,1)$.
$$\mathrm{R}(X,Y)Z=\nabla_X\nabla_YZ-\nabla_Y\nabla_XZ-\nabla_{[X,Y]}Z.$$
Then we have 
\begin{theorem}\cite[Proposition 4.1]{on_harmonic_vector_field}\label{on_harmonic_vector_field}
A vector field $V:(\mathbb{H}^2,g^{\mathbb{H}^2})\to(\mathrm{T}\mathbb{H}^2,\mathbb{G})$ is an harmonic map if and only if $$\overline{\Delta}V+\mathrm{tr}\big((X,Y)\to\mathrm{R}(V,X)Y    \big)=0$$
\end{theorem}
In fact, Theorem \ref{on_harmonic_vector_field} deals with vector fields defined on a general pseudo-Riemannian manifold $M$. Observe that since $(\mathbb{H}^2,g^{\mathbb{H}^2})$ has sectional curvature $-1$, the Riemann tensor $\mathrm{R}$ takes the following simple form:
$$\mathrm{R}(X,Y)Z=g^{\mathbb{H}^2}(X,Z)Y-g^{\mathbb{H}^2}(Y,Z)X.$$ Thus if $e_1$ and $e_2$ is an oriented orthonormal local frame of $g^{\mathbb{H}^2}$, then:

\begin{align*}
\mathrm{tr}\big((X,Y)\to\mathrm{R}(V,X)Y    \big)&=\mathrm{R}(V,e_1)e_1+\mathrm{R}(V,e_2)e_2\\
&=g^{\mathbb{H}^2}(V,e_1)e_1-g^{\mathbb{H}^2}(e_1,e_1)V+g^{\mathbb{H}^2}(V,e_2)e_2-g^{\mathbb{H}^2}(e_2,e_2)V\\
&= (g^{\mathbb{H}^2}(V,e_1)e_1+g^{\mathbb{H}^2}(V,e_2)e_2)-2V\\
&=V-2V\\
&=-V.
\end{align*}
Combining the last computation with Theorem \ref{on_harmonic_vector_field}, we get:

\begin{cor}\label{cor_on_harmonic_vector_field}
A vector field $V:(\mathbb{H}^2,g^{\mathbb{H}^2})\to(\mathrm{T}\mathbb{H}^2,\mathbb{G})$ is an harmonic map if and only if $$\overline{\Delta}V=V.$$
\end{cor}

We now proceed to describe several characterizations of harmonic Lagrangian vector fields. Before that, we need some preparation. Given a vector field $V$ on $\mathbb{H}^2$, we denote by $\alpha^V$ the smooth 1-form on $\mathbb{H}^2$ dual to $V$
$$\alpha^V=g^{\mathbb{H}^2}(V,\cdot).$$
Recall that the \textit{Lie derivative} of $g^{\mathbb{H}^2}$ with respect to $V$ is given by:
$$\mathcal{L}_Vg^{\mathbb{H}^2}(X,Y)=g^{\mathbb{H}^2}(X,\nabla_YV)+g^{\mathbb{H}^2}(Y,\nabla_X V)$$
for any vector fields $X$ and $Y$. 
The next Theorem is the principal result of this section.
\begin{theorem}\label{charac_HL}
    Let $V$ be a vector field on $\mathbb{H}^2$. The following are equivalent:
    \begin{enumerate}
        \item \label{charac_HL1} There exists a smooth function $u:\mathbb{H}^2\to\mathbb{R}$ satisfying  $\Delta^{\mathbb{H}^2} u-2u=0$ such that $$V=\J\grad(u).$$ In other words, $V$ is the vector field associated to the mean surface $S=\mathrm{gr}(u)\subset\HP$ (see Lemma \ref{Shapeoperator_of_graph} and Corollary \ref{cor_formule_XS}).
              \item \label{charac_HL2} $V$ is harmonic Lagrangian.
        \item \label{charac_HL3} The unique self-adjoint $(1, 1)$-tensor $\bb:\mathrm{T}\mathbb{H}^2\to \mathrm{T}\mathbb{H}^2$ such that $\mathcal{L}_Vg^{\mathbb{H}^2}=g^{\mathbb{H}^2}(\bb\cdot,\cdot)$ satisfies the conditions:\begin{equation}\label{condition_harmonic_lagrangian}
    \mathrm{tr}(\bb)=0\ \ \ \ \ \ \mathrm{d}^{\nabla}\bb=0.
\end{equation}
    \end{enumerate}
\end{theorem}
 We recall that, for a $(1,1)-$tensor $\bb$, the \textit{exterior derivative} $\mathrm{d}^{\nabla}\bb$ is defined as
$$\mathrm{d}^{\nabla}\bb(X,Y) = \nabla_X(\bb(Y)) -\nabla_Y(\bb(X))-\bb([X, Y]).$$ 
A tensor satisfying $\mathrm{d}^{\nabla}\bb=0$ is called \textit{Codazzi} tensor with respect to the metric $g^{\mathbb{H}^2}.$
\begin{remark}\label{remark_on_codazzi}
It is worth noting the following points, which we will use freely in the rest of this paper:
\begin{itemize}
    \item For any $u:\mathbb{H}^2\to\mathbb{R}$, $\B=\hess(u)-u\mathbbm{1}$ is a self-adjoint Codazzi tensor with respect to $g^{\mathbb{H}^2}$, (see for instance \cite[Lemma 2.1]{BS_flat_conical}).
    \item Let $f:\mathbb{H}^2\to\mathbb{R}$ be a smooth function. Consider $e_1$ and $e_2$ to be an oriented orthonormal local frame
for $g^{\mathbb{H}^2}$. Then it is not difficult to check that:
$$\mathrm{d}^{\nabla}\left( f\mathbbm{1} \right)(e_1,e_2)=\J\grad(f).$$ In particular, $f$ is constant if and if $f\mathbbm{1}$ is a Codazzi tensor.

\item  A Codazzi tensor $\bb$ is self-adjoint and traceless if and only if $\J\bb$ is self-adjoint, traceless, and Codazzi. 
\end{itemize}
\end{remark}

To prove Theorem \ref{charac_HL}, we need some preliminary results. Given a vector field $V$ on $\mathbb{H}^2$, define a self-adjoint $(1,1)$-tensor $\mathrm{A}$ and a skew-symmetric $(1,1)$-tensor $\Phi$ by:
\begin{equation}\label{Kozul_formulaa}
  \mathcal{L}_Vg^{\mathbb{H}^2}(X,Y)=g^{\mathbb{H}^2}(2\mathrm{A} X,Y) \ \ \ \mathrm{d}\alpha^V(X,Y)=g^{\mathbb{H}^2}(2\Phi X,Y)
\end{equation}
Note that since for each point $p$ in $\mathbb{H}^2$, $\mathrm{dim}(\mathrm{T}\mathbb{H}^2_p)=2$, there exists a smooth function $\phi:\mathbb{H}^2\to\mathbb{R}$ such that \begin{equation}\label{46.phi}
    \Phi X=\phi\mathrm{J}X\end{equation} for any vector field $X.$
\begin{lemma}\label{Kozul_formula}
Let $V$ be a vector field on $\mathbb{H}^2$ and consider the self-adjoint tensor field $\mathrm{A}$ of type $(1,1)$ and the function $\phi:\mathbb{H}^2\to \mathbb{R}$ as in \eqref{Kozul_formulaa} and \eqref{46.phi}. Then
\begin{itemize}
    \item $\nabla_XV=\mathrm{A} X+\phi\J X$.
    \item $\phi=-\frac{1}{2}\mathrm{div}(\J V).$
\end{itemize}
\end{lemma}

\begin{proof}
 By definition we have \begin{equation}\label{liederivatiove}
     \mathcal{L}_Vg^{\mathbb{H}^2}(X,Y)=g^{\mathbb{H}^2}(\nabla_XV,Y)+g^{\mathbb{H}^2}(\nabla_YV,X).
 \end{equation} 
 To prove the first assertion in is enough to remark that by an elementary computation we have:
\begin{equation}\label{diff_1form}
\mathrm{d}\alpha^V(X,Y)=g^{\mathbb{H}^2}(\nabla_XV,Y)-g^{\mathbb{H}^2}(\nabla_YV,X)\end{equation}
 Adding the equations \eqref{liederivatiove} and \eqref{diff_1form} we get the desired formula.
Let us now prove the second assertion. Consider $e_1, e_2$ an oriented orthonormal local frame of $g^{\mathbb{H}^2}$. Then by \eqref{diff_1form}:
\begin{align*}
\mathrm{d}\alpha^V(e_1,e_2)&=g^{\mathbb{H}^2}(\nabla_{e_1}V,e_2)-g^{\mathbb{H}^2}(\nabla_{e_2}V,e_1)\\
&=g^{\mathbb{H}^2}(\nabla_{e_1}V,\J e_1)-g^{\mathbb{H}^2}(\nabla_{e_2}V,-\J e_2)\\
&=g^{\mathbb{H}^2}(-\J\nabla_{e_1}V, e_1)+g^{\mathbb{H}^2}(-\J\nabla_{e_2}V, e_2)\\
&= -\mathrm{div}(\J V).
\end{align*}
\end{proof}
The following Lemma expresses the rough Laplacian of $V$ using the tensors $\mathrm{A}$ and $\Phi$ of \eqref{Kozul_formulaa}. 
\begin{lemma}\label{prepa1_CH_HL}
Let $V$ be a vector field on $\mathbb{H}^2$ and consider the self-adjoint tensor field $\mathrm{A}$ of type $(1,1)$ and a function $\phi:\mathbb{H}^2\to \mathbb{R}$ as in \eqref{Kozul_formulaa} and \eqref{46.phi}:
$$  \mathcal{L}_Vg_{\mathbb{H}^2}(X,Y)=g^{\mathbb{H}^2}(2\mathrm{A} X,Y) \ \ \ \mathrm{d}\alpha^V(X,Y)=g^{\mathbb{H}^2}(2\phi \J X,Y).$$ Let $e_1$, $e_2$ be an oriented orthonormal local frame for $g^{\mathbb{H}^2}$. Then 
     $$\overline{\Delta}V=-\mathrm{d}^{\nabla}(\mathrm{A}\J)(e_1,e_2)+\J\grad(\phi).$$
\end{lemma}
\begin{proof}
    By Lemma \ref{Kozul_formula}, for all tangent vectors $X$,
    $$\nabla_X V = \mathrm{A} X + \phi \J X.$$
    Hence,
    $$\nabla_X \nabla_X V = \nabla_X (\mathrm{A} X) + \mathrm{d}\phi(X) \J X + \phi \nabla_X \J X.$$
    Now, since $\J$ is parallel with respect to $\nabla$ (i.e., $\nabla_X \J Y = \J \nabla_X Y$), we obtain
    \begin{align*}
        \overline{\Delta} V &= \nabla_{e_1} \nabla_{e_1} V - \nabla_{\nabla_{e_1} e_1} V + \nabla_{e_2} \nabla_{e_2} V - \nabla_{\nabla_{e_2} e_2} V \\
        &= \nabla_{e_1} (\mathrm{A} e_1) + \mathrm{d}\phi(e_1) \J e_1 + \phi \nabla_{e_1} \J e_1 - \mathrm{A} \nabla_{e_1} e_1 - \phi \J \nabla_{e_1} e_1 + \\
        & \quad \nabla_{e_2} (\mathrm{A} e_2) + \mathrm{d}\phi(e_2) \J e_2 + \phi \nabla_{e_2} \J e_2 - \mathrm{A} \nabla_{e_2} e_2 - \phi \J \nabla_{e_2} e_2 \\
        &= \nabla_{e_1} (\mathrm{A} e_1) + \nabla_{e_2} (\mathrm{A} e_2) - \mathrm{A} \nabla_{e_1} e_1 - \mathrm{A} \nabla_{e_2} e_2 + \J \grad(\phi).
    \end{align*}
    Finally, using the fact that $e_2 = \J e_1$ and $e_1 = -\J e_2$, we get:
    \begin{align*}
        \overline{\Delta} V &= -\nabla_{e_1} (\mathrm{A} \J e_2) + \nabla_{e_2} (\mathrm{A} \J e_1) + \mathrm{A} \J (\nabla_{e_1} e_2 - \nabla_{e_2} e_1) + \J \grad(\phi) \\
        &= -\nabla_{e_1} (\mathrm{A} \J e_2) + \nabla_{e_2} (\mathrm{A} \J e_1) + \mathrm{A} \J [e_1, e_2] + \J \grad(\phi) \\
        &= -\mathrm{d}^{\nabla}(\mathrm{A} \J)(e_1, e_2) + \J \grad(\phi).
    \end{align*}
    This finishes the proof.
\end{proof}
The next Lemma provides some essential computations in the case where \( V = \J \grad(u) \) for a smooth function \( u : \mathbb{H}^2 \to \mathbb{R} \).
\begin{lemma}\label{Lie_derivative_of_Ju}
Let $u:\mathbb{H}^2\to\mathbb{R}$ be a smooth function and $V=\J\grad(u)$. Consider $\B=\hess(u)-u\mathbbm{1}$ the shape operator of $\mathrm{gr}(u)\subset\HP$, and $\B_0=\B-\frac{\mathrm{tr}(\B)}{2}\mathbbm{1}$ the traceless part of $\B$. Then we have:
\begin{enumerate} 
    \item \label{1Lie_derivative_of_Ju} $\mathcal{L}_Vg^{\mathbb{H}^2}=g^{\mathbb{H}^2}(2\J\B_0\cdot,\cdot)$.
    \item \label{2Lie_derivative_of_Ju} $\mathrm{d}\alpha^V=(\Delta^{\mathbb{H}^2} u) g^{\mathbb{H}^2}(\J\cdot,\cdot)$.
    \item \label{3Lie_derivative_of_Ju} $
    \overline{\Delta}V-V=\J\grad\left( \frac{\mathrm{tr}(\B)}{2} \right).$
\end{enumerate}
\end{lemma}

\begin{proof}
Let $X$ and $Y$ be vector fields on $\mathbb{H}^2$. Then
$$\nabla_XV=\J\hess u(X)=\J \B X+u\J X.$$
Note that $g^{\mathbb{H}^2}(\J X,Y)+g^{\mathbb{H}^2}(\J Y,X)=0$, thus
\begin{align*}
\mathcal{L}_Vg_{\mathbb{H}^2}(X,Y)&=g^{\mathbb{H}^2}(\J\B X,Y)+g^{\mathbb{H}^2}(\J\B Y,X)\\
&=\big(g^{\mathbb{H}^2}(\J\B_0 X,Y)+g^{\mathbb{H}^2}(\J\B_0 Y,X)\big)+ \frac{\mathrm{tr}(\B)}{2}\big(g^{\mathbb{H}^2}(\J X,Y)+g^{\mathbb{H}^2}(\J Y,X)\big)\\
&=g^{\mathbb{H}^2}(\J\B_0 X,Y)+g^{\mathbb{H}^2}(\J\B_0 Y,X). 
\end{align*}
Since $\B_0$ is self-adjoint and traceless, $\J\B_0 $ is also self-adjoint and traceless, thus
\begin{equation}\label{eq_50_rough_for_divergence_free}
    \mathcal{L}_Vg^{\mathbb{H}^2}(X,Y)=g^{\mathbb{H}^2}(2\J\B_0 X,Y).
\end{equation} This completes the proof of \eqref{1Lie_derivative_of_Ju}. \\

Let us prove the second formula (\ref{2Lie_derivative_of_Ju}). By \eqref{Kozul_formulaa} and \eqref{46.phi}, there exists \(\phi : \mathbb{H}^2 \to \mathbb{R}\) such that
\begin{equation}\label{eq.phi_laplace}
     \mathrm{d}\alpha^V = g^{\mathbb{H}^2}(2\phi \J \cdot, \cdot).
\end{equation}
Based on Lemma \ref{Kozul_formula}, we have
\[ \phi = -\frac{1}{2} \mathrm{div}(\J V) = \frac{1}{2} \mathrm{div}(\grad(u)) = \frac{1}{2} \Delta^{\mathbb{H}^2} u. \]
Thus, by \eqref{eq.phi_laplace}, \(\mathrm{d}\alpha^V = (\Delta^{\mathbb{H}^2} u) g^{\mathbb{H}^2}(\J \cdot, \cdot)\). This completes the proof of (\ref{2Lie_derivative_of_Ju}).\\

Finally, we will prove the third formula \eqref{3Lie_derivative_of_Ju}. Consider $e_1$, $e_2$ be an oriented orthonormal local frame for $g^{\mathbb{H}^2}$. Then by Lemma \ref{prepa1_CH_HL} and the second formula \eqref{2Lie_derivative_of_Ju}, we have:
\begin{equation}\label{roughrough}
    \begin{split}
\overline{\Delta}V&=-\mathrm{d}^{\nabla}(\J\B_0\J)(e_1,e_2)+\J\grad\left(\frac{\Delta^{\mathbb{H}^2} u}{2}\right)\\
&=-\mathrm{d}^{\nabla}\B_0(e_1,e_2)+\J\grad\left(\frac{\Delta^{\mathbb{H}^2} u}{2}\right)\\
&= -\mathrm{d}^{\nabla}\B(e_1,e_2)+\mathrm{d}^{\nabla}\left( \frac{\mathrm{tr}(\B)}{2} \right)(e_1,e_2)+  \J\grad\left(\frac{\Delta^{\mathbb{H}^2} u}{2}\right)\\
&=-\mathrm{d}^{\nabla}\B(e_1,e_2)+\mathrm{d}^{\nabla}\left( \frac{\mathrm{tr}(\B)}{2} \right)(e_1,e_2)+  \J\grad\left(\frac{\Delta^{\mathbb{H}^2} u-2u}{2}\right)+  \J\grad u \\
&=-\mathrm{d}^{\nabla}\B(e_1,e_2)+\mathrm{d}^{\nabla}\left( \mathrm{tr}(\B)\right)(e_1,e_2)+  \J\grad u,
\end{split}
\end{equation}
where we have used in the second equality the fact that \(\J\B_0 = -\B_0\J\), and in the last equality, we have used the fact that \(\mathrm{tr}(\B) = \Delta u - 2u\) together with Remark \ref{remark_on_codazzi}. Now, since \(\B\) is a Codazzi tensor and \(\mathrm{d}^{\nabla}\left( \frac{\mathrm{tr}(\B)}{2}\mathbbm{1} \right)(e_1,e_2) = \J\grad\left( \frac{\mathrm{tr}(\B)}{2} \right)\) (see Remark \ref{remark_on_codazzi}), formula \eqref{roughrough} becomes:

\begin{equation}\label{Tension_Ju}
    \overline{\Delta}V-V=\J\grad\left( \frac{\mathrm{tr}(\B)}{2} \right).
\end{equation}
This concludes the proof of \eqref{3Lie_derivative_of_Ju}.

\end{proof}

\begin{proof}[Proof of Theorem \ref{charac_HL}]
The proof is based on the formulas established in Lemma \ref{Lie_derivative_of_Ju}. \\

\textbullet \ (\ref{charac_HL1}) $\implies$ (\ref{charac_HL2}): Assume that \( V = \J \grad(u) \), where \( u \) satisfies \(\Delta u - 2u = 0\). Thus by Theorem \ref{Lagrange_vs_divergence_free}, \( V(\mathbb{H}^2) \) is a Lagrangian surface. It remains to show that \( V \) is harmonic. Lemma \ref{Lie_derivative_of_Ju} implies that
\[\overline{\Delta} V - V = \J \grad\left( \frac{\mathrm{tr}(\B)}{2} \right).\]
But \(\mathrm{tr}(\B) = \Delta u - 2u = 0\) and thus
\[\overline{\Delta} V - V = 0.\]
Hence \( V \) is harmonic by Corollary \ref{cor_on_harmonic_vector_field}. \\

\textbullet \ (\ref{charac_HL2}) $\implies$ (\ref{charac_HL3}): Let $V$ be an harmonic Lagrangian vector field of $\mathbb{H}^2$. Since $V(\mathbb{H}^2)$ is a Lagrangian surface, by Theorem \ref{Lagrange_vs_divergence_free}, there exists a smooth function $u:\mathbb{H}^2 \to \mathbb{R}$ such that $V = \J\grad(u).$ By Lemma \ref{Lie_derivative_of_Ju}, we have
\[ \mathcal{L}_Vg^{\mathbb{H}^2} = g^{\mathbb{H}^2}(2\J\B_0 \cdot, \cdot), \]
where $\B_0$ is, as usual, the traceless part of $\B = \hess(u) - u\mathbbm{1}.$ To prove (\ref{charac_HL3}), we need to show that $\J\B_0$ satisfies the conditions
\[ \mathrm{tr}(\J\B_0) = 0 \ \mathrm{and} \ \mathrm{d}^{\nabla}(\J\B_0) = 0. \]
Since $\B_0$ is a traceless self-adjoint operator, $\mathrm{tr}(\J\B_0) = 0$ follows immediately. To prove that $\J\B_0$ is a Codazzi tensor, we need to use the harmonicity of $V$. Indeed, since $\overline{\Delta}V - V = 0$ (see Corollary \ref{cor_on_harmonic_vector_field}), then by the third formula of Lemma \ref{Lie_derivative_of_Ju}, we get
\[ \J\grad\left( \frac{\mathrm{tr}(\B)}{2} \right) = 0. \]
Hence the trace of $\B$ is constant. In particular, by Remark \ref{remark_on_codazzi}, $\mathrm{tr}(\B)\mathbbm{1}$ is a Codazzi tensor. Therefore,
\begin{align*}
\mathrm{d}^{\nabla}\B_0 &= \mathrm{d}^{\nabla}\B - \mathrm{d}^{\nabla}\left( \frac{\mathrm{tr}(\B)\mathbbm{1}}{2} \right) \\
&= \mathrm{d}^{\nabla}\B.
\end{align*}
Again, by Remark \ref{remark_on_codazzi}, the shape operator $\B$ is a Codazzi tensor and hence $\B_0$ is a Codazzi tensor, and the same holds for $\J\B_0$ (because $\B_0$ is self-adjoint and traceless).\\

\textbullet \ (\ref{charac_HL3}) $\implies$ (\ref{charac_HL1}): Let \( V \) be a vector field on \(\mathbb{H}^2\) for which the unique self-adjoint \((1,1)\)-tensor \(\bb: \mathrm{T}\mathbb{H}^2 \to \mathrm{T}\mathbb{H}^2\) such that \(\mathcal{L}_V g^{\mathbb{H}^2} = g^{\mathbb{H}^2}(\bb \cdot, \cdot)\) satisfies the conditions:
\[
\mathrm{tr}(\bb) = 0 \quad \text{and} \quad \mathrm{d}^{\nabla} \bb = 0.
\]
Let \( e_1 \) and \( e_2 \) be an oriented orthonormal local frame of \( g^{\mathbb{H}^2} \). Then by the definition of divergence operator (see \eqref{formula_divergence}), we have 
\begin{align*}
2 \mathrm{div}(V) &= \mathcal{L}_Vg^{\mathbb{H}^2}(e_1,e_1) + \mathcal{L}_Vg^{\mathbb{H}^2}(e_2,e_2) \\
&= g^{\mathbb{H}^2}(\bb e_1, e_1) + g^{\mathbb{H}^2}(\bb e_2, e_2) \\
&= \mathrm{tr}(\bb) \\
&= 0.
\end{align*}
Hence \( V \) is divergence-free, and so by Theorem \ref{Lagrange_vs_divergence_free}, there is a smooth function \( u:\mathbb{H}^2 \to \mathbb{R} \) such that 
\[ V = \J \grad(u). \]
By Lemma \ref{Lie_derivative_of_Ju}, \(\bb = 2 \J \B_0\). Based on Remark \ref{remark_on_codazzi}, \(\B = \hess(u) - u \mathbbm{1}\) is a Codazzi tensor and so \(\mathrm{tr}(\B)\mathbbm{1}\) is a Codazzi tensor since \(\B_0 = \B - \frac{\mathrm{tr}(\B)}{2} \mathbbm{1}\). Therefore, \(\mathrm{tr}(\B)\) is constant, say \( c \). Since \(\Delta^{\mathbb{H}^2} u - 2u = c\), it follows that \(\Delta^{\mathbb{H}^2}(u + \frac{c}{2}) - 2(u + \frac{c}{2}) = 0\), and so \( V \) is equal to \(\J \grad(u + \frac{c}{2})\). This finishes the proof of \eqref{charac_HL1}. 
\end{proof}

\begin{remark}\label{Killing_are_harmonic}
It follows from Theorem \ref{charac_HL} that a Killing vector field on $\mathbb{H}^2$ is harmonic Lagrangian. Indeed, if $V$ is such a Killing vector field, then there exists $\sigma\in\minko$ such that $V=\Lambda(\sigma)$. In other words, $$V(p)=p\boxtimes\sigma,$$ for all $p\in\mathbb{H}^2$. Now, we consider the map $u:\mathbb{H}^2\to\mathbb{R}$ defined by $u(p)=\inner{p,\sigma}_{1,2}$. This implies that $\grad_p u=\sigma+u(p)p$, hence $V=\J\grad u$. On the other hand, $u$ satisfies $\Delta u-2u=0$, and thus $V$ is harmonic Lagrangian.  
\end{remark}

\subsection{$\overline{\partial}-$operator}\label{Sec_dbar_operator}
In this section, we establish a relationship between the shape operator of a mean surface and the \textit{$\overline{\partial}$-derivative} of the harmonic Lagrangian vector field associated to it. Let $V$ be a vector field on $\mathbb{H}^2$, then for each vector field $Y$ on $\mathbb{H}^2$, we can use the complex structures $\J$ and $\mathbb{J}$ of $\mathbb{H}^2$ and $\mathrm{T}\mathbb{H}^2$ respectively (see Theorem \ref{canonical_pseudo}) to decompose $\mathrm{d}V(Y)$ as
$$\mathrm{d}V(Y)=L_1(Y)+L_2(Y),$$
where \begin{itemize}
    \item $L_1(Y):=\frac{1}{2}(\mathrm{d}V(Y)-\mathbb{J}\mathrm{d}V(\J Y))$ satisfying $$L_1(\J Y)=\mathbb{J}L_1(Y),$$
    \item $L_2(Y):=\frac{1}{2}(\mathrm{d}V(Y)+\mathbb{J}\mathrm{d}V(\J Y))$ satisfying $$L_2(\J Y)=-\mathbb{J}L_2(Y).$$
\end{itemize}
According to Proposition \ref{dV(X)}, we have $\mathrm{d}V(Y)=(Y,\nabla_YV)$, hence $$L_2(Y)=(0,\frac{1}{2}(\nabla_YV+\J\nabla_{\J Y}V)).$$
Following this, we introduce:
\begin{defi}
    Let $V$ be a smooth vector field on $\mathbb{H}^2$. Then we define the self-adjoint $(1,1)$-tensor $\overline{\partial}V$ by:
   \begin{equation}\label{dbar_H2}
   \overline{\partial}V(Y):=\frac{1}{2}(\nabla_YV+\J\nabla_{\J Y}V).\end{equation}
\end{defi}
The next proposition gives a relation between the shape operator of a mean surface and the $\overline{\partial}-$operator.
\begin{lemma}\label{d_bar_and_shape_operator}
  Let $u:\mathbb{H}^2\to\mathbb{R}$ be a smooth function satisfying $\Delta^{\mathbb{H}^2} u-2u=0$ such that $V=\J\grad(u)$ is an harmonic Lagrangian vector field (see Theorem
  \ref{charac_HL}). Then 
$$\overline{\partial}V=\J\B,$$ where $\B=\hess(u)-u\mathbbm{1}$ is the shape operator of $\mathrm{gr}(u)\subset\HP.$ 
\end{lemma}
\begin{proof}
Let $Y$ be a vector field on $\mathbb{H}^2$ then 
\begin{equation}\label{eq1_Prop.3.19}
    \nabla_YV=\J\hess u(X)=\J\B Y+u\J Y.
\end{equation}
Using the fact that $\J \B=-\B\J$, we get:
\begin{align}\label{eq2_Prop.3.19}
\begin{split}
\J\nabla_{\J Y}V&=\J(\J \B \J Y-uY)\\
&=-\B \J Y-u\J Y\\
&=\J \B Y-u\J Y.
\end{split}
\end{align}
Adding \eqref{eq1_Prop.3.19} and \eqref{eq2_Prop.3.19}, we get $ \overline{\partial}V(Y)=\J \B Y$, as required. \end{proof}
A quantity that will be discussed in the remainder of this paper is the norm of the $\overline{\partial}$-operator. Recall that the norm of a $(1,1)$-tensor $\bb: \mathrm{T}\mathbb{H}^2 \to \mathrm{T}\mathbb{H}^2$ is defined as follows. For each $p \in \mathbb{H}^2$, we denote by $\bb_p$ the induced linear map on $\mathrm{T}_p\mathbb{H}^2$. Then the \textit{norm} of $\bb_p$ is given by:
\begin{equation}\label{norme}
\lVert \bb_p \rVert := \sup_{\underset{Y\neq 0}{Y\in\mathrm{T}_p\mathbb{H}^2}} \frac{\lVert \bb_p(Y) \rVert_{\mathbb{H}^2}}{\lVert Y \rVert_{\mathbb{H}^2}},
\end{equation}
where $\lVert \cdot \rVert_{\mathbb{H}^2}$ denotes the norm induced from the hyperbolic metric $g^{\mathbb{H}^2}$.

From this, we define the \textit{norm} of $\bb$ by:
\[
\lVert \bb \rVert_{\infty} := \sup_{p \in \mathbb{H}^2} \lVert \bb_p \rVert \in [0, +\infty].
\]

\begin{remark}\label{principal_curvature}
It is worth noticing that if $V=\J\grad(u)$ is harmonic Lagrangian vector field. then 
$$\lVert \overline{\partial}V\rVert_{\infty}=\sup_{p\in\mathbb{H}^2}\lvert\lambda(p)\rvert,$$
where $\lambda$ and $-\lambda$ are the principal curvatures of the mean surface $\mathrm{gr}(u)\subset\HP$, namely, the eigenvalues of the shape operator $\B=\hess-u\mathbbm{1}$. This immediately follows from Lemma \ref{d_bar_and_shape_operator}.
\end{remark}

Consider now the Poincaré disk $\mathbb{B}^2$. This is a conformal model of the hyperbolic plane. It is the unit disk endowed with the conformal metric $g^{\mathbb{B}^2}$ given by:
$$g^{\mathbb{B}^2}_{z}:=\frac{4}{(1-\lvert z\rvert^2)^2}\cdot g^{E}$$
where $g^E$  is the standard Euclidean metric tensor on $\mathbb{R}^2$. Notice $\mathbb{B}^2$ and $\mathbb{D}^2$ are the same space which is the unit disk, however we used different notations to distinguishe between the Klein model and the Poincaré model.

A vector field on $\mathbb{B}^2$ can be seen as a map from $\mathbb{B}^2$ to $\mathbb{R}^2\cong\mathbb{C}$, and the complex structure on $\mathbb{B}^2$ is just the multiplication multiplication by $i=\sqrt{-1}$. The goal at the end of this section is to explain the relation between the $\overline{\partial}$-derivative of a vector field given in \eqref{dbar_H2} and the usual complex derivative with respect to $\overline{z}$, for which this quantity appears in \cite{Extension_with_bounded_derivative}; see also \cite{FanJun}.

Let $\nabla^{\mathbb{B}^2}$ be the Levi-Civita connection of $g^{\mathbb{B}^2}$ and denote by $D$ the flat connection of $\mathbb{R}^2$. Let $f:\mathbb{B}^2\to\mathbb{R}$ be such that $g^{\mathbb{B}^2}=e^{2f}g^E$. The following formula relates the Levi-Civita  $\nabla^{\mathbb{B}^2}$ and $D$:
\begin{equation}\label{Levi_Civita_conformal}
    \nabla^{\mathbb{B}^2}_XY=D_XY+\mathrm{d}f(X)Y+\mathrm{d}f(Y)X-g^E(X,Y)\mathrm{grad}^E(f)
\end{equation}
where $\mathrm{grad}^E(f)$ is the euclidean gradient of $f$. Let $V:\mathbb{B}^2\to\mathbb{C}$ be a smooth vector field on $\mathbb{B}^2$, then we may write the $\overline{\partial}$-operator in the Poincaré $\mathbb{B}^2$ as :
\begin{equation}\label{dbar_B2}
    \overline{\partial}_{\mathbb{B}^2}V(Y)=\frac{1}{2}(\nabla^{\mathbb{B}^2}_YV+i\nabla^{\mathbb{B}^2}_{iY}V).
\end{equation}
Using \eqref{Levi_Civita_conformal} we obtain:
\begin{equation}\label{eq_53}
    \overline{\partial}_{\mathbb{B}^2}V(Y)=\frac{1}{2}(D_YV+iD_{iY}V)+\frac{1}{2}(\mathrm{d}f(X)+i\mathrm{d}f(iX))V-\frac{1}{2}(g^{E}(X,V)+ig^E(iX,V))\mathrm{grad}^E(f).
\end{equation}
Denote by $(\frac{\partial}{\partial x}, \frac{\partial}{\partial y})$ the standard basis in $\mathbb{R}^2$ and $(\mathrm{d}x, \mathrm{dy})$ the dual basis. We now make the following standard notation:
$$\frac{\partial}{\partial\overline{z}}:=\frac{1}{2}\left(\frac{\partial}{\partial x}+i\frac{\partial}{\partial y}\right), \ \ \mathrm{d}\overline{z}:=\mathrm{d}x-i\mathrm{d}y.$$
By an elementary computation, we may check the following \begin{enumerate}
    \item $\frac{1}{2}(D_{\frac{\partial}{\partial x}}V+iD_{i\frac{\partial}{\partial x}}V)=\frac{\partial V}{\partial \overline{z}}$ and $\frac{1}{2}(D_{\frac{\partial}{\partial y}}V+iD_{i\frac{\partial}{\partial y}}V)=-i\frac{\partial V}{\partial \overline{z}}$.
    \item $\mathrm{d}f(\frac{\partial}{\partial x})+i\mathrm{d}f(i\frac{\partial}{\partial x})=\mathrm{grad}^E(f)$ and $\mathrm{d}f(\frac{\partial}{\partial y})+i\mathrm{d}f(i\frac{\partial}{\partial y}))=-i\mathrm{grad}^E(f)$.
    \item $g^{E}(\frac{\partial}{\partial x},V)+ig^E(i\frac{\partial}{\partial x},V)=V$ and $g^{E}(\frac{\partial}{\partial y},V)+ig^E(i\frac{\partial}{\partial y},V)=-iV.$
\end{enumerate}
Applying the last formulas into \eqref{eq_53} yields:
\begin{equation}\label{dbar_relation0}\overline{\partial}_{\mathbb{B}^2} V(Y)=\frac{\partial V}{\partial \overline{z}}\mathrm{d}\overline{z}(Y).\end{equation} In conclusion, we have the following Lemma.

\begin{lemma}\label{same_d_bar}
    Let $V$ be a smooth vector field on $\mathbb{B}^2$. Then:
    \begin{equation}\label{dbar_relation}
        \overline{\partial}_{\mathbb{B}^2} V = \frac{\partial V}{\partial \overline{z}} \, \mathrm{d}\overline{z}.
    \end{equation}
    Moreover, for each $z \in \mathbb{B}^2$, the norm of $\overline{\partial}V_z$ in the sense of \eqref{norme} is given by:
    \begin{equation}\label{norme2}
        \lVert \overline{\partial}_{\mathbb{B}^2} V_z \rVert = \left| \frac{\partial V}{\partial \overline{z}}(z) \right|.
    \end{equation}
\end{lemma}

\begin{proof}
The first formula \eqref{dbar_relation} follows immediately from \eqref{dbar_relation0}. To prove \eqref{norme2}, it is enough to observe that the norm in the sense of \eqref{norme} is invariant under conformal changes of the metric. Indeed, since $g^{\mathbb{B}^2} = e^{2f} g^{E}$, then for $z \in \mathbb{B}^2$ and $Y \in \mathrm{T}_z \mathbb{B}^2 \cong \mathbb{C}$,
\begin{align*}
\lVert \overline{\partial}_{\mathbb{B}^2}V_z \rVert &= \sup_{\underset{Y\neq 0}{Y\in\mathrm{T}_p\mathbb{B}^2}} \frac{\lVert \overline{\partial}V_z(Y) \rVert_{\mathbb{B}^2}}{\lVert Y \rVert_{\mathbb{B}^2}} \\
&= \sup_{\underset{Y\neq 0}{Y\in\mathrm{T}_p\mathbb{B}^2}}\frac{e^{f(z)} \lvert \overline{\partial}_{\mathbb{B}^2}V_z(Y) \rvert}{e^{f(z)} \lvert Y \rvert} \\
&= \sup_{\underset{Y\neq 0}{Y\in\mathrm{T}_p\mathbb{B}^2}} \frac{\lvert \overline{\partial}_{\mathbb{B}^2}V_z(Y) \rvert}{\lvert Y \rvert}.
\end{align*}
The proof is completed by applying \eqref{dbar_relation0}.
\end{proof}
Using Lemma \ref{same_d_bar}, we may deduce the following Corollary.
\begin{cor}\label{d_bar_isometry}
    Let \( V \) be a smooth vector field on \(\mathbb{H}^2\). Consider \( G:\mathbb{H}^2 \to \mathbb{B}^2 \) an isometry between the hyperboloid model and the Poincaré disk model of the hyperbolic space. Then 
    \[
    \lVert \overline{\partial} V_p \rVert = \Bigg| \frac{\partial (G_*V)}{\partial \overline{z}}(G(p)) \Bigg|.
    \] 
\end{cor}

\section{Extension by harmonic Lagrangian vector field}\label{sec6_extension}
The goal of this section is to prove Theorem \ref{intro_ex_HL}, which we will restate here.
\begin{theorem}\label{THHL_Existence}
Let $X$ be a continuous vector field on $\mathbb{S}^1$. Then there exists an harmonic Lagrangian vector field on $\mathbb{H}^2$ that extends continuously to $X$ on $\mathbb{S}^1$.
\end{theorem}
We will give two completely different proofs of Theorem \ref{THHL_Existence}. The first proof is based on the study of the so-called \textit{infinitesimal Douady-Earle extension}. The second proof is more in the spirit of Half-Pipe geometry.

\subsection{Infinitesimal Douady-Earle extension}\label{subsec_inf_DEE}
In this section, we recall the construction of \textit{infinitesimal Douady-Earle extension} and then state the main results needed to prove Theorem \ref{THHL_Existence}. We denote by $\Gamma(\mathbb{S}^1)$ and $\Gamma(\mathbb{H}^2)$ the vector spaces of continuous vector fields on $\mathbb{S}^1$ and $\mathbb{H}^2$, respectively. These two spaces are equipped with the compact-open topology, meaning that a sequence $V_n$ in $\Gamma(\mathbb{S}^1)$ (resp. $\Gamma(\mathbb{H}^2)$) converges to $V$ in $\Gamma(\mathbb{S}^1)$ (resp. $\Gamma(\mathbb{H}^2)$) if $V_n$ converges uniformly to $V$ on $\mathbb{S}^1$ (resp. on compact subsets of $\mathbb{H}^2$).

The full isometry group of $\mathbb{H}^2$, denoted by $\mathrm{Isom}(\mathbb{H}^2)$, acts on $\Gamma(\mathbb{S}^1)$ and $\Gamma(\mathbb{H}^2)$ by pushforward. That is,
$$A_* X(p)=\mathrm{d}_{A^{-1}\cdot p}A(X(A^{-1}\cdot p)),$$ for $A \in \mathrm{Isom}(\mathbb{H}^2)$ and $V \in \Gamma(\mathbb{S}^1)$ or $\Gamma(\mathbb{H}^2)$. We say that a linear map $L:\Gamma(\mathbb{S}^1) \to \Gamma(\mathbb{H}^2)$ is \textit{conformally natural} if for all $V \in \Gamma(\mathbb{S}^1)$ and $A \in \mathrm{Isom}(\mathbb{H}^2)$, we have:
$$L(A_*V)=A_*L(V).$$ 

One such map, $L_0:\Gamma(\mathbb{S}^1) \to \Gamma(\mathbb{H}^2)$, is the \textit{infinitesimal Douady-Earle} extension, which can be defined on the Poincaré model $\mathbb{B}^2$ as follows: Consider a vector field $X$ on $\mathbb{S}^1$ and define
\begin{equation}\label{inf_DE}
    L_0(X)(z)=\frac{(1-\vert z\vert^2)^3}{2i\pi}\int_{\mathbb{S}^1}\frac{X(x)}{(1-\overline{z}x)(x-z)}dx,
\end{equation}
for all $z\in\mathbb{B}^2$. It is explained in \cite{Uniqueness_of_operator_L} that the operator $L_0$ can be seen as an infinitesimal version of the Douady-Earle extension operator. More precisely, let $f_t: \mathbb{S}^1 \to \mathbb{S}^1$ be a smooth family of diffeomorphisms of $\mathbb{S}^1$ and let $X = \frac{d}{dt}\big|_{t=0} f_t$  be a tangent vector field on $\mathbb{S}^1$. Set $\mathrm{DE}(f_t): \mathbb{H}^2 \to \mathbb{H}^2$ to be the Douady-Earle extension of $f_t$ (see \cite{Douady_Earle}). Then we have
\begin{equation}\label{derivative_douady_earle}
    L_0(X)=\frac{d}{dt}\bigg\lvert_{ t=0} \mathrm{DE}(f_t). 
\end{equation}

\begin{remark}\label{douadyearly_killing}
Let $X \in \Gamma(\mathbb{S}^1)$ be the restriction on $\mathbb{S}^1$ of a Killing vector field of $\mathbb{H}^2$, which we also denote by $X$. Then $L_0(X)=X$. Indeed, let $f_t:\mathbb{S}^1\to\mathbb{S}^1$ be the flow of $X$ so that $X=\frac{d}{dt}\big\lvert_{t=0}f_t$. Consider $F_t: \mathbb{H}^2 \to \mathbb{H}^2$ to be the one-parameter family of isometries of $\mathbb{H}^2$ that extends to $f_t$ in $\mathbb{S}^1$. Then it is known that $\mathrm{DE}(f_t) = F_t$ and hence by \eqref{derivative_douady_earle}, $L_0(X)=X.$
\end{remark}
Now we state the principal result of \cite{Uniqueness_of_operator_L}:

\begin{theorem}\cite{Uniqueness_of_operator_L}\label{Uniqueness}
Let $L:\Gamma(\mathbb{S}^1) \to \Gamma(\mathbb{B}^2)$ be a continuous linear map that is conformally natural. Then $L$ is a multiple of the infinitesimal Douady-Earle extension $L_0$ given in \eqref{inf_DE}.
\end{theorem}
From Theorem \ref{Uniqueness}, we can prove that $L_0(X)$ is an harmonic Lagrangian vector field on $\mathbb{H}^2.$
\begin{prop}\label{L_0_harmonic_lagrangian}
    Let $X$ be a continuous vector field on $\mathbb{S}^1$. Then $L_0(X)$ is an harmonic Lagrangian vector field on $\mathbb{B}^2$.
\end{prop}

\begin{proof}
By the definition of an harmonic Lagrangian vector field, Corollary \ref{cor_on_harmonic_vector_field}, and Theorem \ref{Lagrange_vs_divergence_free}, we need to show that \begin{equation}
    \overline{\Delta}(L_0(X))=L_0(X) \ \ \mathrm{and} \ \ \mathrm{div}(L_0(X))=0.
\end{equation}
Consider the continuous linear operator given by
$$\begin{array}{ccccl}
 L_1& : & \Gamma(\mathbb{S}^1) & \to & \Gamma(\mathbb{B}^2) \\
 & & X & \mapsto &  \overline{\Delta}(L_0(X)) \\
\end{array}$$
It is not difficult to check that $L_1$ is conformally natural, and hence by Theorem \ref{Uniqueness}, there is $\lambda\in\mathbb{R}$ such that 
\begin{equation}\label{eq_harmo_L0}
    L_1(X)=\lambda L_0(X),
\end{equation} for all vector field $X$. 
According to Remark \ref{douadyearly_killing}, if $Y$ is the restriction on $\mathbb{S}^1$ of a Killing vector field of $\mathbb{H}^2$, then $L_0(Y)=Y$. On the other hand, a Killing vector field is harmonic (see Remark \ref{Killing_are_harmonic}), thus 
$$L_1(Y)=\overline{\Delta}(L_0(Y))=L_0(Y)$$ and hence $\lambda=1$. This proves the harmonicity of $L_0(X)$.

Next, we prove that $L_0(X)$ is a divergence-free vector field on $\mathbb{B}^2$. For this, we consider the following continuous linear map:
$$\begin{array}{ccccl}
 L_2& : & \Gamma(\mathbb{S}^1) & \to & \Gamma(\mathbb{H}^2) \\
 & & X & \mapsto &  \mathrm{div}(L_0(X))L_0(X) \\
\end{array}$$
Again, $L_2$ is conformally natural and so by Theorem \ref{Uniqueness}, $L_2=\mu L_0$ for some $\mu\in\mathbb{R}$. Since Killing vector fields are divergence-free, $L_2$ sends Killing vector fields to $0$, hence $\mu=0$. We conclude that $L_2=0$, equivalently \begin{equation}\label{function_divergence}
    \mathrm{div}(L_0(X))L_0(X)=0,
\end{equation} for all $X\in\Gamma(\mathbb{S}^1)$. We claim that $\mathrm{div}(L_0(X))=0$. By contradiction, if there is $z\in\mathbb{B}^2$ such that $\mathrm{div}(L_0(X))(z)\neq0$, then by \eqref{function_divergence} we must have $L_0(X)(z)=0$. Now observe that if $V$ is a vector field on $\mathbb{B}^2$ and $\phi$ is a smooth scalar function on $\mathbb{B}^2$, then by the divergence formula \eqref{formula_divergence}, we have:
\begin{equation}\label{function_divergence_formula}
    \mathrm{div}(\phi V)=\phi\mathrm{div}(V)+\mathrm{d}\phi(V).
\end{equation}
Using \eqref{function_divergence} and applying formula \eqref{function_divergence_formula} for $\phi=\mathrm{div}(L_0(X))$ and $V=L_0(X),$ we get
\begin{equation}\label{eq_f2}
\phi^2+\mathrm{d}\phi(V)=0.\end{equation} Since $V(z)=L_0(X)(z)=0$, then \eqref{eq_f2} implies $\phi(z)=0$, this is a contradiction with our hypothesis $\phi(z)=\mathrm{div}(L_0(X))(z)\neq 0$. This completes the proof.
\end{proof}
The remaining result needed to prove Theorem \ref{THHL_Existence} is the following Theorem due to Reich and Chen.
\begin{theorem}\cite{Extension_with_bounded_derivative}\label{extension_Reich}
Let $X$ be a continuous vector field on $\mathbb{S}^1$. Then $L_0(X)$ defines a continuous extension of $X$ from $\mathbb{S}^1$ to $\mathbb{B}^2$.
\end{theorem}
We now have all the tools to prove Theorem \ref{THHL_Existence}.
\begin{proof}[First Proof of Theorem \ref{THHL_Existence}]
Let $X$ be a continuous vector field on $\mathbb{S}^1$. Then by Proposition \ref{L_0_harmonic_lagrangian}, $L_0(X)$ is harmonic Lagrangian, and the extension of $L_0(X)$ to $X$ follows from Theorem \ref{extension_Reich}. This concludes the proof.
\end{proof}

\subsection{Extension in terms of mean surface in $\HP$}\label{section_5.2_extension_mean_surface}
The goal of this section, which can be read independently from Section \ref{subsec_inf_DEE}, is to give a second proof of Theorem \ref{THHL_Existence}. This proof does not use any knowledge on infinitesimal Douady-Earle extension. Furthermore, it provides some necessary tools to prove the uniqueness result in Theorem \ref{intro_uniqu_HL}.

\subsubsection{Infinitesimal earthquake}\label{appendix6.1}
In this part, we will recall the construction of particular vector fields on $\mathbb{H}^2$. These vector fields will be used in our proof of Theorem \ref{THHL_Existence}. We define the \textit{left infinitesimal earthquake}:
\begin{equation}\label{infearth_convexhull}\begin{array}{ccccc}
\mathcal{E}_{X}^- & : & \mathbb{D}^2 & \to & \mathbb{R}^2 \\
 & & \eta & \mapsto & \mathrm{d}_{(1,\eta)}\Pi\left((1,\eta)\boxtimes  \sigma\right), \\
\end{array}\end{equation}where $\sigma \in \mathbb{R}^{2,1}$ is a point for which the dual spacelike plane $\mathrm{P}_{\sigma}\subset\HP$ is a \textit{support plane} of $\mathrm{gr}(\phi_X^-)$ at $(\eta, \phi_X^-(\eta))$. That is, $(\eta, \phi_X^-(\eta)) \in \mathrm{P}_{\sigma}$ and
$$\inner{(1,x),\sigma}_{1,2}\leq\phi_X^-(x),$$ for $x\in\mathbb{D}^2.$ Notice that one may similarly define the right infinitesimal earthquake by taking support planes of $\mathrm{gr}(\phi_X^+)$ instead of $\mathrm{gr}(\phi_X^-)$, we refer the reader to \cite{diaf2023Infearth} for more details about this construction.
\begin{prop}\cite[Proposition 4.6]{diaf2023Infearth}\label{infinitesimal_earthquake_thm}
 Let $X$ be a continuous vector field on $\mathbb{S}^1$ then $\mathcal{E}_X^{-}$ extends continuously to $X.$   
\end{prop}
\begin{remark}
The proof of the extension of the left infinitesimal earthquake $\mathcal{E}_X^-$ is based on convexity arguments. A similar approach can be used to show the following fact: let $u:\mathbb{H}^2\to\mathbb{R}$ be a convex function, i.e., the shape operator of the surface $S=\mathrm{gr}(u)$ is positive definite. Then the vector field $V_S=\J\grad(u)$ associated to $S$ extends continuously to $X$.
\end{remark}
\subsubsection{Analysis of the mean surface equation}
The key step to proving Theorem \ref{THHL_Existence} is the following gradient estimate proved by Li and Yau.

\begin{theorem}\cite[Theorem 1.3]{Li_Yau}\label{Li_Yau}
Let $M$ be a complete Riemannian manifold of dimension $n$ without boundary. Suppose $f(x, t)$ is a positive solution on $M \times (0, +\infty)$ of the equation:
$$\Delta f-qf-\partial_tf=0,$$
where $q$ is a smooth function defined on $M \times (0, +\infty)$, $\Delta$ is the Laplacian operator on $M \times (0, +\infty)$, and the subscript $t$ denotes the partial differentiation with respect to the $t$ variable. Assume that the Ricci curvature of $M$ is bounded from below by $-K$ for some constant $K > 0$. We also assume that there exist constants $A$ and $B$ such that:
$$\lVert \mathrm{grad}\ q\rVert_{M\times(0,+\infty)}\leq A\ \ \Delta q\leq B.$$ Then there is a constant $C$ such that for any given $\alpha \in (1, 2)$, $f$ satisfies the estimate:
$$\frac{\lVert \mathrm{grad} f\rVert_{M\times(0,+\infty)}^2}{f^2}-\frac{\partial_tf}{f}\leq \alpha q+\frac{n\alpha^2}{2t}+C(B+\frac{K}{\alpha-1}+\sqrt{A}),$$ where $\lVert\cdot\rVert_{M \times (0, +\infty)}$ denotes the norm induced from the Riemannian product metric on $M \times (0, +\infty).$
\end{theorem}
\begin{cor}\label{Li_Yau_Cor}
Let $u: \mathbb{H}^2 \to \mathbb{R}$ be a positive solution of the equation $\Delta^{\mathbb{H}^2} u - 2u = 0$. Then there exists a constant $C_0$ such that 
$$\frac{\lVert \grad u\rVert_{\mathbb{H}^2}}{u} \leq C_0.$$
\end{cor}

\begin{proof}
Let $u$ be a positive solution of $\Delta^{\mathbb{H}^2} u - 2u = 0$. Applying Theorem \ref{Li_Yau} to this time-independent solution, we arrive at the estimate: 
$$\frac{\lVert \grad u\rVert_{\mathbb{H}^2}^2}{u^2} \leq 2\alpha + \frac{\alpha^2}{t} + C\frac{1}{\alpha-1}$$ 
for all $\alpha \in (1,2)$ and $t \in (0, +\infty)$. Letting $t \to +\infty$ and taking $\alpha = \frac{3}{2}$, we get the desired estimate with $C_0 = \sqrt{3 + 2C}$. This completes the proof.
\end{proof}

Now we turn our attention to the PDE satisfied by the function $\overline{u}:\mathbb{D}^2 \to \mathbb{R}$ for which, according to Proposition \ref{BF_platau_HP}, we have:
\begin{equation}
    \begin{cases}\label{platau_HPE2}
        \Delta^{\mathbb{E}^2}\overline{u}(\eta) - \mathrm{Hess}_{\eta}^{\mathbb{E}^2}(\eta,\eta) = 0, \\
        \overline{u}|_{\mathbb{S}^1} = \phi.
    \end{cases}
\end{equation}
In other words, if $\eta = (x,y) \in \mathbb{D}^2$, then \eqref{platau_HPE2} is equivalent to:
\begin{equation}
    \begin{cases}
        (1-x^2)\partial_{xx}\overline{u} + (1-y^2)\partial_{yy}\overline{u} - 2xy\partial_{xy}\overline{u} = 0, \\
        \overline{u}|_{\mathbb{S}^1} = \phi.
    \end{cases}
\end{equation}
Therefore, dividing the equation $\Delta^{\mathbb{E}^2}\overline{u}(\eta) - \mathrm{Hess}_{\eta}^{\mathbb{E}^2}(\eta,\eta) = 0$ by $(1-\lvert\eta\rvert^2)$, we obtain a strictly elliptic equation for which we have the following \textit{strong maximum principle}.

\begin{theorem}\cite{Protter1967MaximumPI}\label{Strong_Max_P}
    Let $\Omega$ be an open connected set in $\mathbb{R}^n$ and let $\mathcal{L}$ be the second-order differential operator: 
    $$
    \mathcal{D} = \sum_{i,j=1}^{n}a_{ij}(x)\partial_{x_ix_j} + \sum_{i=1}^nb_i(x)\partial_{x_i},
    $$
    with $a_{ij} = a_{ji}.$ Assume the following:
    \begin{itemize}
        \item The coefficients $a_{ij}$ and $b_i$ are locally bounded on $\Omega$.
        \item The operator $\mathcal{D}$ is \textit{strictly elliptic} on $\Omega$. Namely, there is $\lambda > 0$ such that
        $$
        \sum_{i,j=1}^{n} x_i x_j a_{ij}(x) \geq \lambda \lVert x \rVert^2, \quad \text{for all } x \in \mathbb{R}^n, \ x \in \Omega.
        $$
    \end{itemize}
    Let $f$ be a smooth function satisfying the differential inequality $\mathcal{D}(f) \geq 0$. If $f$ attains a maximum $M$ at a point in $\Omega$, then $f \equiv M$ in $\Omega$.
\end{theorem}

As a consequence, we get the following.

\begin{lemma}\label{mean_up_convex_core}
Let $\phi: \mathbb{S}^1 \to \mathbb{R}$ be a continuous function and $\overline{u}: \mathbb{D}^2 \to \mathbb{R}$ be a solution of \eqref{platau_HPE2}. We furthermore assume that $\overline{u}$ is not the restriction of an affine function. Let $\sigma \in \minko$ such that $\mathrm{P}_{\sigma}$ is a support plane of $\mathrm{gr}^+(\phi^-)$. Then for all $\eta \in \mathbb{D}^2$, we have the strict inequality:
$$\langle (1,\eta), \sigma \rangle_{1,2} < \overline{u}(\eta).$$
\end{lemma}

\begin{proof}
Let $a: \mathbb{D}^2 \to \mathbb{R}$ be the affine function given by 
$$a(\eta) = \langle (1,\eta), \sigma \rangle_{1,2},$$
and assume by contradiction the existence of $\eta_0$ such that $\overline{u}(\eta_0) = a(\eta_0)$. Let $\overline{h} = a - \overline{u}$. Then clearly we have $\Delta^{\mathbb{E}^2}\overline{h}(\eta) - \mathrm{Hess}_{\eta}^{\mathbb{E}^2}\overline{h}(\eta,\eta) = 0$.\\
Consider now $\mathcal{D}$ to be the second-order differential operator:
\begin{equation}\label{strictelliptic}
\mathcal{D} = \frac{1-x^2}{1-x^2-y^2}\partial_{xx} + \frac{1-y^2}{1-x^2-y^2}\partial_{yy} - \frac{xy}{1-x^2-y^2}\partial_{xy} - \frac{xy}{1-x^2-y^2}\partial_{yx}.
\end{equation}
The differential operator $\mathcal{D}$ clearly satisfies the hypothesis of Theorem \ref{Strong_Max_P}. Observe that $\mathcal{D}(\overline{h}) = 0$; indeed, if $\eta = (x,y)$, then we get $\mathcal{D}(\overline{h})$ by dividing the equation $\Delta^{\mathbb{E}^2}\overline{h}(\eta) - \mathrm{Hess}_{\eta}^{\mathbb{E}^2}\overline{h}(\eta,\eta)$ by $(1-x^2-y^2)$.

To arrive at a contradiction, observe that $\overline{h} \leq 0$ because $\mathrm{P}_{\sigma}$ is a support plane of $\mathrm{gr}(\phi^-)$ and $\phi^- \leq \overline{u}$ (see Proposition \ref{BF_platau_HP}). On the other hand, $\overline{h}(\eta_0) = 0$ and so $\overline{h}$ achieves its maximum at $\eta_0 \in \mathbb{D}^2$. Thus, according to Theorem \ref{Strong_Max_P}, $\overline{h}$ is constantly equal to $0$, and hence $\overline{u}$ is the restriction of an affine function, which is a contradiction.
\end{proof}
The second proof of Theorem \ref{THHL_Existence} follows from this proposition.
\begin{prop}\label{second_proof}
    Let $X$ be a continuous vector field on $\mathbb{S}^1$ and let $\phi_X:\mathbb{S}^1\to\mathbb{R}$ be its support function. Let $u_X:\mathbb{H}^2\to\mathbb{R}$ be the unique function such that \begin{equation}\label{platau_HP_proof}
\Delta^{\mathbb{H}^2} u_X-2u_X=0 ,\ \ \  \overline{u}|_{\mathbb{S}^1}=\phi_X. 
\end{equation}
Then, $\mathrm{HL}(X):=\J\grad(u_X)$ is an harmonic Lagrangian vector field which extends continuously to $X.$
\end{prop}
\begin{proof}
First, observe that by Theorem \ref{charac_HL}, $\mathrm{HL}(X)$ is an harmonic Lagrangian vector field. To prove the extension, we shall make computations in the Klein model $\mathbb{D}^2$. Consider $\overline{\mathrm{HL}(X)}=(\Pi)_*\mathrm{HL}(X)$, the vector field $\mathrm{HL}(X)$ written in $\mathbb{D}^2$, namely for $\eta\in\mathbb{D}^2$ we have:
\begin{equation}\label{pullback_of_vector_field_half_computation}
\overline{\mathrm{HL}(X)}(\eta)=\mathrm{d}_{\Pi^{-1}(\eta)}\Pi\left( \mathrm{HL}(X)(\Pi^{-1}(\eta)) \right)=\mathrm{d}_{\Pi^{-1}(\eta)}\Pi\left( \Pi^{-1}(\eta)\boxtimes\grad_{\Pi^{-1}(\eta)} u_X \right).
\end{equation}
It is immediate to check that for $\eta\in \mathbb{D}^2$ and $v$ a tangent vector at $\Pi(\eta)$, we have:
\begin{equation}\label{Homogenity_of_radial_projection}
\mathrm{d}_{\Pi^{-1}(\eta)}\Pi(v)=\sqrt{1-\lvert\eta\rvert^2}\mathrm{d}_{(1,\eta)}\Pi(v).\end{equation}
Thus, combining \eqref{pullback_of_vector_field_half_computation} and \eqref{Homogenity_of_radial_projection}, we obtain
\begin{equation}\label{pullback_of_vector_field}
\overline{\mathrm{HL}(X)}(\eta)=\mathrm{d}_{(1,\eta)}\Pi\left( (1,\eta)\boxtimes\grad_{\Pi^{-1}(\eta)} u_X \right).
\end{equation}
Let $\eta_n=(x_n,y_n)$ be a sequence converging to $\eta_{\infty}=(x_{\infty},y_{\infty})$. The goal is to show that:
\begin{equation}\label{goal}
    \overline{\mathrm{HL}(X)}(\eta_n)\to X(\eta_{\infty}).
\end{equation}
Let $\mathrm{P}_{s_n}$ be a sequence of spacelike support planes of $\mathrm{gr}(\phi_X^-)$ at $(\eta_n,\phi_X^-(\eta_n))$ (see \eqref{formule_dual_plan} for the formula of $\mathrm{P}_{s_n}$). Then consider  the function $h_n$ defined on $\mathbb{H}^2$ by $$h_n(q)=u_X(q)-\inner{q,s_n}_{1,2}.$$ Similarly, we define $\overline{h_n}:\mathbb{D}^2\to\mathbb{R}$ by \begin{equation}\label{eq_h_n_positive}
    \overline{h_n}(\eta)=\overline{u_X}(\eta)-\inner{(1,\eta),s_n}_{1,2}.
\end{equation} We claim that 
\begin{equation}\label{claiminextension}
    \mathrm{d}_{(1,\eta_n)}\Pi\left( (1,\eta_n)\boxtimes\grad_{\Pi^{-1}(\eta_n)} h_n \right)\to 0.
\end{equation}
To prove the claim, consider $$\sigma_n=\grad_{\Pi^{-1}(\eta_n)}h_n-h_n(\Pi^{-1}(\eta_n))\Pi^{-1}(\eta_n)$$
so that spacelike plane $\mathrm{P}_{\sigma_n}$ is the tangent plane of $\mathrm{gr}(\overline{h_n})$ at $(\eta_n,\overline{h_n}(\eta_n))$ (see Lemma \ref{sigma_lemma}). In particular \begin{equation}\label{u(pi2)}
   \overline{h_n}(\eta_n)= \inner{(1,\eta_n),\sigma_n}_{1,2}.
\end{equation}
Let $p = (1, 0, 0)$, $r_n = \sqrt{x_n^2 + y_n^2}$, and $v_n = \frac{1}{r_n}(-y_n, x_n)$. Consider $w_n = \frac{1}{r_n}\eta_n$ so that $(p, (0,w_n), (0,v_n))$ is an oriented orthonormal basis. We can write $\sigma_n$ in this basis as: \begin{equation}\label{sigma_n}
    \sigma_n=a_np+b_n(0,w_n)+c_n(0,v_n).
\end{equation}
 Now, by an elementary computation, one may check that the differential of the radial projection $\Pi$ satisfies the following:
$$\mathrm{d}_{(1,x,y)}\Pi(v_0,v_1,v_2)=\left( -x v_0+v_1, -y v_0+v_2       \right),$$ for all $(x,y)\in\mathbb{D}^2$ and $(v_0,v_1,v_2)\in\mathbb{R}^3$. This implies that $$\mathrm{d}_{(1,\eta_n)}\Pi(p) = -\eta_n,\  \mathrm{d}_{(1,\eta_n)}\Pi(0,w_n) = w_n, \ \mathrm{d}_{(1,\eta_n )}\Pi(0,v_n) = v_n,$$ 
therefore:
\begin{equation}\label{pullback_of_vector_field2}
  \mathrm{d}_{(1,\eta_n)}\Pi( (1,\eta_n)\boxtimes\grad_{\Pi^{-1}(\eta_n)} h_n )=\mathrm{d}_{(1,\eta_n)}\Pi( (1,\eta_n)\boxtimes\sigma_n )=-c_n(1-r_n^2)w_n+(b_n-r_na_n)v_n
\end{equation}
Hence, by \eqref{pullback_of_vector_field2} it is enough to show that:
\begin{equation}\label{new_condition}
(1 - r_n^2) c_n \to 0 \quad \text{and} \quad b_n - r_n a_n \to 0.
\end{equation}
To prove this, observe that:
\begin{equation}\label{grad_to_0}
\sqrt{1 - r_n^2} \lVert \grad_{\Pi^{-1}(\eta_n)} h_n \rVert_{\mathbb{H}^2} \to 0.
\end{equation}
Indeed, from Lemma \ref{mean_up_convex_core}, observe that \(h_n\) defined in \eqref{eq_h_n_positive} is a positive function which is moreover a solution of \eqref{platau_HP_proof}. Hence, we may apply Corollary \ref{Li_Yau_Cor} to get the gradient estimate: $$\lVert \grad_{\Pi^{-1}(\eta_n)} h_n\rVert_{\mathbb{H}^2}\leq C_0 h_n(\Pi^{-1}(\eta_n)).$$ Thus \begin{align*}
\sqrt{1-r_n^2}\lVert\grad_{\Pi^{-1}(\eta_n)} h_n\rVert_{\mathbb{H}^2}&\leq C_0\sqrt{1-r_n^2}h_n(\Pi^{-1}(\eta_n))\\
&=C_0\overline{h_n}(\eta_n)\to 0.
\end{align*}
Note that \begin{align*}
c_n&=\inner{\sigma_n,(0,v_n)}_{1,2}\\
&=\inner{\grad_{\Pi^{-1}(\eta_n)}h_n-h_n(\Pi^{-1}(\eta_n))\Pi^{-1}(\eta_n),(0,v_n)}_{1,2}\\
&= \inner{\grad_{\Pi^{-1}(\eta_n)}h_n,(0,v_n)}_{1,2}.
\end{align*} It follows from Cauchy Schwartz inequality and from the limit \eqref{grad_to_0} that:
$$(1-r_n^2)\lvert c_n\rvert\leq (1-r_n)^2\lVert\grad_{\Pi^{-1}(\eta_n)}h_n\rVert_{\mathbb{H}^2} \to 0.$$
We turn now to the proof of the other limit: $b_n-r_na_n\to 0$. First, observe that by \eqref{sigma_n} we have
$$b_n-r_na_n=\inner{\sigma_n,(0,w_n)+r_np}_{1,2}.$$
Since $\Pi^{-1}(\eta_n)\boxtimes(0,v_n)=\frac{-1}{\sqrt{1-r_n^2}}((0,w_n)+r_np)$ then $$b_n-r_na_n=-\sqrt{1-r_n^{2}}\inner{\grad_{\Pi^{-1}(\eta_n)}h_n,\Pi^{-1}(\eta_n)\boxtimes (0,v_n)}_{1,2}.$$ Since $\Pi^{-1}(\eta_n)\boxtimes (0,v_n)$ is a tangent vector at $\Pi^{-1}(\eta_n)$ of norm $1$, then again by Cauchy Schwartz inequality and \eqref{grad_to_0} we have 
\begin{align*}
\lvert b_n-r_na_n\rvert&=\sqrt{1-r_n^{2}}\ \lvert\inner{\grad_{\Pi^{-1}(\eta_n)}h_n,\Pi^{-1}(\eta_n)\boxtimes (0,v_n)}_{1,2}\vert\\
&\leq  \sqrt{1-r_n^{2}}\ \lVert \grad_{\Pi^{-1}(\eta_n)}h_n\rVert_{\mathbb{H}^2}\to 0.  \\
\end{align*} This finishes the proof of \eqref{new_condition} and thus the proof of \eqref{claiminextension}.\\

The last step of the proof is to show the limit \eqref{goal}. Since $h_n(q)=u_X(q)-\inner{q,s_n}_{1,2}$ and $\grad_q(\inner{q,s_n})=s_n+\inner{q,s_n}_{1,2}q$, then $$\grad_{\Pi^{-1}(\eta_n)}h_n=\grad_{\Pi^{-1}(\eta_n)}u_X-s_n-\inner{\Pi^{-1}(\eta_n),s_n}\Pi^{-1}(\eta_n).$$ Thus 
$$(1,\eta_n)\boxtimes\grad_{\Pi^{-1}(\eta_n)}h_n=(1,\eta_n)\boxtimes\grad_{\Pi^{-1}(\eta_n)}u_X-(1,\eta_n)\boxtimes s_n$$
This implies that
\begin{equation}\label{eenfin}
\overline{\mathrm{HL}(X)}(\eta_n)=\mathrm{d}_{(1,\eta_n)}\Pi\left( (1,\eta_n)\boxtimes\grad_{\Pi^{-1}(\eta_n)} h_n \right)+\mathrm{d}_{(1,\eta_n)}\Pi\left( (1,\eta_n)\boxtimes s_n\right).\end{equation}
It follows form Proposition \ref{infinitesimal_earthquake_thm} that $  \mathcal{E}^-_X(\eta_n)=\mathrm{d}_{(1,\eta_n)}\Pi\left( (1,\eta_n)\boxtimes s_n\right)\to X(\eta_{\infty})$ and hence using \eqref{claiminextension} in \eqref{eenfin} we obtain 
$$\overline{\mathrm{HL}(X)}(\eta_n)\to X(\eta_{\infty}),$$ this completes the proof.
\end{proof}

\subsection{Infinitesimal Douady-Earle extension and mean surface in $\HP$.}
The goal of this section is to explain that the vector field associated to the mean surface in $\HP$ coincides with the infinitesimal Douady-Earle extension. Before stating the result, recall that the hyperboloid model $\mathbb{H}^2$ and the Poincaré disk model $\mathbb{B}^2$ can be identified through the isometry (see \cite[Chapter 2]{martelli2022introduction}):
\begin{equation}\label{H2_to_B2}\begin{array}{ccccc}
G & : & \mathbb{H}^2 & \to & \mathbb{B}^2 \\
 & & (x_0,x_1,x_2) & \mapsto & (\frac{x_1}{1+x_0}, \frac{x_2}{1+x_0}) \\
\end{array}\end{equation}
By pushing forward vector fields, the map $G$ induces a linear map between $\Gamma(\mathbb{H}^2)$ and $\Gamma(\mathbb{B}^2)$ which is conformally natural, we denote such map by $G_*$. 
\begin{prop}\label{HL_DE}
   Let $\mathrm{HL}:\Gamma(\mathbb{S}^1)\to\Gamma(\mathbb{H}^2)$ be the linear operator given by
$$\begin{array}{ccccc}
\mathrm{HL} & : & \Gamma(\mathbb{S}^1) & \to & \Gamma(\mathbb{H}^2) \\
 & & X & \mapsto & \J\grad(u_X) \\
\end{array},$$ where $u_X:\mathbb{H}^2\to\mathbb{R}$ is the unique solution of
   \begin{equation}\label{platau_HP}
    \begin{cases}
\Delta^{\mathbb{H}^2} u_X-2u_X=0  \\
\overline{u_X}|_{\mathbb{S}^1}=\phi_X.
\end{cases}\end{equation}
Then, we have:
$$G_*\mathrm{HL}=L_0,$$ 
where $L_0$ is the infinitesimal Douady-Earle extension defined in \eqref{inf_DE}. 
\end{prop}

\begin{proof}
First, we will show that $\mathrm{HL}$ is a continuous linear map that is conformally natural. We start with the conformal naturality. Let $A \in \mathrm{Isom}(\mathbb{H}^2)$ and $X \in \Gamma(\mathbb{S}^1)$. Consider $\Phi_X$ as the 1-homogeneous function associated to $X$. Then, using elementary computations from Lemma \ref{homogene}, we deduce that if $A$ preserves orientation, then $\Phi_X \circ A^{-1}$ is the 1-homogeneous function associated to $A_* X$. Thus, if $\phi_{A_* X}:\mathbb{S}^1 \to \mathbb{R}$ is the support function of $A_* X$, then according to Lemma \ref{equii}, we have $\mathrm{gr}(\phi_{A_* X}) = \mathrm{Is}(A, 0) \mathrm{gr}(\phi_X)$. Hence, it is straightforward to check that $$u_{A_*X}=u\circ A^{-1}.$$ Therefore,
\begin{align*}
\mathrm{HL}(A_*X)&=\J\grad(u_{A_*X})\\
&=\J\grad(u_X\circ A^{-1})\\
&=\J A_* \grad(u_X)\\
&=A_*\J\grad(u_X) \ \ \ \ \ \ \ \ \ \ \ \ (\text{because $\J A_*=A_*\J$})\\
&=A_*\mathrm{HL}(X).
\end{align*}
It remains to prove the invariance of $\mathrm{HL}$ with respect to isometries that reverse the orientation. To show this, one may note that the isometry group of $\mathbb{H}^2$ is generated by isometries which preserve the orientation, that is $\mathrm{O}_0(1,2)$ and the orientation reversing isometry $\gamma$ given by $$\gamma=\begin{bmatrix}
1 & 0 & 0 \\
0 & 1 & 0 \\
0 & 0 & -1
\end{bmatrix}.$$
As we already proved the invariance with respect to $\mathrm{O}_0(1,2)$, then it is enough to show the invariance of $\mathrm{HL}$ by $\gamma.$ We have for $z \in \mathbb{S}^1$,
$$\phi_{\gamma_*X}(z)=-\phi_X(\overline{z})$$ and so $$u_{\gamma_*X}=-u_X\circ\gamma^{-1}.$$ This implies that \begin{align*}
\mathrm{HL}(\gamma_*X)&=\J\grad(u_{\gamma_*X})\\
&=-\J\grad(u_X\circ \gamma^{-1})\\
&=-\J \gamma_* \grad(u_X)\\
&=\gamma_*\J\grad(u_X) \ \ \ \ \ \ \ \ \ \ \ \ (\text{because $\J \gamma_*=-\gamma_*\J$})\\
&=\gamma_*\mathrm{HL}(X).
\end{align*}
Next, we prove the continuity of the linear operator $\mathrm{HL}$. Let $X_n$ be a sequence of vector fields converging to $X$. Then $\phi_{X_n}:\mathbb{S}^1 \to \mathbb{R}$ is a sequence of continuous functions uniformly converging to $\phi_X:\mathbb{S}^1 \to \mathbb{R}$. It follows from Lemma 2.38 in \cite{barbotfillastre} that $u_{X_n}$ converges to $u_X$ uniformly on compact sets of $\mathbb{H}^2$. By Corollary \ref{Li_Yau_Cor}, there is a constant $C_0$ independent of $n$ such that
$$\lVert \grad_p(u_{X_n}-u_X)\rVert_{\mathbb{H}^2}\leq C_0\lvert u_{X_n}(p)-u_X(p)\rvert.$$
This shows that $\grad(u_{X_n})$ converges to $\grad(u_X)$ uniformly on compact sets of $\mathbb{H}^2$, which concludes the proof of the continuity of $\mathrm{HL}$. (Notice that one may apply classical Schauder estimates (see \cite{PDE_Book}) on compact sets on the elliptic equation $\Delta^{\mathbb{H}^2} u - 2u=0$ instead of the Li-Yau estimate to prove the uniform convergence of $\grad(u_{X_n})$).

Since \(G:\mathbb{H}^2 \to \mathbb{B}^2\) is an isometry, then \(G_*\mathrm{HL}(X)\) is also a continuous linear map from \(\Gamma(\mathbb{S}^1)\) to \(\Gamma(\mathbb{B}^2)\) which is conformally natural. Therefore, Theorem \ref{Uniqueness} implies that:
\begin{equation}\label{G_lambda_0}
G_*\mathrm{HL} = \lambda L_0,
\end{equation}
for some \(\lambda \in \mathbb{R}\). We claim that \(\lambda = 1\). To prove this, consider the vector field
\[ X(z) = iz. \]
Let \(\mathrm{R}_t:\mathbb{B}^2 \to \mathbb{B}^2\) be the one-parameter family of rotations given by \(\mathrm{R}_t(z) = e^{it}z\), then
\[ X = \frac{d}{dt}\bigg\lvert_{t=0}\mathrm{R}_t. \]
It follows from Remark \ref{douadyearly_killing} that
\begin{equation}\label{L_0_X}
L_0(X) = X.
\end{equation}
Now we will compute \(G_*\mathrm{HL}(X)\). First, observe that \(\phi_X = 1\), which can be written as
\[ \phi_X(z) = \langle (1,z), (-1,0,0) \rangle_{1,2}. \]
Therefore, it is immediate to check that the solution \(u_X\) of \eqref{platau_HP} is given by
\[ u_X(p) = \langle p, (-1,0,0) \rangle_{1,2}, \]
for all \(p \in \mathbb{H}^2\). This implies that \(\grad_p u_X = \sigma + u_X(p) p\) and so
\begin{equation}\label{HL(p)}
\mathrm{HL}(X)(p) = p \boxtimes (-1,0,0).
\end{equation}
For \((x_0,x_1,x_2) \in \mathbb{H}^2\), we have:
\begin{equation}\label{dG}
\mathrm{d}_{(x_0,x_1,x_2)}G = \frac{1}{1 + x_0} \begin{pmatrix}
-\frac{x_1}{1 + x_0} & 1 & 0 \\
-\frac{x_2}{1 + x_0} & 0 & 1
\end{pmatrix},
\end{equation}
The inverse of \(G\) is given by:
\begin{equation}\label{inverseG}
\begin{array}{ccccc}
G^{-1} & : & \mathbb{B}^2 & \to & \mathbb{H}^2 \\
& & z = (x,y) & \mapsto & \left( \frac{1 + \lvert z \rvert^2}{1 - \lvert z \rvert^2}, \frac{2x}{1 - \lvert z \rvert^2}, \frac{2y}{1 - \lvert z \rvert^2} \right) \\
\end{array}
\end{equation}
Combining \eqref{HL(p)}, \eqref{dG}, and \eqref{inverseG}, we obtain by tedious but elementary computation:
\begin{equation}\label{GHL_X}
G_*\mathrm{HL}(X)(z) = iz,
\end{equation}
for \(z \in \mathbb{B}^2\). In conclusion, equations \eqref{G_lambda_0}, \eqref{L_0_X}, and \eqref{GHL_X} imply that \(\lambda = 1\), that is, \(G_*\mathrm{HL} = L_0\) as desired.

\end{proof}

\section{Uniqueness of extension}\label{sec_uniquness}
The goal of this section is to prove the following uniqueness result $(i)$ in Theorem \ref{intro_uniqu_HL}. More precisely, we show the following.

\begin{theorem}\label{uniqTH1HL}
Let $X$ be a continuous vector field on $\mathbb{S}^1$. Let $V$ be an harmonic Lagrangian vector field on $\mathbb{H}^2$ that extends continuously to $X$ on $\mathbb{S}^1$. Assume that $\lVert \overline{\partial}V\rVert_{\infty}$ is finite. Then $V=\mathrm{HL}(X)$ (see Proposition \ref{HL_DE}). 
\end{theorem}
It is tempting to have uniqueness without the additional hypothesis on the boundedness of the $\overline{\partial}$-operator; however, our proof relies on this hypothesis.

The rest of this section is devoted to proving Theorem \ref{uniqTH1HL}. The next Lemma is a key step of the proof.

\begin{prop}\label{zygumund_imply_extension}
Let \( u: \mathbb{H}^2 \to \mathbb{R} \) be a smooth function such that \( \Delta^{\mathbb{H}^2} u - 2u = 0 \). Consider \( \B = \hess(u) - u \mathbbm{1} \), the shape operator of the graph of \( u \), and assume that \( \B \) is bounded. Then the function \( \overline{u}: \mathbb{D}^2 \to \mathbb{R} \) defined in \eqref{u_et_u_bar} extends continuously to a function \( \phi: \mathbb{S}^1 \to \mathbb{R} \).
\end{prop}

\begin{proof}
Let \(\lambda\) and \(-\lambda\) be the principal curvatures of \(\mathrm{gr}(u) \subset \HP\). Then, by Lemma \ref{d_bar_and_shape_operator}, 
$$\lVert \B \rVert_{\infty} = \lVert \lambda \rVert_{\infty},$$ 
where \(\lVert \lambda \rVert_{\infty} := \sup_{p \in \mathbb{H}^2} \lvert \lambda(p) \rvert.\)
For each \( t \geq \lVert \lambda \rVert_{\infty} \), consider the functions \( u_{{+t}} = u + t \) and \( u_{{-t}} = u - t \) and let \( \B_{\pm t} = \hess(u_{\pm t}) - u_{\pm t} \mathbbm{1} \). Consider \( \overline{u_{{-t}}} \) and \( \overline{u_{{+t}}} \), the functions defined on \( \mathbb{D}^2 \) (see \eqref{u_et_u_bar}). We claim that \( \overline{u_{{-t}}} \) and \( \overline{u_{{+t}}} \) are convex and concave functions, respectively. To show this, it is enough to show that the Euclidean Hessian \(  \mathrm{Hess}^{\mathbb{E}^2}(\overline{u_{{+t}}}) \) (resp. \(  \mathrm{Hess}^{\mathbb{E}^2}(\overline{u_{{-t}}}) \)) is positive definite (resp. negative definite). Let us focus on proving the convexity of \( \overline{u_{{-t}}} \).

The following formula relates the Euclidean Hessian and the hyperbolic Hessian (see \cite[Corollary 2.7]{barbotfillastre}): for a function \( f: \mathbb{D}^2 \to \mathbb{R} \), we have
$$L^{-1} \mathrm{Hess}^{\mathbb{E}^2} f(\eta)(X,Y) = \mathrm{Hess}^{\mathbb{D}^2}(L^{-1} f)(\eta)(X,Y) - (L^{-1} f)(\eta) g_{\mathbb{D}^2}(\eta)(X,Y),$$ 
where \( L(\eta) = \sqrt{1 - \lvert \eta \rvert^2} \) and \( g_{\mathbb{D}^2} \) is the hyperbolic metric in the Klein model \(\mathbb{D}^2\). Hence, we need to show that 
\begin{equation}\label{klein_def_positif}
    \mathrm{Hess}^{\mathbb{D}^2}(L^{-1} \overline{u_{{-t}}})(\eta)(X,Y) - (L^{-1} \overline{u_{{-t}}})(\eta) g_{\mathbb{D}^2}(\eta)(X,Y)
\end{equation} 
is positive definite. Recall that \( u_{-t} \circ \Pi^{-1} = L^{-1} \overline{u_{-t}} \), where \( \Pi: \mathbb{H}^2 \to \mathbb{D}^2 \) is the radial projection (see \eqref{radial}). Thus \eqref{klein_def_positif} is equivalent to proving that 
\begin{equation}\label{klein_def_positif2}
    \mathrm{Hess}^{\mathbb{D}^2}(u_{-t} \circ \Pi^{-1})(\eta)(X,Y) - (u_{-t} \circ \Pi^{-1})(\eta) g_{\mathbb{D}^2}(\eta)(X,Y)
\end{equation} 
is positive definite. However, since \(\Pi: \mathbb{H}^2 \to \mathbb{D}^2\) is an isometry between the hyperboloid model and the Klein model of the hyperbolic space, \eqref{klein_def_positif2} is equivalent to the fact that \( \B_{{-t}} = \hess(u_{-t}) - u_{-t} \mathbbm{1} \) is positive definite. Since \( \B_{{-t}} = \hess(u - t) - (u - t) \mathbbm{1} = \B + t \mathbbm{1} \), then
$$\det(\B_{{-t}}) = t^2 - \lambda^2 > 0 \quad \text{and} \quad \mathrm{tr}(\B_{{-t}}) = 2t > 0.$$ 
This implies that \( \B_{{-t}} \) is positive definite. This finishes the proof of the convexity of \( \overline{u_{{-t}}} \). Similarly, one can show that \( \overline{u_{{+t}}} \) is a concave function.
Note that
$$\overline{u_{-t}}(\eta) := \overline{u} - t\sqrt{1-\lvert\eta\rvert^2}, \ \ \overline{u_{+t}}(\eta) := \overline{u} +t\sqrt{1-\lvert\eta\rvert^2},$$
on $\mathbb{D}^2$. Since $\eta\mapsto\sqrt{1-\lvert\eta\rvert^2}$ vanishes on $\mathbb{S}^1$, the boundary values of $\overline{u_{-t}}$ and $\overline{u_{+t}}$ coincide. Moreover, $\overline{u_{-t}}$ is lower semicontinuous and $\overline{u_{+t}}$ is upper semicontinuous. Hence, the common boundary value of $\overline{u_{-t}}$ and $\overline{u_{+t}}$ is a continuous function on $\mathbb{S}^1$, as it is both lower and upper semicontinuous. Therefore, $\overline{u_{-t}}$ and $\overline{u_{+t}}$ are continuous functions on $\overline{\mathbb{D}^2}$ by Proposition \ref{boundaryvalueextension}. This implies that $\overline{u}$ extends continuously to $\overline{\mathbb{D}^2}$.\end{proof}

\begin{proof}[Proof of Theorem \ref{uniqTH1HL}]
Let $u_X:\mathbb{H}^2\to\mathbb{R}$ be the unique solution of
\begin{equation}\label{u_X_uniquness}
\Delta^{\mathbb{H}^2} u_X-2u_X=0  \ \ \ 
\overline{u_X}|_{\mathbb{S}^1}=\phi_X,
\end{equation}
so that $\mathrm{HL}(X)=\J\grad(u_X).$
Let $V$  be an harmonic Lagrangian vector field on $\mathbb{H}^2$ that extends continuously to $X$ with $\lVert\overline{\partial}V\rVert_{\infty}< \infty$. The goal is to show that $V=\mathrm{HL}(X)$. By Theorem \ref{charac_HL}, we can write $V=\J\grad(u)$ where $\Delta^{\mathbb{H}^2}u-2u=0$. We claim that $\overline{u}|_{\mathbb{S}^1}=\phi_X$. Indeed, if we have such a fact, then $u$ and $u_X$ would be two solutions of the PDE \eqref{u_X_uniquness}, and by the uniqueness of the solution (see Proposition \ref{BF_platau_HP}), we have $u=u_X$, and so $V=\mathrm{HL}(X)$, which is what we wanted to prove.

We will now show that $\overline{u}_{\mathbb{S}^1}=\phi_X$. Let $\B=\hess(u)-u\mathbbm{1}$ be the shape operator of the graph of $u$. By Lemma \ref{d_bar_and_shape_operator}, $\lVert\overline{\partial} V\rVert_{\infty}=\lVert \B\rVert_{\infty}<\infty$. Thus, Proposition \ref{zygumund_imply_extension} implies that $\overline{u}$ extends to a continuous function $\phi:\mathbb{S}^1\to\mathbb{R}.$ Consider the vector field $X'$ given by:
$$X'(z)=iz\phi(z).$$ By definition, $V=\mathrm{HL}(X')$ and so by Proposition \ref{second_proof}, $V$ extends to $X'$. On the other hand, $V$ extends to $X$ by hypothesis, thus $X'=X$, and so, $\phi=\phi_X$. Therefore, $\overline{u}_{\mathbb{S}^1}=\phi_X$, which concludes the proof.
\end{proof}

\section{Regularity of vector field: width and mean surface}\label{sec8_reg}
In this section, we characterize vector fields on the circle of different regularities using the Half-Pipe width. Let $X$ be a continuous vector field, and let $\phi_X: \mathbb{S}^1 \to \mathbb{R}$ be its support function. Consider the function $w_X: \mathbb{D}^2 \to \mathbb{R}$ defined by
\begin{equation}\label{w_xz}
w_X(\eta) = \frac{\phi_X^+(\eta) - \phi_X^-(\eta)}{\sqrt{1 - \lvert \eta \rvert^2}},
\end{equation}
Recall that for Definitions \ref{width} and \ref{widthX}, the width of $X$ is the supremum of $w_X$ over $\mathbb{D}^2$. Furthermore, $w_X$ satisfies the following invariance property: for all $A \in \mathrm{O}_0(1,2)$ and $\sigma \in \minko$, we have:
\begin{equation}\label{eq_width_equivariant}
    w_{A_*X+\Lambda(\sigma)}(\eta)=w_{X}(A^{-1}\cdot \eta).
\end{equation}
We have the following useful estimate:
\begin{prop}\label{prop_d_bar_less_width}
Let $X$ be a continuous vector field on the circle. Then for all $p \in \mathbb{H}^2$, we have
\begin{equation}\label{eqq_d_bar_less_width}
\lVert \overline{\partial}\mathrm{HL}(X)_p \rVert \leq 6 w_X(\Pi(p)),
\end{equation}
where $\Pi: \mathbb{H}^2 \to \mathbb{D}^2$ is the radial projection defined in \eqref{radial}.
\end{prop}

The next lemma is a key step to proving Proposition \ref{prop_d_bar_less_width}.

\begin{lemma}\cite[Lemma 6.4]{diaf2023Infearth}\label{simple_resultt}
Let $X$ be a continuous vector field. Assume that $\mathbb{D}^2 \times \{0\}$ is a support plane of $\mathrm{gr}(\phi_X^-)$ at $(0, 0, \phi_X^{-}(0, 0))$ (and so $\phi^-_X(0, 0) = 0).$ Then for all $z \in \mathbb{S}^1$, we have:
    \begin{equation}\label{eq_phi_less_width}
        0\leq\phi_X(z)\leq 2 \  (\phi_X^+(0,0)-\phi_X^-(0,0)).
    \end{equation}
\end{lemma}
Lemma 6.4 in \cite{diaf2023Infearth} is stated differently. In fact, the width of \( X \) is used instead of \( \phi_X^+(0,0) - \phi_X^-(0,0) \). However, the proof is exactly the same line by line, so we omit it here.
\begin{proof}[Proof of Proposition \ref{prop_d_bar_less_width}]
Using the equivariance of the two terms in the estimate \eqref{prop_d_bar_less_width} under the isometry group of the hyperbolic plane, it suffices to prove the statement for \( p = (1,0,0) \) and thus \(\Pi(p) = (0,0) \in \mathbb{D}^2\). By adding a Killing vector field, we may also assume that \(\mathbb{D}^2 \times \{0\}\) is a support plane of \(\mathrm{gr}(\phi_X^-)\) at \((0,0)\), so that we are in the configuration of Lemma \ref{simple_resultt}.

Based on the discussion above, we need to show that:
\begin{equation}\label{eq_d_bar_less_width}
    \lVert \overline{\partial} \mathrm{HL}(X)_{(1,0,0)} \rVert \leq 6 w_X(0).
\end{equation}
Let \( L_0(X): \mathbb{B}^2 \to \mathbb{C} \) be the infinitesimal Douady-Earle extension given in \eqref{inf_DE} and consider the isometry \( G:\mathbb{H}^2 \to \mathbb{B}^2 \) defined in \eqref{H2_to_B2}. Note that \( G(1,0,0) = (0,0) \), thus according to Proposition \ref{HL_DE} and Corollary \ref{d_bar_isometry}, the estimate \eqref{eq_d_bar_less_width} is equivalent to showing that
\begin{equation}\label{eq_d_bar_less_width1}
    \left| \frac{\partial L_0(X)}{\partial \overline{z}}(0) \right| \leq 6 w_X(0).
\end{equation}
Since \( w_X(0) = \phi_X^+(0) - \phi_X^-(0) \), then \eqref{eq_d_bar_less_width1} is equivalent to 
\[
\left| \frac{\partial L_0(X)}{\partial \overline{z}}(0) \right| \leq 6(\phi_X^+(0) - \phi_X^-(0)).
\]
Taking the \(\overline{z}\)-derivative in the integral of \eqref{inf_DE} (see \cite[Page 380]{Extension_with_bounded_derivative}), we obtain 
\[
\frac{\partial L_0(X)}{\partial \overline{z}}(z) = \frac{3(1 - \vert z \vert^2)^2}{2i\pi} \int_{\mathbb{S}^1} \frac{X(x)}{(1 - \overline{z}x)^4} dx.
\]
Thus, we compute:
\begin{align*}
\left| \frac{\partial L_0(X)}{\partial \overline{z}}(0) \right| &= \left| \frac{3}{2i\pi} \int_{\mathbb{S}^1} X(x) dx \right| \\
&= \frac{3}{2\pi} \left| \int_{0}^{2\pi} i e^{i\theta} X(e^{i\theta}) d\theta \right| \\
&\leq \frac{3}{2\pi} \int_{0}^{2\pi} \left| i e^{i\theta} \phi_X(e^{i\theta}) \right| d\theta \\
&\leq \frac{3}{2\pi} \int_{0}^{2\pi} \left| \phi_X(e^{i\theta}) \right| d\theta \\
&\leq 6(\phi_X^+(0) - \phi_X^-(0)),
\end{align*}
where we use Lemma \ref{simple_resultt} in the last inequality. This concludes the proof.

\end{proof}

\subsection{Zygmund vector fields.}
The goal of this section is to provide a quantitative estimate between the width of a vector field on the circle and the $\overline{\partial}$ norm of its infinitesimal Douady-Earle extension. We start by recalling some terminology. Given a quadruple $Q=[a,b,c,d]$ of points on $\mathbb{S}^1$ arranged in counter-clockwise order, the \textit{cross-ratio} of $Q$ is given by $$\mathrm{cr}(Q)=\frac{  (b-a)(d-c) }{ (c-b)(d-a)   }.$$ Let $X$ be a vector field on $\mathbb{S}^1$. The \textit{cross-ratio distortion} norm of $X$ is defined as 
$$\lVert  X \rVert_{cr}=\sup_{\mathrm{cr}(Q)=1}\vert X[Q]\vert \in[0,+\infty],$$ where 
$$X[Q]=\frac{X(b)-X(a)}{b-a}-\frac{X(c)-X(b)}{c-b}+\frac{X(d)-X(c)}{d-c}-\frac{X(a)-X(d)}{a-d},$$
for $Q=[a,b,c,d]$. For instance, one can check that $X$ is an extension to $\mathbb{S}^1$ of a Killing vector field of $\mathbb{H}^2$ if and only if $\lVert X\rVert_{cr}=0.$

\begin{defi}
 A continuous vector field $X$ of $\mathbb{S}^1$ is \textit{Zygmund} if and only if the cross-ratio distortion norm of $X$ is finite.
\end{defi}
Recall that an orientation-preserving homeomorphism $\Phi:\mathbb{S}^1\to\mathbb{S}^1$ is said to be \textit{quasisymmetric} if the \textit{cross-ratio norm} defined as
 $$\lVert  \Phi \rVert_{cr}=\sup_{\mathrm{cr}(Q)=1} \vert\ln{\mathrm{cr}\left(\Phi(Q)\right)}    \vert ,$$ is finite. Denote by $\mathcal{QS}(\mathbb{S}^1)$ the space of quasisymmetric homeomorphisms of $\mathbb{S}^1$. Then the \textit{universal Teichmüller space} is the space of quasisymmetric homeomorphisms of the circle up to post-composition with an isometry of $\mathbb{H}^2$:
$$\mathcal{T}(\mathbb{H}^2)=\mathrm{Isom}(\mathbb{H}^2)\backslash \mathcal{QS}(\mathbb{S}^1).$$
Thus, $\mathcal{T}(\mathbb{H}^2)$ can be identified with the space of quasisymmetric homeomorphisms of $\mathbb{S}^1$ fixing $1$, $i$, and $-1$. It is known that $\mathcal{T}(\mathbb{H}^2)$ is an infinite-dimensional complex Banach manifold for which the tangent space at the identity corresponds to Zygmund vector fields on $\mathbb{S}^1$ that vanish at $1$, $i$, and $-1$ (see \cite{Gardiner1999QuasiconformalTT}).

In \cite[Theorem 1.3]{diaf2023Infearth}, we show that the width of $X$ is equivalent to the cross-ratio norm of $X$. In particular, we have:

\begin{cor}\cite[Theorem 1.3 and equation (96)]{diaf2023Infearth}\label{Cor_diaf_inf}
    $X$ is a Zygmund vector field of $\mathbb{S}^1$ if and only if the width $\omega(X)$ is finite.
\end{cor}

In \cite{Extension_with_bounded_derivative}, Reich and Chen proved that \( \lVert \partial L_0(X) \rVert_{\infty} \) is finite if and only if \( X \) is a Zygmund vector field. Later, Fan and Hu showed in \cite{FanJun} that the cross-ratio norm \( \lVert X \rVert_{cr} \) and \( \lVert \partial L_0(X) \rVert_{\infty} \) (which is equal to \( \lVert \overline{\partial} \mathrm{HL}(X) \rVert_{\infty} \)) are equivalent. The estimate \eqref{Intro_left_estimate} in Theorem \ref{intro_uniqu_HL} provides a similar result using the width of \( X \).
\subsection{Proof of Theorem \ref{intro_uniqu_HL}}
This subsection is devoted to completing the proof of Theorem \ref{intro_uniqu_HL} stated in the introduction.
\begin{proof}[Proof of Theorem \ref{intro_uniqu_HL}]
First, recall that the uniqueness result $(i)$ in Theorem \ref{intro_uniqu_HL} is already established in Theorem \ref{uniqTH1HL}.
The equivalence between \eqref{1_intro_uniqu} and \eqref{3_intro_uniqu} follows from Corollary \ref{Cor_diaf_inf}. Now we will show the equivalence between \eqref{3_intro_uniqu} and \eqref{4_intro_uniqu} by proving the estimate \eqref{Intro_left_estimate}.
If the width is finite, then the estimate \(\frac{1}{6} \lVert \overline{\partial} \mathrm{HL}(X) \rVert_{\infty} \leq \omega(X)\) follows immediately by taking the supremum in the pointwise estimate \eqref{eqq_d_bar_less_width} of Proposition \ref{prop_d_bar_less_width}. We will now prove the other estimate. Let \(u_X:\mathbb{H}^2 \to \mathbb{R}\) be the unique solution of
\begin{equation}\label{platau_HP}
\Delta^{\mathbb{H}^2} u_X - 2u_X = 0, \quad
\overline{u_X}|_{\mathbb{S}^1} = \phi_X.
\end{equation}
Let \(\B = \mathrm{Hess}(u_X) - u_X \mathbbm{1}\) be the shape operator of the graph of \(u_X\). Since \(\mathrm{HL}(X) = \J\grad(u_X)\), by Proposition \ref{d_bar_and_shape_operator}, we have \(\lVert \overline{\partial} \mathrm{HL}(X) \rVert_{\infty} = \lVert \B \rVert_{\infty}\). Assume that this is finite and let \(\lambda\) and \(-\lambda\) be the eigenvalues of $\B$. Then
$$\lVert \B \rVert_{\infty} = \lVert \lambda \rVert_{\infty},$$ 
where \(\lVert \lambda \rVert_{\infty} := \sup_{p \in \mathbb{H}^2} \lambda(p).\)
For each \( t \geq \lVert \lambda \rVert_{\infty} \), consider the functions \( u_{+t} = u_X + t \) and \( u_{-t} = u_X - t \) and let \(\B_{\pm t} = \mathrm{Hess}(u_{\pm t}) - u_{\pm t} \mathbbm{1} \). Consider \( \overline{u_{-t}} \) and \( \overline{u_{+t}} \), the functions defined on \(\mathbb{D}^2\) (see \eqref{u_et_u_bar}). We apply the same method as in the proof of Proposition \ref{zygumund_imply_extension} to deduce that \(\overline{u_{-t}}\) and \(\overline{u_{+t}}\) are convex and concave functions, respectively. It follows from property ($\mathrm{P2}$) in \ref{P2} that for all \( \eta \in \mathbb{D}^2 \),
$$\phi_X^+ \leq \overline{u_X} + t \sqrt{1 - \lvert \eta \rvert^2}, \quad \overline{u_X} - t \sqrt{1 - \lvert \eta \rvert^2} \leq \phi_X^-.$$
Hence, 
$$\overline{u_X} - t \sqrt{1 - \lvert \eta \rvert^2} \leq \phi_X^- \leq \phi_X^+ \leq \overline{u_X} + t \sqrt{1 - \lvert \eta \rvert^2}.$$
This implies
$$\frac{\phi_X^+(\eta) - \phi_X^-(\eta)}{\sqrt{1 - \lvert \eta \rvert^2}} \leq 2t.$$
Hence,
$$\omega(X) = \sup_{\eta \in \mathbb{D}^2} \frac{\phi_X^+(\eta) - \phi_X^-(\eta)}{\sqrt{1 - \lvert \eta \rvert^2}} \leq 2t,$$
and this holds for all \( t \geq \lVert \lambda \rVert_{\infty}, \) thus, \( \omega(X) \leq 2 \lVert \B\rVert_{\infty}. \) This finishes the proof of the estimate \eqref{Intro_left_estimate} in Theorem \ref{intro_uniqu_HL} and hence the equivalence between \eqref{3_intro_uniqu} and \eqref{4_intro_uniqu}.

We need to show the equivalence between \eqref{1_intro_uniqu} and \eqref{2_intro_uniqu}. Let \(V\) be a harmonic Lagrangian vector field with finite \(\lVert \overline{\partial} V \rVert_{\infty}\). Then by Theorem \ref{uniqTH1HL}, \(V = \mathrm{HL}(X)\) and so, by \eqref{Intro_left_estimate}, the width is finite. Hence, \(X\) is Zygmund by Corollary \ref{Cor_diaf_inf}. This shows the implication \eqref{2_intro_uniqu} \(\implies\) \eqref{1_intro_uniqu}. Conversely, if \(X\) is a Zygmund vector field, then the width is finite by Corollary \ref{Cor_diaf_inf} and hence \(\lVert \overline{\partial} \mathrm{HL}(X) \rVert_{\infty}\) is finite by \eqref{Intro_left_estimate}. This finishes the proof of \eqref{1_intro_uniqu} \(\implies\) \eqref{2_intro_uniqu}.
\end{proof}

\subsection{Little Zygmund vector fields}
The aim of this section is to characterize \textit{little Zygmund} vector fields on the circle in terms of their width. We start this section by discussing the notion of \textit{little Zygmund} vector fields. Given a quadruple $Q=[a,b,c,d]$ of points on $\mathbb{S}^1$, the \textit{minimal scale} $S(Q)$ is defined as:
 $$S(Q)=\min\{\lvert a-b\rvert, \ \lvert b-d\rvert,\ \lvert c-d\rvert, \ \lvert d-a\rvert     \}.$$
A sequence $ \{Q_n = [a_n, b_n, c_n, d_n]\}_{n\in\mathbb{N}}$ of quadruples of $\mathbb{S}^1$ is said to be \textit{degenerating} if $\mathrm{cr}(Q_n)=1$ for each $n$ and $\lim\limits_{n \rightarrow +\infty} S(Q_n)=0$.
\begin{defi}
    A Zygmund vector field $X$ is said to be \textit{little Zygmund} if 
$$\sup\limsup\limits_{n \rightarrow +\infty} \lvert V[Q_n]\rvert=0$$
where the supremum is taken over all degenerating sequences $Q_n$ of quadruples.
\end{defi}
Little Zygmund vector fields are related to the so-called \textit{little Teichmüller space}. Following \cite{Fan_Hu_Little}, an element $\Phi$ in $\mathcal{QS}(\mathbb{S}^1)$ is said to \textit{symmetric} if 
$$\sup\limsup\limits_{n \rightarrow +\infty} \lvert \mathrm{cr}(\Phi(Q_n))\rvert=0,$$
where the supremum is taken over all degenerating quadruples $Q_n$. Denote by $\mathcal{S}(\mathbb{S}^1)$ the space of symmetric homeomorphisms. In \cite{Gardiner_Sullivan}, Gardiner and Sullivan proved that $\mathcal{S}(\mathbb{S}^1)$ is a normal topological subgroup of $\mathcal{QS}(\mathbb{S}^1)$. Furthermore, the \textit{little Teichmüller space}, which is defined as the space of symmetric homeomorphisms of the circle up to post-composition with an isometry of  $\mathbb{H}^2$:
$$\mathcal{T}_0(\mathbb{H}^2):=\mathrm{Isom}(\mathbb{H}^2)\backslash \mathcal{S}(\mathbb{S}^1),$$ is an infinite-dimensional complex manifold modeled on a Banach space. Moreover, the tangent space at the identity corresponds to little Zygmund vector fields on $\mathbb{S}^1$ that vanishes at $1$, $i$ and $-1$. For more details, we refer the reader to the survey \cite{HU_Survey} and the references therein.

Fan and Hu characterize the little Zygmund regularity in terms of the $\overline{\partial}$-operator.
\begin{theorem}\cite[Theorem 3]{FanJun}\label{fan_hu_little}
    Let $X$ be a continuous vector field and $L_0(X)$ be the infinitesimal Douday-Earle extension defined in \eqref{inf_DE}. Then $X$ is little Zygmund if and only if $$\left|\frac{\partial L_0(X)}{\partial \overline{z}}(z)\right| \to 0 \ \text{as}\  \lvert z\rvert\to 1.$$
\end{theorem}
\begin{remark}\label{remark_little_zygmund}
Recall that by Corollary \ref{d_bar_isometry} and Proposition \ref{HL_DE}, we have for all \(z \in \mathbb{B}^2\):
\begin{equation}
   \lVert \overline{\partial} \mathrm{HL}(X)_{G^{-1}(z)} \rVert =  \left| \frac{\partial L_0(X)}{\partial \overline{z}}(z) \right|,
\end{equation}
where \(G: \mathbb{H}^2 \to \mathbb{B}^2\) is the isometry defined in \eqref{H2_to_B2}. As a consequence, Theorem \ref{fan_hu_little} can be stated by saying that \(X\) is little Zygmund if and only if 
\(\lVert \overline{\partial} \mathrm{HL}(X)_{G^{-1}(z)} \rVert\) tends to $0$ as \(\lvert z \rvert \to 1\), which is equivalent by Lemma \ref{d_bar_and_shape_operator} to the fact that the principal curvature of the mean surface \(\mathrm{gr}(u_X)\) in \(\HP\) tends to zero at infinity.
\end{remark}

The next Theorem provides a characterization of little Zygmund vector fields in terms of the width.
\begin{theorem}[Theorem \ref{Intro_chara_little}]\label{LZVF}
    \label{width_goes_to_0}
Let $X$ be a continuous vector field and consider the function $\omega_{X}:\mathbb{D}^2\to \mathbb{R}$ defined in \eqref{w_xz}. Then the following are equivalent:
\begin{enumerate}
\item $X$ is little Zygmund.
\item $\lVert \overline{\partial}\mathrm{HL}(X)_{G^{-1}(z)}\rVert$ tends to zero as $\lvert z\rvert\to 1$.
\item $\omega_X(z)$ tends to zero as $\lvert z\rvert\to 1$.
\end{enumerate}
\end{theorem}
The equivalence between $(1)$ and $(2)$ in Theorem \ref{LZVF} is due to Fan and Hu \cite{FanJun} (and Remark \ref{remark_little_zygmund}). The new result on the little Zygmund vector fields concerns the characterisation with the width. To prove this, we need the following Lemma proved in \cite{FanJun}.
\begin{lemma}\cite[Page 1167]{FanJun}\label{teck_little}
    Let $X$ be a little Zygmund vector field and let $z_n$ be a sequence of points in $\mathbb{D}^2$ such that $\lvert z_n\rvert\to1$. Then there exists an isometry $A_n$ of $\mathbb{D}^2$ and a Killing vector field $K_n$ such that the vector field given by $X_n=A_{n*}X+K_n$ satisfies
$$X_n(1)=X_n(-1)=X_n(-i)=0.$$
and $X_n$ converges to $0$ uniformly on $\mathbb{S}^1$. 
\end{lemma}
The non-trivial part of the Lemma is to show the uniform convergence to $0$. Indeed, Zygmund vector fields are $\alpha$-Hölder for all $0<\alpha<1$ with a Hölder exponent that depends only on the cross-ratio norm. Then one may apply the Ascoli-Arzelà Theorem to show the uniform convergence of the vector field $X_n$ to some vector field $X$. Therefore, one needs to show that $X=0$, and here the hypothesis of little Zygmund is crucial.

\begin{proof}[Proof of Theorem \ref{width_goes_to_0}]
The equivalence $(1)\iff(2)$ follows from Fan-Hu Theorem \ref{fan_hu_little} and Remark \ref{remark_little_zygmund}. The implication $(3)\implies (2)$ follows from Proposition \ref{prop_d_bar_less_width}. It remains to show the implication $(1)\implies (3)$. As in Lemma \ref{teck_little}, for each $z_n$, let $A_n$ be an isometry of $\mathbb{D}^2$ such that $A_nz_n=0$, and consider $K_n$ to be a Killing vector field such that the vector field given by $V_n=A_{n*}X+K_n$ converges uniformly to $0$.
Let \(\phi_{V_n}\) be the support function of \(V_n\). Since \(\phi_{V_n}(z) = \mathrm{det}(V_n(z), z)\), \(\phi_{V_n}\) converges to \(0\) uniformly on \(\mathbb{S}^1\). On the other hand, using the invariance property of the function \(w_X\) (see \eqref{eq_width_equivariant}), we get
\[ w_X(z_n) = w_{V_n}(0) \to 0, \]
which concludes the proof.
\end{proof}

\begin{remark}\label{remark_weil_petersson}
The infinitesimal Douady-Earle extension is used in \cite{integral_operator_Zygmund} to characterize \textit{Weil-Petersson} vector fields on the circle. These represent the tangent space of the so-called \textit{universal Weil-Petersson Teichmüller space}. The authors of \cite{integral_operator_Zygmund} proved that a vector field $X$ is Weil-Petersson if and only if:
  \[ \int_{\mathbb{H}^2} \left|\frac{\partial L_0(X)}{\partial \overline{z}}(z)\right| ^2 \, \mathrm{d}\mathrm{Vol}_{\mathbb{B}^2} <\infty,\] where $\mathrm{d}\mathrm{Vol}_{\mathbb{B}^2} $ is the volume form of the Poincaré disk $\mathbb{B}^2$.  It is tempting to prove the same characterization with the width. Namely, one may wonder if $X$ is Weil-Petersson if and only if the function $w_X$, defined in \eqref{w_xz}, is square integrable with respect to the volume form of $\mathbb{B}^2$. Proposition \ref{prop_d_bar_less_width} provides one implication. Indeed, if $w_X$ is square integrable, then  $$\left|\frac{\partial L_0(X)}{\partial \overline{z}}(z)\right| =\rVert\overline{\partial}\mathrm{HL}(X)_{G^{-1}(z)}\lVert,$$ is square integrable and hence $X$ is Weil-Petersson by the work of \cite{integral_operator_Zygmund}. It is natural to conjecture that if $X$ is Weil-Petersson, then $w_X$ is square integrable. We leave this question for future investigation. Note that this question is motivated by the work of Bishop \cite{bishop}, which characterized the Weil-Petersson Teichmüller space as those circle homeomorphisms that are conformal weldings of Jordan curves in $\mathbb{S}^2$ with square integrable width in $\mathbb{H}^3$.
\end{remark}

\bibliographystyle{alpha}
\bibliography{harmonic}

\end{document}